\documentclass{article}%
\usepackage{amsmath}
\usepackage{amsfonts}
\usepackage{amssymb}
\usepackage{graphicx}%
\providecommand{\U}[1]{\protect\rule{.1in}{.1in}}
\newtheorem{theorem}{Theorem}

\newtheorem{definition}[theorem]{Definition}

\newtheorem{lemma}[theorem]{Lemma}

\newtheorem{proposition}[theorem]{Proposition}
\newtheorem{remark}[theorem]{Remark}

\newenvironment{proof}[1][Proof]{\noindent\textbf{#1.} }{\ \rule{0.5em}{0.5em}}
\begin{document}

\title{Two-end solutions to the Allen-Cahn equation in $\mathbb{R}^{3}$}
\author{Changfeng Gui\\Department of Mathematics, U-9,\\University of Connecticut, Storrs, CT 06269, USA, \\e-mail: gui@math.uconn.edu
\and Yong Liu\\School of Mathematics and Physics, \\North China Electric Power University, Beijing, China, \\e-mail: liuyong@ncepu.edu.cn
\and Juncheng Wei\\Department of Mathematics, \\University of British Columbia, Vancouver, BC, V6T 1Z2, Canada,\\e-mail: jcwei@math.ubc.ca}
\maketitle

\section{Introduction}

The Allen-Cahn equation
\begin{equation}
-\Delta u=u-u^{3},\left\vert u\right\vert <1 \label{AC}%
\end{equation}
has been studied for several decades and is an important nonlinear PDE due to
the fact that it lies at the interface of several different mathematical
fields. The famous De Giorgi conjecture states that any entire solution to
$\left(  \ref{AC}\right)  $ in $\mathbb{R}^{n}$ which is monotone in one
direction should be one dimensional, at least for $n\leq8.$ The conjecture was
proved in dimension $n=2$ by Ghoussoub-Gui (\cite{GG}) and dimension $3$ by
Ambrosio-Cabre (\cite{ACa}), and in dimensions $4\leq n\leq8$ by Savin
(\cite{Savin}), under an additional assumption. For $n\geq9,$ counter-examples
have been constructed by del Pino-Kowalczyk-Wei (\cite{MM}).  Note that
monotone solutions are indeed minimizers with respect to local perturbations (\cite{ACa}).

A natural extension of De Giorgi's conjecture is to classify stable or finite
Morse index solutions. Regarding stable solutions, it has been shown
(\cite{ACa}, \cite{GG}) that stable solutions in $\mathbb{R}^{2}$ are
necessarily one-dimensional, while Pacard-Wei (\cite{PW}) constructed stable
solutions in $\mathbb{R}^{8}$ which are not one-dimensional.

The study of finite Morse index solutions is much more involved. In
$\mathbb{R}^{2}$, we have now a rather complete picture of Morse index one
solutions. They are so-called four-end solutions, parametrized by the angle
between the two lines. The cross solution, constructed by Dang-Fife-Peletier
(\cite{DFP}), represents a four-end solution with angle $\frac{\pi}{4}$, while
the almost parallel line solution, constructed by del
Pino-Kowalczyk-Pacard-Wei (\cite{MR2557944}), represents a four-end solution
with angle close to $\frac{\pi}{2}$ or $0$. The existence of four-end
solutions with any angle between $0$ and $\frac{\pi}{2}$ was proved by two
methods: the first through the moduli space theory by Kowalczyk-Liu-Pacard
(\cite{MR3148064}), and the second approach by the mountain-pass variational
method by us (Gui-Liu-Wei \cite{GLW}). It was also shown that the four-end
solutions have Morse index one (\cite{GLW}). On the other hand,  monotone and symmetric properties of a  general four end solution have been obtained in \cite{MR2911416}, where more general finite morse index
solutions in $\mathbb{R}^{2}$ have also been studied under an extra energy  condition or  a condition on the asymptotical structure of nodal curves.

In this paper we are interested in the structure of two-end solutions of
$\left(  \ref{AC}\right)  $ in $\mathbb{R}^{3}.$ It turns out that without the
monotone condition, there are actually a lot of solutions. One simple example
is the so called saddle solution whose nodal sets are precisely the $xoy,yoz$
and $xoz$ planes (Alessio-Montecchiari \cite{MR3057178}). del
Pino-Kowalczyk-Wei (\cite{MR3019512}) proved that for each non-degenerate
minimal surfaces with finite total curvature, one could find a solution to
$\left(  \ref{AC}\right)  $ whose nodal sets are close to a rescaled version
of this minimal surface. In particular, there are axially solutions whose
nodal sets are close to catenoids with very large waist. Axially symmetric
solutions with multiple interfaces which are governed by the Jocobi-Toda
system are constructed in Agudelo-del Pino-Wei \cite{MR3281950}.

In spite of all these developments, some important questions for $\left(
\ref{AC}\right)  $ in $\mathbb{R}^{3}$ remain unanswered, even for axially
symmetric solutions. In this paper, we will study those axially symmetric
solutions which are additionally even with respect to the $xoy$ plane. In
terms of the cylindrical coordinate $\left(  r,z\right)  $ , they satisfy
\begin{equation}
\left\{
\begin{array}
[c]{l}%
u_{zz}+u_{rr}+r^{-1}u_{r}+u-u^{3}=0,r\in\lbrack0,+\infty),z\in\mathbb{R},\\
u\left(  r,z\right)  =u\left(  r,-z\right)  ,\text{ }u_{r}\left(  0,z\right)
=0.
\end{array}
\right.  \label{axial}%
\end{equation}

Let $H\left(  x\right)  =\tanh\frac{x}{\sqrt{2}}$ be the one dimensional
\textit{heteroclinic} solution:
\[
-H^{\prime\prime}=H-H^{3},H\left(  0\right)  =0,H\left(  \pm\infty\right)
=\pm1.
\]
The solutions we are interested will have $H$ as its asymptotic profile. We
say that a solution $u$ of $\left(  \ref{axial}\right)  $ has growth rate $k$
if it has the following asymptotic behavior:
\begin{equation}
\left\Vert u\left(  r,\cdot\right)  -H\left(  \cdot-k\ln r+c\right)
\right\Vert _{L^{\infty}\left(  0,+\infty\right)  }\rightarrow0,\text{ as
}r\rightarrow+\infty, \label{asy}%
\end{equation}
for certain $c\in\mathbb{R}.$ The existence results obtained in
\cite{MR3281950} and \cite{MR3019512} based on Lyapunov-Schmidt reduction
arguments tell us that there are solutions whose growth rate is in the
interval $\left(  \sqrt{2},\sqrt{2}+\delta\right)  $ and $\left(  \delta
^{-1},+\infty\right)  ,$ where $\delta$ is a very small constant. A natural
question is, whether or not there are solutions with growth rate in the range
$\left[  \sqrt{2}+\delta,\delta^{-1}\right]  .$ In this paper we answer this
question affirmatively and our main result is the following

\begin{theorem}
\label{main}For each $k\in\left(  \sqrt{2},+\infty\right)  ,$ there exists a
solution to $\left(  \ref{axial}\right)  $ which has growth rate $k.$
\end{theorem}

As we will see later, these solutions indeed are monotone in the following
sense:
\begin{equation}
u_{z}>0\text{ for }z>0;u_{r}<0,\text{for }r>0. \label{monotone}%
\end{equation}

Outside a large ball, the nodal set of the solutions given by Theorem
\ref{main} has two components, each component is asymptotic to a catenoidal
end and around each end, the solution look likes the one dimensional
heteroclinic solution. Borrowing a terminology from minimal surface theory, a
solution satisfying $\left(  \ref{axial}\right)  $ and $\left(  \ref{asy}%
\right)  $ will be called a two-end solution. We emphasize that here by
definition the two-end solutions are axially symmetric. Comparing with the
corresponding definition of two-end minimal surfaces, it seems that a more
general definition of two-end solution should\ also involve those solutions
which are not axially symmetric and only assume that their nodal set is
asymptotic to two catenoidal ends. In the minimal surface theory, a classical
result proved by R. Schoen (\cite{MR730928}) is that a minimal surface with
two catenoidal ends is a catenoid. We expect that the analogous result for
Allen-Cahn equation should hold. A major difficulty we encounter is to show
that a solution whose nodal set is asymptotic to two catenoidal ends is
axially symmetric.

Observe that for each $k\in\left(  0,+\infty\right)  ,$ there exists a
catenoid with growth rate $k.$ Taking into account the relation between
minimal surface theory and Allen-Cahn equation, at first glance, one may think
that for each $k\in\left(  0,+\infty\right)  ,$ there should be a two-end
solution of Allen-Cahn equation. But this turns out to be false. In fact, we have

\begin{theorem}
\label{main2}There does not exist two-end solution with growth rate
$k\in(0,\frac{\sqrt{2}}{2}].$
\end{theorem}

Indeed, one expects that for $k\in(\frac{\sqrt{2}}{2},\sqrt{2}],$ two-end
solution with growth rate $k$ also should not exist, while for each
$k\in\left(  \sqrt{2},+\infty\right)  ,$ there should be a unique two-end
solution with growth rate $k.$ We remark that the lower bound $\sqrt{2}$ is
somehow related to the deep facts that the two ends of the solutions to the
Allen-Cahn equation actually \textquotedblleft interact\textquotedblright%
\ with each other and in this regime one naturally encounters the so called
Toda system, as we will see later in the analysis. This constitutes a major
difference with the theory of minimal surfaces. Roughly speaking, the
Allen-Cahn equation interplays between the theory of minimal surfaces and the
theory of Toda system, which is a classical integrable system. In the minimal
surface theory, the catenoids are basic blocks for the construction of other
minimal surfaces or constant mean curvature surfaces (see for example
\cite{MR2906930}, \cite{MR1837428}, \cite{MR1924593} and the references
therein). It is therefore natural and  interesting to  ask what  role
 two-end solutions of the Allen-Cahn equation play  in constructing other
solutions. We also remark that the two-end solutions given by Theorem
\ref{main} are all unstable, and the stable solution conjecture says that the
only bounded stable solution of the Allen-Cahn equation in $\mathbb{R}^{3}$
should be one dimensional. It is also expected that the Morse index of two-end
solutions should be equal to one, again, we don't have a proof of this
statement, although it is known that the two-end solutions
constructed in \cite{MR3281950} and \cite{MR3019512} have Morse index one.

The results in Theorem \ref{main} could be regarded as a generalization of the
corresponding results for four-end solutions in $\mathbb{R}^{2}$
(\cite{MR2971030}, \cite{MR3148064})$.$ To explain this, let us say a few
words about the multiple-end solutions in $\mathbb{R}^{2}.$ By definition, a
$2k$-end solution of $\left(  \ref{AC}\right)  $ in $\mathbb{R}^{2}$ is a
solution whose nodal set outside a large ball is asymptotic to $2k$ half
straight lines at infinity. It is known (del Pino-Kowalczyk-Pacard
\cite{dkp-2009}) that the set of $2k$-end solution in $\mathbb{R}^{2}$ has a
structure of real analytic variety of formal dimension $2k.$ Some examples of
solutions near the boundary of this \textquotedblleft moduli
space\textquotedblright\ have been constructed (del Pino-Kowalczyk-Pacard-Wei
\cite{MR2557944}, Kowalczyk-Liu-Pacard-Wei \cite{KLW}). As we mentioned above,
for $k=2,$ it is proved in Kowalczyk-Liu-Pacard \cite{MR3148064} that the
moduli space of four-end solutions, modulo rigid motion, is diffeomorpic to
the open interval $\left(  0,1\right)  .$ It is also proved there that for
each $\theta\in\left(  0,\frac{\pi}{2}\right)  ,$ there exists a four-end
solution $u_{\theta}$ which is even with respect to both $x$ and $y$ axis and
the asymptotic line of the nodal set of $u_{\theta}$ in the first quadrant
makes angle $\theta$ with the $x$ axis. Now we observe that if we reflect the
solution to $\left(  \ref{axial}\right)  $ across the $z$ axis, then we get a
solution defined for all $\left(  r,z\right)  \in\mathbb{R}^{2}$ and even in
both variables. Moreover, the nodal set of this solution outside a large ball
also has four components. Hence the two-end solutions in $\mathbb{R}^{3}$ are
in certain sense analogy of the four-end solutions in $\mathbb{R}^{2}.$ The
equations they satisfied are different from each other only by the term
$r^{-1}u_{r}$ which makes the problem inhomogeneous. In fact in $\mathbb{R}%
^{2}$ a major fact used is that there are two linearly independent kernels
$u_{x}$ and $u_{y}$. The construction of four-end solution is done by first
proving that all four-end solutions are \emph{nondegenerate}. Here we face the
problem of degeneracy. We overcome this difficulty by applying global
bifurcation theory developed by Buffoni, Dancer and Toland for real analytical
variety of formal dimension $1$.

Before proceeding to proof of our main results, let us outline the main ideas
of the proof, since this also gives a description of the set of solutions in
Theorem \ref{main}. The main steps of the proof are similar to the
corresponding analysis for four-end solution in dimension two, performed in
\cite{MR2971030} and \cite{MR3148064}. Our basic strategy is to analyze the
structure of the set $M$ of axially symmetric two-end solutions with growth
rate larger than $\sqrt{2}.$ In section \ref{compactness}, we prove that the
solutions in $M$ with some additional natural constraints are compact in
certain sense. In section \ref{Moduli}, we show that $M$ has the structure of
a real analytic variety of formal dimension $1.$ In section \ref{Uniqueness},
we analyze those solutions near the \textquotedblleft
boundary\textquotedblright\ of the moduli space $M$ and prove that they are
unique in certain sense. Combining these results, we conclude the proof by
applying a structure theorem for real analytic varieties. We remark that there
are two main differences between the $\mathbb{R}^{2}$ and $\mathbb{R}^{3}$
case. Firstly, the proof of compactness in 3D case is much more delicate than
the 2D case and detailed asymptotic analysis is needed. Secondly, as we
mentioned, the four-end solutions in 2D are all non-degenerate (\cite{ML}, \cite{MR2971030}),
but we don't know whether it is true for two-end solutions in 3D.

\section{\label{compactness}Compactness of two-end solutions}

The compactness of moduli spaces of minimal surface plays an important role in
the minimal surface theory, for example, it is an important step towards the
classification of certain type of minimal surfaces. We refer to Perez-Ros
\cite{MR1901613} and  references therein for more details on this and other
related subjects. In this section, we shall investigate the compactness
property for two-end solutions of the Allen-Cahn equation.

Throughout the paper we denote by $\mathbb{E}$ the set $[0,+\infty
)\times\mathbb{R}$ and by $\mathbb{E}^{+}$ the set $[0,+\infty)\times
\lbrack0,+\infty).$ Let $\left\{  u_{n}\right\}  $ be a sequence of two-end
solutions which has growth rate $k_{n}>\sqrt{2}.$ (Note that for $k\in
(\frac{\sqrt{2}}{2},\sqrt{2}],$ we don't know whether or not there is a
two-end solution with growth rate $k$.) Then by definition,
\begin{equation}
\left\Vert u_{n}\left(  r,\cdot\right)  -H\left(  \cdot-k_{n}\ln
r-c_{n}\right)  \right\Vert _{L^{\infty}\left(  0,+\infty\right)  }%
\rightarrow0,\text{ as }r\rightarrow+\infty. \label{as}%
\end{equation}
We will show that if the distance of the nodal set of $u_{n}$ to the origin is
uniformly bounded with respect to $n,$ then up to a subsequence, $\left\{
u_{n}\right\}  $ converges strongly to a two-end solution $u_{\infty}.$\ Here
converging strongly means that $u_{n}\rightarrow u_{\infty}$ in $C_{loc}%
^{2}\left(  \mathbb{E}\right)  $\ and there exist constants $k_{\infty
},c_{\infty}$ such that $k_{n}\rightarrow k_{\infty},c_{n}\rightarrow
c_{\infty},$
\[
\left\Vert u_{\infty}\left(  r,\cdot\right)  -H\left(  \cdot-k_{\infty}\ln
r-c_{\infty}\right)  \right\Vert _{L^{\infty}\left(  0,+\infty\right)
}\rightarrow0,\text{ as }r\rightarrow+\infty,
\]
and
\[
u_{n}-H\left(  z-k_{n}\ln r-c_{n}\right)  \rightarrow u_{\infty}-H\left(
z-k_{\infty}\ln r-c_{\infty}\right)
\]
in $L^{\infty}\left(  \mathbb{E}\right)  .$

Under the assumption that a solution $u$ has the asymptotic behavior $\left(
\ref{as}\right)  ,$ one could actually prove that $u$ has certain monotonicity property.

\begin{lemma}
\label{monoto}Suppose $u$ is a two-end solution with growth rate $k>\sqrt{2}.$
Then
\begin{equation}
\partial_{r}u<0\text{ for }r>0;\partial_{z}u>0\text{ for }z>0. \label{m}%
\end{equation}

\end{lemma}

The proof of this result is based on the moving plane method and its proof
will be given in the appendix. By Lemma \ref{monoto}, the two-end solutions we
are analyzing always satisfy $\left(  \ref{m}\right)  .$ We will denote the
nodal set of a solution $u$ in the upper $r$-$z$ plane by
\[
\mathcal{N}_{u}:=\left\{  p\in\mathbb{E}^{+},u\left(  p\right)  =0\right\}  .
\]
Due to the monotonicity property, the set $\mathcal{N}_{u}\cap\partial
\mathbb{E}^{+}$ contains a unique point, call it $\mathcal{P}_{u}.$

In the rest of the paper, we use $C$ and $\alpha$ to denote universal
constants which may vary from step to step. The main result of this section is
the following

\begin{proposition}
\label{compact}Let $u_{n}$ be a sequence of two-end solutions with growth rate
larger than $\sqrt{2}.$ Assume
\[
\left\vert \mathcal{P}_{u_{n}}\right\vert \leq C.
\]
Then there exists a two-end solution $u_{\infty}$ such that up to a
subsequence $\left\{  u_{n}\right\}  $ converges strongly to $u_{\infty}$.
\end{proposition}

The rest of this section is devoted to the proof of Proposition \ref{compact}.

By Lemma \ref{monoto}, for each solution $u_{n},$ $\mathcal{N}_{u_{n}}$ will
be the graph of a function
\[
z=f_{n}\left(  r\right)  ,r\in\lbrack\mathtt{t}_{n},+\infty).
\]
Here $\mathtt{t}_{n}$ satisfies $\mathcal{P}_{u_{n}}=\left(  \mathtt{t}%
_{n},f_{n}\left(  \mathtt{t}_{n}\right)  \right)  .$ In particular, if
$\mathcal{P}_{u_{n}}$ is on the $z$ axis, then $\mathtt{t}_{n}=0.$

\begin{lemma}
\label{diverge}Under the assumption of Proposition \ref{compact}, we have
\[
\lim_{r\rightarrow+\infty}f_{n}\left(  r\right)  =+\infty,
\]
uniformly in $n.$
\end{lemma}

\begin{proof}
We argue by contradiction. If this was not true, there will exist a constant
$C_{0}$ such that for any $l\in\mathbb{N},$ one could find a solution
$u_{n_{l}}$ satisfying
\[
f_{n_{l}}\left(  r\right)  \leq C_{0},\text{ for }r\in\left(  t_{n_{l}%
},l\right)  .
\]
Since $\left\vert u_{n}\right\vert <1,$ the sequence $\left\{  u_{n_{l}%
}\right\}  $ will converge in $C_{loc}^{2}\left(  \mathbb{E}\right)  $ to a
nontrivial solution $W$ whose nodal set is contained in the strip $\left\vert
z\right\vert \leq C_{0}$. Moreover, by the monotonicity of $u_{n},$ $W_{r}<0$
for $r>0.$ The limit $w\left(  z\right)  :=\lim_{r\rightarrow+\infty}W\left(
r,z\right)  $ then exists and is a solution of the Allen-Cahn equation in
dimension $1:$%
\begin{equation}
-w^{\prime\prime}=w-w^{3},\text{ in }\mathbb{R}. \label{1d}%
\end{equation}
By the symmetric property of $W,$ $w\left(  z\right)  =w\left(  -z\right)  .$
We also have $w\left(  z\right)  \rightarrow1,$ as $z\rightarrow+\infty.$ But
$\left(  \ref{1d}\right)  $ does not have a solution with nodal set containing
only in a bounded set $\{ |z|\leq C_{0} \}$. This is a contradiction and the
proof is finished.
\end{proof}

By the monotonicity of $u_{n},$ $\mathcal{N}_{u_{n}}$ is also the graph of a
function over the $z$ axis:
\[
\mathcal{N}_{u_{n}}=\left\{  \left(  r,z\right)  :r=g_{n}\left(  z\right)
\right\}  .
\]
Similar arguments with slight modification as that of $\left(  \ref{diverge}%
\right)  $ imply that
\[
\lim_{z\rightarrow+\infty}g_{n}\left(  z\right)  \rightarrow+\infty,\text{ }%
\]
also uniformly in $n.$

To obtain more information on the functions $f_{n},$ we should use the
\textit{balancing formula} (\cite{dkp-2009}). Let $X=\left(  0,0,1\right)  $
be the constant vector field on $\mathbb{R}^{3}.$ If $u=u\left(  x,y,z\right)
$ is a solution to the Allen-Cahn equation, then one could check that
\[
\operatorname{div}\left\{  \left(  \frac{1}{2}\left\vert \nabla u\right\vert
^{2}+F\left(  u\right)  \right)  X-\left(  \nabla u\cdot X\right)  \nabla
u\right\}  =0.
\]
Here $F\left(  u\right)  =\frac{1}{4}\left(  u^{2}-1\right)  ^{2}.$ Therefore
for each regular domain $\Omega,$ we have the following balancing formula:
\begin{equation}
\int_{\partial\Omega}\left\{  \left(  \frac{1}{2}\left\vert \nabla
u\right\vert ^{2}+F\left(  u\right)  \right)  X-\left(  \nabla u\cdot
X\right)  \nabla u\right\}  \cdot vds=0. \label{b}%
\end{equation}
Here $v$ is the outward unit normal vector of the boundary $\partial\Omega.$

We would like to use $\left(  \ref{b}\right)  $ to control the slope of the
function $f_{n}.$

\begin{lemma}
\label{slope}For each $\varepsilon>0,$ there exists $t_{\varepsilon}>0$ such
that for all $n\in\mathbb{N},$
\[
f_{n}^{\prime}\left(  r_{2}\right)  -f_{n}^{\prime}\left(  r_{1}\right)
<\varepsilon\text{ for }t_{\varepsilon}<r_{1}<r_{2}.
\]

\end{lemma}

\begin{proof}
We first show that for each $\delta>0$ and $\bar{r}>0,$ there exists $r^{\ast
}>\bar{r},$ such that
\[
f_{n}^{\prime}\left(  r_{n}\right)  <\delta,\text{ }n\in\mathbb{N},
\]
for some $r_{n}\in\left(  \bar{r},r^{\ast}\right)  $. Indeed, if this was not
true, then for any $\hat{r}>\bar{r},$ there exists $n,$ such that
\begin{equation}
f_{n}^{\prime}\left(  r\right)  \geq\delta,\text{ }r\in\left(  \bar{r},\hat
{r}\right)  . \label{s1}%
\end{equation}
Let $\bar{r}<r_{1}<r_{2}<\hat{r}.$ Consider the region $\Omega\subset
\mathbb{R}^{3}$ given by
\[
\left\{  \left(  x,y,z\right)  :z>0,L_{1}\left(  r\right)  <z<L_{2}\left(
r\right)  \right\}  ,
\]
where
\begin{align*}
L_{1}\left(  r\right)   &  =f_{n}\left(  r_{1}\right)  -\frac{1}{f_{n}%
^{\prime}\left(  r_{1}\right)  }\left(  r-r_{1}\right)  ,\\
L_{2}\left(  r\right)   &  =f_{n}\left(  r_{2}\right)  -\frac{1}{f_{n}%
^{\prime}\left(  r_{2}\right)  }\left(  r-r_{2}\right)  .
\end{align*}
By the balancing formula,%
\[
\int_{\partial\Omega}\left\{  \left(  \frac{1}{2}\left\vert \nabla
u\right\vert ^{2}+F\left(  u\right)  \right)  X-\left(  \nabla u\cdot
X\right)  \nabla u\right\}  \cdot vds=0.
\]
Using $\left(  \ref{s1}\right)  $ and the fact that $u$ exponentially decays to
$1$ away from the interface, we deduce that as $r_{1}\rightarrow+\infty,$
\[
\int_{\partial\Omega\cap\left\{  z=0\right\}  }\left\{  \left(  \frac{1}%
{2}\left\vert \nabla u\right\vert ^{2}+F\left(  u\right)  \right)  X-\left(
\nabla u\cdot X\right)  \nabla u\right\}  \cdot vds\rightarrow0.
\]
On the other hand,
\begin{align*}
&  \int_{\partial\Omega\cap\left\{  z=L_{1}\left(  r\right)  \right\}
}\left\{  \left(  \frac{1}{2}\left\vert \nabla u\right\vert ^{2}+F\left(
u\right)  \right)  X-\left(  \nabla u\cdot X\right)  \nabla u\right\}  \cdot
vds\\
&  \sim-c_{1}\frac{r_{1}f_{n}^{\prime}\left(  r_{1}\right)  }{\sqrt
{1+f_{n}^{\prime}\left(  r_{1}\right)  ^{2}}},
\end{align*}
where $c_{1}$ is a fixed constant, and
\begin{align*}
&  \int_{\partial\Omega\cap\left\{  z=L_{2}\left(  r\right)  \right\}
}\left\{  \left(  \frac{1}{2}\left\vert \nabla u\right\vert ^{2}+F\left(
u\right)  \right)  X-\left(  \nabla u\cdot X\right)  \nabla u\right\}  \cdot
vds\\
&  \sim c_{1}\frac{r_{2}f_{n}^{\prime}\left(  r_{2}\right)  }{\sqrt
{1+f_{n}^{\prime}\left(  r_{2}\right)  ^{2}}}.
\end{align*}
Combining these estimates, we get a contradiction if $r_{2}$ is large enough.

Fix an $\varepsilon>0,$ if the conclusion of the lemma was not true, then for
any $l>0,$ one could find $l<r_{1}<r_{2}$ such that
\begin{equation}
f_{n}^{\prime}\left(  r_{2}\right)  -f_{n}^{\prime}\left(  r_{1}\right)
=\varepsilon\label{difference}%
\end{equation}
and
\[
\frac{1}{2}\varepsilon<\left\vert f_{n}^{\prime}\left(  r\right)  \right\vert
<\frac{3}{2}\varepsilon,r\in\left(  r_{1},r_{2}\right)  .
\]
In this case, similarly as above, we could estimate
\begin{align*}
&  \int_{\partial\Omega\cap\left\{  z=0\right\}  }\left\{  \left(  \frac{1}%
{2}\left\vert \nabla u\right\vert ^{2}+F\left(  u\right)  \right)  X-\left(
\nabla u\cdot X\right)  \nabla u\right\}  \cdot vds\\
&  =o\left(  1\right)  r_{1}.
\end{align*}
Therefore the balancing formula tells us that
\[
f_{n}^{\prime}\left(  r_{2}\right)  r_{2}-f_{n}^{\prime}\left(  r_{1}\right)
r_{1}=o\left(  1\right)  r_{1}.
\]
This contradicts with $\left(  \ref{difference}\right)  .$
\end{proof}

Lemma \ref{slope} tells us that the slope of the nodal line could not increase
much as $r$ increases. This in particular implies the following

\begin{lemma}
\label{small}Under the assumption of Proposition \ref{compact}, we have
\[
\lim_{r\rightarrow+\infty}f_{n}^{\prime}\left(  r\right)  =0,
\]
uniformly in $n.$
\end{lemma}

\begin{proof}
We argue by contradiction and assume that there exists $\delta>0$ such that
for each $C_{0}>0,$ one could find $t_{0}>C_{0}$ and $n$ satisfying
\[
f_{n}^{\prime}\left(  t_{0}\right)  >\delta.
\]
Then one could find $\bar{t}_{1},\bar{t}_{2},$ such that
\begin{align*}
f_{n}^{\prime}\left(  r\right)   &  \in\left(  \frac{\delta}{2},\delta\right)
,r\in\left(  \bar{t}_{1},\bar{t}_{2}\right)  ,\\
f_{n}^{\prime}\left(  \bar{t}_{1}\right)   &  =\frac{\delta}{2},f_{n}^{\prime
}\left(  \bar{t}_{2}\right)  =\delta.
\end{align*}
This contradicts with Lemma \ref{slope}.
\end{proof}

Since $\left\vert u_{n}\right\vert <1,$ one could assume without loss of
generality that $u_{n}\rightarrow u_{\infty}$ in $C_{loc}^{2}\left(
\mathbb{E}\right)  $ for a solution $u_{\infty}.$ Note that $u_{\infty}$ could
not be identically $0.$ By the monotonicity property of $u_{n},$ $u_{\infty}$
is also monotone, that is, $u_{\infty}$ satisfies $\left(  \ref{monotone}%
\right)  $. We need to prove $u_{n}\rightarrow u_{\infty}$ in the strong
sense, which in particular implies that $u_{\infty}$ is a solution to $\left(
\ref{axial}\right)  $ whose nodal line is asymptotic to a log curve at
infinity and therefore a two-end solution.

The previous lemmas yield certain information on the nodal curve
$\mathcal{N}_{u_{n}},$ but we still don't have precise description of their
shape. On the other hand, we know that in the region where $r$ is large,
locally around the nodal line, the solution resembles the heteroclinic
solution. Our main idea to prove the strong convergence of $\left\{
u_{n}\right\}  $ is to get a precise decay estimate of the solution $u_{n}$ to
$H$ along their nodal lines. For notational simplicity, sometimes we will drop
the subscript $n$ for the function $u_{n},f_{n}.$

To proceed, let us introduce the \textit{Fermi coordinate} around the smooth
curve $z=f\left(  \cdot\right)  .$ This coordinate, denoted by $\left(
r_{1},z_{1}\right)  ,$ is defined through the relation%
\[
\left(  r,z\right)  =\left(  r_{1},f\left(  r_{1}\right)  \right)
+z_{1}\mathtt{n,}%
\]
where $\mathtt{n}$ is the unit normal vector at the point $\left(
r_{1},f\left(  r_{1}\right)  \right)  $ of the curve $z=f\left(  \cdot\right)
,$ upward pointed. Explicitly,
\begin{equation}
\left\{
\begin{array}
[c]{c}%
r=r_{1}-z_{1}\frac{f^{\prime}}{\sqrt{1+f^{\prime2}}},\\
z=f+z_{1}\frac{1}{\sqrt{1+f^{\prime2}}}.
\end{array}
\right.  \label{Fermi2}%
\end{equation}
Here the function $f$ is evaluated at the point $r_{1}.$ The map $X:\left(
r_{1},z_{1}\right)  \rightarrow\left(  r,z\right)  $ from the Fermi coordinate
to the original coordinate system will be a diffeomorphism in certain tubular
neighborhood of the graph of $f.$

From $\left(  \ref{Fermi2}\right)  $, we get%
\begin{align}
\left(
\begin{array}
[c]{cc}%
\partial_{r}r_{1} & \partial_{z}r_{1}\\
\partial_{r}z_{1} & \partial_{z}z_{1}%
\end{array}
\right)   &  =\left(
\begin{array}
[c]{cc}%
\partial_{r_{1}}r & \partial_{z_{1}}r\\
\partial_{r_{1}}z & \partial_{z_{1}}z
\end{array}
\right)  ^{-1}\nonumber\\
&  =\left(
\begin{array}
[c]{cc}%
B & -\frac{f^{\prime}}{\sqrt{1+f^{\prime2}}}\\
f^{\prime}B & \frac{1}{\sqrt{1+f^{\prime2}}}%
\end{array}
\right)  ^{-1}\nonumber\\
&  =\left(
\begin{array}
[c]{cc}%
\frac{1}{\left(  1+f^{\prime2}\right)  B} & \frac{f^{\prime}}{\left(
1+f^{\prime2}\right)  B}\\
-\frac{f^{\prime}}{\sqrt{1+f^{\prime2}}} & \frac{1}{\sqrt{1+f^{\prime2}}}%
\end{array}
\right)  . \label{derivative}%
\end{align}
Here
\[
B=1-\frac{z_{1}f^{\prime\prime}}{\left(  1+f^{\prime2}\right)  ^{\frac{3}{2}}%
}.
\]
Let us denote by $\Delta_{\left(  r,z\right)  }=\partial_{r}^{2}+\partial
_{z}^{2}$ and $\Delta_{\left(  r_{1},z_{1}\right)  }=\partial_{r_{1}}%
^{2}+\partial_{z_{1}}^{2}.$ These two operators are related by%
\begin{align}
\Delta_{\left(  r,z\right)  }  &  =\Delta_{\left(  r_{1},z_{1}\right)
}+\left(  \frac{1}{A}-1\right)  \partial_{r_{1}}^{2}\nonumber\\
&  +\frac{1}{2}\frac{\partial_{z_{1}}A}{A}\partial_{z_{1}}-\frac{1}{2}%
\frac{\partial_{r_{1}}A}{A^{2}}\partial_{r_{1}}, \label{Laplacian}%
\end{align}
where
\begin{equation}
A=1+f^{\prime2}-2z_{1}\frac{f^{\prime\prime}}{\sqrt{1+f^{\prime2}}}+z_{1}%
^{2}\frac{f^{\prime\prime2}}{\left(  1+f^{\prime2}\right)  ^{2}}. \label{a}%
\end{equation}
This follows from direct computation, see \cite{MR2557944} for more details.
From formula $\left(  \ref{Laplacian}\right)  $, one sees that if $f^{\prime
},f^{\prime\prime},f^{\left(  3\right)  }$ is small, then $\Delta_{\left(
r,z\right)  }$ is a small perturbation of $\Delta_{\left(  r_{1},z_{1}\right)
}.$ One should keep in mind that in $\left(  \ref{Laplacian}\right)  ,$ there
are terms (in $\partial_{r_{1}}A$) involving the third derivatives of $f$ and
these terms should be dealt with very carefully.

For any function $\Theta\left(  r,z\right)  $, the pullback of $\Theta$ by the
diffeomorphism $X$ will be defined as
\[
X^{\ast}\Theta\left(  r_{1},z_{1}\right)  :=\Theta\circ X\left(  r_{1}%
,z_{1}\right)  .
\]
Occasionally, $X^{\ast}\Theta$ will also be denoted by $\Theta^{\ast}.$ Keep
in mind that $X^{\ast}\Theta$ is only defined in the region where the Fermi
coordinate is well defined. Conversely, for any function $\Theta\left(
r_{1},z_{1}\right)  $, we set $\left(  X^{-1}\right)  ^{\ast}\Theta\left(
r,z\right)  =\Theta\circ X^{-1}\left(  r,z\right)  .$

To prove the compactness, we need to know the precise asymptotic behavior of
$u_{n}.$ We shall define suitable approximate solutions and analyze the error
between the approximate and true solutions. It turns out that to get a good
approximate solution, it will be convenient to use the Fermi coordinate.
Clearly there is a technical issue concerning the Fermi coordinate. Namely,
the Fermi coordiante is in general not well defined in the whole $r$-$z$
plane. Note that for $r_{1}$ large, by Lemma \ref{small}, $f^{\prime}$ is
small. We also know that $u$ is close to the one dimensional heteroclinic
solution locally around the nodal line at far away. Hence intuitively, for $r$
large, $f^{\prime\prime}$,$f^{\left(  3\right)  },...,$ should also be small.
Indeed, we have

\begin{lemma}
As $r$ tends to infinity, $\left\vert f^{\prime\prime}\left(  r\right)
\right\vert , \left\vert f^{(3)}\left(  r\right)  \right\vert , \left\vert
f^{(4)}\left(  r\right)  \right\vert $ tend to zero.
\end{lemma}

\begin{proof}
Let $r_{0}$ be a fixed large constant. Around the point $\left(
r_{0},f\left(  r_{0}\right)  \right)  ,$ let us consider the line $l_{1}$
\[
z-f\left(  r_{0}\right)  =f^{\prime}\left(  r_{0}\right)  \left(
r-r_{0}\right)  .
\]
Let $l_{2}$ be the line orthogonal to it and passing through $\left(
r_{0},f\left(  r_{0}\right)  \right)  .$ Use these two lines to build an
orthogonal coordinate system, denoted by $\left(  s,t\right)  ,$ where the $s$
coordinate corresponding to line $l_{1}.$ We know that $u$ is close to
$H\left(  t\right)  .$ By analyzing the equation satisfied by the error
$\phi\left(  r,z\right)  :=u\left(  r,z\right)  -H\left(  t\right)  ,$ we find
that for $r_{0}$ large, around $\left(  r_{0},f\left(  r_{0}\right)  \right)
,$ the $C^{2}$ norm of $\phi$ is small.

Observe that $u\left(  r,f\left(  r\right)  \right)  =0.$ Hence using the
definition of $\phi,$
\begin{equation}
H\left(  \frac{f\left(  r\right)  -\left[  f\left(  r_{0}\right)  +f^{\prime
}\left(  r_{0}\right)  \left(  r-r_{0}\right)  \right]  }{\sqrt{1+\left(
f^{\prime}\left(  r_{0}\right)  \right)  ^{2}}}\right)  +\phi\left(
r,f\left(  r\right)  \right)  =0. \label{nodal}%
\end{equation}
Differentiate this equation with respect to $r,$ using the fact that
$\partial_{z}\phi$ is small, we get $\left\vert f^{\prime}\left(  r\right)
-f^{\prime}\left(  r_{0}\right)  \right\vert $ is small, which we already
know. Differentiate $\left(  \ref{nodal}\right)  $ twice, we obtain
\[
H^{\prime}\left(  \cdot\right)  f^{\prime\prime}\left(  r\right)
+H^{\prime\prime}\left(  \cdot\right)  \left(  f^{\prime}\left(  r\right)
-f^{\prime}\left(  r_{0}\right)  \right)  ^{2}+\partial_{r}^{2}\phi
+2\partial_{r}\partial_{z}\phi f^{\prime}+\partial_{z}\phi f^{\prime\prime
}=0.
\]
Using the fact that $\partial_{z}\phi$ is small, we deduce $f^{\prime\prime
}\left(  r_{0}\right)  $ is small.

Indeed, one could continue this process and show that the higher order
derivatives $f^{\left(  3\right)  },f^{\left(  4\right)  }$ are also small.
\end{proof}

Since $\left\vert f^{\prime\prime}\right\vert $ is small, the Fermi coordinate
actually will be well defined in a very large tubular neighborhood of
$\mathcal{N}_{u}.$ For later purposes, we need to be more precise in
describing the size of the region where Fermi coordinate is well defined.

For each point $Z=\left(  r,z\right)  \in\mathbb{E},$ we use $dist\left(
Z,\mathcal{N}_{u}\right)  $ to denote the distance between $Z$ and the curve
$\mathcal{N}_{u}.$ That is,
\[
dist\left(  Z,\mathcal{N}_{u}\right)  =\min\left\{  \left\vert Z-p\right\vert
,p\in\mathcal{N}_{u}\right\}  .
\]
Let $\pi\left(  Z\right)  $ be the set of points which realize this distance.
Hence
\[
\pi\left(  Z\right)  =\left\{  p\in\mathcal{N}_{u}:\left\vert Z-p\right\vert
=dist\left(  Z,\mathcal{N}_{u}\right)  \right\}  .
\]
For each $r_{1},$ let us consider the set
\[
D_{r_{1}}:=\left\{  Z:\pi\left(  Z\right)  =\left\{  \left(  r_{1},f\left(
r_{1}\right)  \right)  \right\}  \right\}  .
\]
Note that by the triangle inequality, $D_{r_{1}}$ is connected. Then
\[
D_{r_{1}}=\left\{  X\left(  r_{1},t\right)  :t\in\left(  -d_{1}\left(
r_{1}\right)  ,d_{2}\left(  r_{1}\right)  \right)  \right\}  ,
\]
for some functions $d_{1}$ and $d_{2}$($d_{i}$ may not be differentiable).

Define
\[
\bar{d}\left(  r_{1}\right)  =\min\left\{  d_{1}\left(  r_{1}\right)
,d_{2}\left(  r_{1}\right)  ,3f\left(  r_{1}\right)  \right\}
\]
and
\[
\bar{D}_{r_{1}}=\left\{  X\left(  r_{1},t\right)  :t\in\left(  -\bar{d}\left(
r_{1}\right)  ,\bar{d}\left(  r_{1}\right)  \right)  \right\}  .
\]
Here the constant $3$ has not particular importance. Fix a large constant
$r_{0}$. Let $\mathcal{B}_{u}:=$ $\mathcal{\cup}_{r_{1}>r_{0}}\bar{D}_{r_{1}%
}.$ It is worth to be pointed out that in principle the boundary of
$\mathcal{B}_{u}$ could be quite complicated. Let $\eta$, whose existence is
related to assumption $\left(  \ref{Assum}\right)  $ below, be a smooth cutoff
function equals $0$ outside $\mathcal{\cup}_{r_{1}>r_{0}-1}\bar{D}_{r_{1}}$
and
\[
X^{\ast}\eta\left(  r_{1},z_{1}\right)  =1,\text{ for }z_{1}\in\left(
-\bar{d}\left(  r_{1}\right)  +1,\bar{d}\left(  r_{1}\right)  -1\right)
,r_{1}>r_{0}.
\]
Similarly, let $\mathcal{B}_{u}^{+}=\mathcal{B}_{u}\cap\mathbb{E}^{+}$ and we
define a similar cutoff function $\eta^{+}$ supported in $\mathcal{B}_{u}^{+}%
$(note that the notation $\eta^{+}$ does not mean the positive part of $\eta$).

Since $\mathcal{B}_{u}$ is not the whole $r$-$z$ plane, we shall use the
cutoff function $\eta$ to define a smooth function $\mathcal{H}_{1}$ by the
formula
\[
\mathcal{H}_{1}=\eta H_{1}+\left(  1-\eta\right)  \frac{H_{1}}{\left\vert
H_{1}\right\vert }.
\]
Here the function $H_{1}$ is defined through
\[
X^{\ast}H_{1}\left(  r_{1},z_{1}\right)  =H\left(  r_{1}-h\left(
z_{1}\right)  \right)  ,
\]
where $h$ is a small function to be determined later. We emphasize that since
we are only interested in the behavior of $u$ outside of a bounded set,
$\mathcal{H}_{1}$ should be regarded as a function only defined on the set
$\mathbb{\hat{E}},$ where
\[
\mathbb{\hat{E}=E}\backslash\left\{  \left(  r,z\right)  :-f\left(
r_{0}\right)  +\frac{1}{f^{\prime}\left(  r_{0}\right)  }\left(
r-r_{0}\right)  <z<f\left(  r_{0}\right)  -\frac{1}{f^{\prime}\left(
r_{0}\right)  }\left(  r-r_{0}\right)  \right\}  .
\]
For any function $\Theta,$ let $\Theta^{s}\left(  r,z\right)  =\Theta\left(
r,-z\right)  .$

With all the previous notations introduced, we now define an approximate
solution $\bar{u}$ as
\[
\bar{u}=\mathcal{H}_{1}+\mathcal{H}_{1}^{s}+1.
\]
The function $\bar{u}$ is defined on the set $\mathbb{\hat{E}}$, rather than
the whole plane $\mathbb{E}$. Let $\phi=u-\bar{u}$ be the difference between
the true solution $u$ and the approximate solution $\bar{u}.$ Note that since
$\bar{u}$ is even with respect to the $z$ variable, $\phi$ is also even.
Introduce the function $H_{1}^{\prime}$ by
\[
X^{\ast}H_{1}^{\prime}\left(  r_{1},z_{1}\right)  =H^{\prime}\left(
z_{1}-h\left(  r_{1}\right)  \right)  .
\]
Then the small function $h$ appeared in the definition of the function $H_{1}$
is required to satisfy the following orthogonality condition:%
\begin{equation}
\int_{\mathbb{R}}X^{\ast}\left(  \phi\eta^{+}H_{1}^{\prime}\right)
dz_{1}=0,\text{ }r_{1}>r_{0}. \label{ortho}%
\end{equation}
The existence of $h$ is guaranteed by the following:

\begin{lemma}
\label{functionh}There exists a small function $h$(in $C^{2}$ sense)
satisfying $\left(  \ref{ortho}\right)  .$
\end{lemma}

\begin{proof}
This follows from similar arguments as that of \cite{MR3148064}. The basic
idea is to use the fact that $\phi$ is small and apply the implicit function
theorem. We omit the details.
\end{proof}

One advantage of using the Fermi coordinate with respect to the nodal curve is
that $h$ could be estimate in terms of $\phi$ and $f.$ To see this, let
$\mathcal{N}^{s}\left(  u\right)  $ be the graph of the function $z=-f\left(
r\right)  .$ For $Z=\left(  r,z\right)  \in\mathbb{E},$ let
\[
\mathcal{D}\left(  r,z\right)  =dist\left(  Z,\mathcal{N}\left(  u\right)
\right)  +dist\left(  Z,\mathcal{N}^{s}\left(  u\right)  \right)  ,
\]
that is, the sum of the distance of $Z$ to the nodal line in the upper plane
and lower plane. Set $D\left(  r_{1}\right)  =\mathcal{D}\left(
r_{1},f\left(  r_{1}\right)  \right)  .$

\begin{lemma}
\label{h}For $r_{1}>r_{0},$
\[
\left\vert h\left(  r_{1}\right)  \right\vert \leq C\left\vert X^{\ast}%
\phi\left(  r_{1},0\right)  \right\vert +Ce^{-\sqrt{2}D\left(  r_{1}\right)
}.
\]

\end{lemma}

\begin{proof}
Since $\phi=u-\bar{u},$ around $\mathcal{N}_{u}$ we have
\begin{equation}
X^{\ast}\phi\left(  r_{1},z_{1}\right)  =X^{\ast}u\left(  r_{1},z_{1}\right)
-\left[  H\left(  z_{1}-h\left(  r_{1}\right)  \right)  +X^{\ast}%
\mathcal{H}_{1}^{s}+1\right]  . \label{l1}%
\end{equation}
Setting $z_{1}=0$ in the above equation, using the fact that $X^{\ast}u\left(
r_{1},0\right)  =0,$ we find
\[
X^{\ast}\phi\left(  r_{1},0\right)  +H\left(  -h\left(  r_{1}\right)  \right)
+1+X^{\ast}\mathcal{H}_{1}^{s}\left(  r_{1},0\right)  =0.
\]
On the other hand, by the definition of $\mathcal{H}_{1}^{s}$ and the
asymptotic behavior of $H,$%
\[
\left\vert 1+X^{\ast}\mathcal{H}_{1}^{s}\left(  r_{1},0\right)  \right\vert
\leq Ce^{-\sqrt{2}D\left(  r_{1}\right)  }.
\]
Therefore,
\[
\left\vert h\left(  r_{1}\right)  \right\vert \leq C\left\vert X^{\ast}%
\phi\left(  r_{1},0\right)  \right\vert +Ce^{-\sqrt{2}D\left(  r_{1}\right)
}.
\]
This completes the proof.
\end{proof}

Following the above arguments and differentiate  equation $\left(
\ref{l1}\right)  $ with respect to $r_{1}$ and setting $z_{1}=0$ in the
obtained equation, we get
\[
\left\vert h^{\prime}\left(  r_{1}\right)  \right\vert \leq C\left\vert
\partial_{r_{1}}X^{\ast}\phi\left(  r_{1},0\right)  \right\vert +Ce^{-\sqrt
{2}D\left(  r_{1}\right)  },
\]
Similar arguments yield  estimates for the higher order derivatives and Holder norms.

One of the main ingredients in the proof of Proposition \ref{compact} is to
get suitable decay estimate for the function $\phi.$ In this respect, we shall
first of all prove the following estimate.

\begin{proposition}
\label{P3}For $r_{1}>r_{0},$ the function $\phi$ satisfies%
\[
\left\Vert X^{\ast}\phi\left(  r_{1},\cdot\right)  \right\Vert _{L^{\infty}%
}\leq Ce^{-r_{1}}+Ce^{-D\left(  r_{1}\right)  }+\frac{C}{r_{1}^{2}}.
\]

\end{proposition}

We shall first of all prove this proposition under an additional assumption on
the size of the Fermi coordinate. More precisely, at this stage, we assume
that%
\begin{equation}
\left\vert \bar{d}^{\prime}\right\vert \leq\frac{1}{2},\left\vert \bar
{d}^{\prime\prime}\right\vert \leq C\text{ and }\bar{d}\left(  r\right)
\geq\frac{2D\left(  r\right)  }{3}.\text{ } \tag{A}\label{Assum}%
\end{equation}
Later we will indicate the necessary modification of the proof if this
assumption is not a priori satisfied.

Clearly $\phi$ satisfies
\begin{equation}
L\phi:=-\Delta\phi+\left(  3\bar{u}^{2}-1\right)  \phi=E\left(  \bar
{u}\right)  +P\left(  \phi\right)  . \label{fi}%
\end{equation}
Here the notation $E\left(  \bar{u}\right)  $ stands for
\[
\bar{u}_{rr}+r^{-1}\bar{u}_{r}+\bar{u}_{zz}+\bar{u}-\bar{u}^{3},
\]
which is the error of the approximate solution $\bar{u},$ and
\[
P\left(  \phi\right)  =-3\bar{u}\phi^{2}-\phi^{3}%
\]
is a higher order perturbation term.

For any function $\Theta,$ we define the projection of $\Theta$ onto $\eta
^{+}H_{1}^{\prime}$ as
\[
\Theta_{1}^{\Vert}=\frac{\int_{\mathbb{R}}X^{\ast}\left(  \Theta\eta^{+}%
H_{1}^{\prime}\right)  dz_{1}}{\int_{\mathbb{R}}X^{\ast}\left(  \left(
\eta^{+}\right)  ^{2}H_{1}^{\prime2}\right)  dz_{1}}\eta^{+}H_{1}^{\prime}.
\]
Since the solutions are even in the $z$ variable, we also define%
\[
\Theta^{\Vert}=\Theta_{1}^{\Vert}+\left[  \Theta_{1}^{\Vert}\right]  ^{s}.
\]
Finally, set
\[
\Theta^{\bot}:=\Theta-\Theta^{\Vert}.
\]
Equation $\left(  \ref{fi}\right)  $ could be written in the form
\[
L\phi=\left[  E\left(  \bar{u}\right)  \right]  ^{\bot}+\left[  E\left(
\bar{u}\right)  \right]  ^{\Vert}+P\left(  \phi\right)  .
\]
As we will see later, $\left[  E\left(  \bar{u}\right)  \right]  ^{\Vert}$ is
small(in suitable norm) compared to $\phi$, and $\phi$ is then essentially
controlled by $\left[  E\left(  \bar{u}\right)  \right]  ^{\bot}.$ As a
crucial step towards the proof of Proposition \ref{P3}, we shall take the task
of analyzing $E\left(  \bar{u}\right)  .$ By definition
\[
\bar{u}=\mathcal{H}_{1}+\mathcal{H}_{1}^{s}+1.
\]
A simple manipulation leads to the following expansion
\begin{equation}
E\left(  \bar{u}\right)  =E\left(  \mathcal{H}_{1}\right)  +E\left(
\mathcal{H}_{1}^{s}\right)  +3\left(  \mathcal{H}_{1}^{s}+1\right)  \left(
\mathcal{H}_{1}^{2}-1\right)  +3\left(  \mathcal{H}_{1}+1\right)  \left(
\mathcal{H}_{1}^{s}+1\right)  ^{2}. \label{Eu}%
\end{equation}
Certainly the function $E\left(  \bar{u}\right)  $ is also even in the $z$
variable. One expects that in the upper plane $\mathbb{E}^{+}$, the main order
of $E\left(  \bar{u}\right)  $ should be $E\left(  \mathcal{H}_{1}\right)  $
and one of the interaction term $3\left(  \mathcal{H}_{1}^{s}+1\right)
\left(  \mathcal{H}_{1}^{2}-1\right)  .$

We first analyze the projection of $E\left(  \mathcal{H}_{1}\right)  $ onto
$\eta^{+}H_{1}^{\prime}.$ For technical reasons, we introduce a small
perturbation of $f\,$\ due to the presence of the modulation function $h.$ Let
$p=f+\hat{h},$ where $\hat{h}:=\sqrt{1+f^{\prime2}}h.$ Throughout the paper,
we set $\mathbf{c}_{0}=\int_{\mathbb{R}}H^{\prime2}\left(  t\right)  dt.$ The
notation $O\left(  g\right)  $ represents a function in $r_{1}$ such that
$\left\vert O\left(  g\right)  \right\vert \leq C\left\vert g\left(
r_{1}\right)  \right\vert .$

\begin{lemma}
\label{P2}For $r_{1}>r_{0},$
\begin{align*}
\int_{\mathbb{R}}X^{\ast}\left(  \eta^{+}H_{1}^{\prime}E\left(  \mathcal{H}%
_{1}\right)  \right)  dz_{1}  &  =\left(  1+o\left(  1\right)  \right)
\frac{\mathbf{c}_{0}}{r_{1}}\left(  \frac{r_{1}p^{\prime}}{\sqrt{1+p^{\prime
2}}}\right)  ^{\prime}+O\left(  h^{\prime\prime}f^{\prime\prime}\right)
+O\left(  h^{\prime2}\right) \\
+O\left(  h^{\prime}f^{\prime\prime}\right)  +O\left(  h^{\prime}%
r^{-1}\right)   &  +O\left(  e^{-2\sqrt{2}\bar{d}\left(  r_{1}\right)
}\right)  +O\left(  h^{\prime}h^{\prime\prime}\right)  +O\left(
hf^{\prime\prime}\right) \\
&  +O\left(  hf^{\prime\prime\prime}\right)  +O\left(  f^{\prime\prime
3}\right)  +O\left(  \frac{f^{\prime3}}{r_{1}^{3}}\right)  .
\end{align*}

\end{lemma}

\begin{proof}
In the region where the cutoff function $\eta^{+}\neq0,$ $E\left(
\mathcal{H}_{1}\right)  $ could be expressed in terms of the Fermi coordinate
$\left(  r_{1},z_{1}\right)  :$
\begin{align*}
E\left(  \mathcal{H}_{1}\right)   &  =\Delta_{\left(  r,z\right)  }%
H_{1}+r^{-1}\partial_{r}H_{1}+H_{1}-H_{1}^{3}\\
&  =A^{-1}\partial_{r_{1}}^{2}H_{1}+\partial_{z_{1}}^{2}H_{1}\\
&  +\frac{\partial_{z_{1}}A}{2A}\partial_{z_{1}}H_{1}-\frac{\partial_{r_{1}}%
A}{2A^{2}}\partial_{r_{1}}H_{1}\\
&  +r^{-1}\left(  \partial_{r_{1}}H_{1}\partial_{r}r_{1}+\partial_{z_{1}}%
H_{1}\partial_{r}z_{1}\right)  +H_{1}-H_{1}^{3}.
\end{align*}
Obviously $\partial_{z_{1}}H_{1}=H^{\prime},\partial_{z_{1}}^{2}%
H_{1}=H^{\prime\prime},$ and
\[
\partial_{r_{1}}H_{1}=-h^{\prime}H^{\prime},\partial_{r_{1}}^{2}%
H_{1}=-h^{\prime\prime}H^{\prime}+h^{\prime2}H^{\prime\prime},
\]
where $H^{\prime}$ and $H^{\prime\prime}$ are evaluated at $z_{1}-h\left(
r_{1}\right)  .$ Since $H^{\prime\prime}+H-H^{3}=0,$ we obtain
\begin{align}
E\left(  \mathcal{H}_{1}\right)   &  =A^{-1}\left(  -h^{\prime\prime}%
H^{\prime}+h^{\prime2}H^{\prime\prime}\right)  +\left(  \frac{\partial_{z_{1}%
}A}{2A}+r^{-1}\partial_{r}z_{1}\right)  H^{\prime}\nonumber\\
&  +\left(  -\frac{\partial_{r_{1}}A}{2A^{2}}+r^{-1}\partial_{r}r_{1}\right)
\left(  -h^{\prime}H^{\prime}\right)  . \label{EH1}%
\end{align}
It follows that%
\begin{align*}
&  \int_{\mathbb{R}}X^{\ast}\left(  \eta^{+}H_{1}^{\prime}E\left(
\mathcal{H}_{1}\right)  \right)  dz_{1}\\
&  =\int_{\mathbb{R}}\eta^{+\ast}A^{-1}\left(  -h^{\prime\prime}H^{\prime
}+h^{\prime2}H^{\prime\prime}\right)  H^{\prime}+\int_{\mathbb{R}}\eta^{+\ast
}\left(  \frac{\partial_{z_{1}}A}{2A}+r^{-1}\partial_{r}z_{1}\right)
H^{\prime2}\\
&  +\int_{\mathbb{R}}\eta^{+\ast}\left(  \frac{\partial_{r_{1}}A}{2A^{2}%
}-r^{-1}\partial_{r}r_{1}\right)  h^{\prime}H^{\prime2}+O\left(  e^{-2\sqrt
{2}\bar{d}}\right) \\
&  =-h^{\prime\prime}\int_{\mathbb{R}}\eta^{+\ast}A^{-1}H^{\prime2}%
+\int_{\mathbb{R}}\eta^{+\ast}\left(  \frac{\partial_{z_{1}}A}{2A}%
+r^{-1}\partial_{r}z_{1}\right)  H^{\prime2}\\
&  +O\left(  h^{\prime2}\right)  +O\left(  h^{\prime}f^{\prime\prime}\right)
+O\left(  h^{\prime}f^{\prime\prime\prime}\right)  +O\left(  h^{\prime}%
r^{-1}\right)  +O\left(  e^{-2\sqrt{2}\bar{d}}\right)  .
\end{align*}
The appearing of the term $O\left(  e^{-2\sqrt{2}\bar{d}}\right)  $ is due to
the fact that in the upper boundary of $\mathcal{B}_{u},$ $E\left(
\mathcal{H}_{1}\right)  =O\left(  e^{-\sqrt{2}d}\right)  .$

To proceed, we shall calculate the second term in the above expression, which
roughly speaking should be the main order term of the projection. We have
\begin{align*}
&  \frac{\partial_{z_{1}}A}{2A}+r^{-1}\partial_{r}z_{1}\\
&  =\frac{-\frac{f^{\prime\prime}}{\sqrt{1+f^{\prime2}}}+z_{1}\frac
{f^{\prime\prime2}}{\left(  1+f^{\prime2}\right)  ^{2}}}{A}-\frac{f^{\prime}%
}{\left(  r_{1}-z_{1}\frac{f^{\prime}}{\sqrt{1+f^{\prime2}}}\right)
\sqrt{1+f^{\prime2}}}\\
&  =\left(  -\frac{f^{\prime\prime}}{\left(  1+f^{\prime2}\right)  ^{\frac
{3}{2}}}+z_{1}\frac{f^{\prime\prime2}}{\left(  1+f^{\prime2}\right)  ^{3}%
}\right)  \left(  1+2z_{1}\frac{f^{\prime\prime}}{\left(  1+f^{\prime
2}\right)  ^{\frac{3}{2}}}+O\left(  f^{\prime\prime2}\right)  \right) \\
&  -\frac{f^{\prime}}{r_{1}\sqrt{1+f^{\prime2}}}\left(  1+z_{1}\frac
{f^{\prime}}{r_{1}\sqrt{1+f^{\prime2}}}+O\left(  \frac{f^{\prime2}}{r_{1}^{2}%
}\right)  \right) \\
&  =-\frac{f^{\prime\prime}}{\left(  1+f^{\prime2}\right)  ^{\frac{3}{2}}%
}-\frac{z_{1}f^{\prime\prime2}}{\left(  1+f^{\prime2}\right)  ^{3}}+O\left(
f^{\prime\prime3}\right) \\
&  -\frac{f^{\prime}}{r_{1}\sqrt{1+f^{\prime2}}}-\frac{z_{1}f^{\prime2}}%
{r_{1}^{2}\left(  1+f^{\prime2}\right)  }+O\left(  \frac{f^{\prime3}}%
{r_{1}^{3}}\right)  .
\end{align*}
Using this estimate, we get
\begin{align}
&  \int_{\mathbb{R}}X^{\ast}\left(  \eta^{+}H_{1}^{\prime}E\left(
\mathcal{H}_{1}\right)  \right) \label{h1pro}\\
&  =-\left(  1+o\left(  1\right)  \right)  \left[  \frac{\mathbf{c}%
_{0}h^{\prime\prime}}{1+f^{\prime2}}+\frac{\mathbf{c}_{0}f^{\prime\prime}%
}{\left(  1+f^{\prime2}\right)  ^{\frac{3}{2}}}+\frac{\mathbf{c}_{0}f^{\prime
}}{r_{1}\sqrt{1+f^{\prime2}}}\right] \nonumber\\
&  +O\left(  h^{\prime2}\right)  +O\left(  h^{\prime}f^{\prime\prime}\right)
+O\left(  h^{\prime}f^{\prime\prime\prime}\right)  +O\left(  h^{\prime}%
r^{-1}\right) \nonumber\\
&  +O\left(  h^{\prime\prime}f^{\prime\prime}\right)  +O\left(  e^{-2\sqrt
{2}\bar{d}\left(  r_{1}\right)  }\right)  .\nonumber
\end{align}
In this expression, we are mainly interested in the first term. We calculate
\begin{align*}
&  \frac{r_{1}h^{\prime\prime}}{1+f^{\prime2}}+\frac{r_{1}f^{\prime\prime}%
}{\left(  1+f^{\prime2}\right)  ^{\frac{3}{2}}}+\frac{f^{\prime}}%
{\sqrt{1+f^{\prime2}}}\\
&  =\frac{r_{1}}{1+f^{\prime2}}\left[  \frac{\hat{h}^{\prime\prime}}%
{\sqrt{1+f^{\prime2}}}+2\hat{h}^{\prime}\left(  \frac{1}{\sqrt{1+f^{\prime2}}%
}\right)  ^{\prime}+\hat{h}\left(  \frac{1}{\sqrt{1+f^{\prime2}}}\right)
^{\prime\prime}\right] \\
&  +\frac{r_{1}f^{\prime\prime}}{\left(  1+f^{\prime2}\right)  ^{\frac{3}{2}}%
}+\frac{f^{\prime}}{\sqrt{1+f^{\prime2}}}\\
&  =\left(  \frac{r_{1}p^{\prime}}{\sqrt{1+p^{\prime2}}}\right)  ^{\prime
}+O\left(  h^{\prime}r_{1}h^{\prime\prime}\right)  +O\left(  h^{\prime}%
r_{1}f^{\prime\prime}\right) \\
&  +O\left(  hr_{1}f^{\prime\prime}\right)  +O\left(  hr_{1}f^{\prime
\prime\prime}\right)  +O\left(  h^{\prime}\right)  .
\end{align*}
The conclusion of the lemma then follows from this estimate.
\end{proof}

With the projection being understood, we proceed to analyze the orthogonal part.

\begin{lemma}
\label{P}For $r_{1}>r_{0},$ the following estimate is true:
\begin{align}
E\left(  \mathcal{H}_{1}\right)  -\left[  E\left(  \mathcal{H}_{1}\right)
\right]  _{1}^{\Vert}  &  =-\left(  z_{1}-h\right)  \left(  \frac
{f^{\prime\prime2}}{\left(  1+f^{\prime2}\right)  ^{3}}+\frac{f^{\prime2}%
}{r_{1}^{2}\left(  1+f^{\prime2}\right)  }\right)  H^{\prime}\nonumber\\
&  +O\left(  h^{\prime\prime}f^{\prime\prime}\right)  +O\left(  e^{-\sqrt{2}%
D}\right)  +O\left(  h^{\prime}r^{-1}\right)  +O\left(  h^{\prime\prime
2}\right) \nonumber\\
&  +O\left(  f^{\prime\prime3}\right)  +O\left(  \frac{f^{\prime3}}{r_{1}^{3}%
}\right)  +O\left(  h^{\prime2}\right)  +O\left(  h^{\prime}f^{\prime\prime
}\right)  +O\left(  h^{\prime}f^{\prime\prime\prime}\right)  . \label{H1or}%
\end{align}

\end{lemma}

\begin{proof}
Let us consider typical terms appeared in $\left(  \ref{EH1}\right)  .$
\begin{align*}
&  A^{-1}h^{\prime\prime}H^{\prime}-\left[  A^{-1}h^{\prime\prime}H^{\prime
}\right]  _{1}^{\Vert}\\
&  =A^{-1}h^{\prime\prime}H^{\prime}-\frac{\int_{\mathbb{R}}A^{-1}%
h^{\prime\prime}\eta^{+\ast}H^{\prime2}}{\int_{\mathbb{R}}\left(  \eta^{+\ast
}H^{\prime}\right)  ^{2}}\eta^{+\ast}H^{\prime}\\
&  =\frac{h^{\prime\prime}H^{\prime}}{1+f^{\prime2}}\left[  1-\frac
{\eta^{+\ast}}{\int_{\mathbb{R}}\left(  \eta^{+\ast}H^{\prime}\right)  ^{2}%
}\int_{\mathbb{R}}\eta^{+\ast}H^{\prime2}\right]  +O\left(  f^{\prime\prime
}h^{\prime\prime}\right) \\
&  =O\left(  h^{\prime\prime}\right)  \left(  \int_{\mathbb{R}}\left(
\eta^{+\ast}H^{\prime}\right)  ^{2}-\eta^{+\ast}\int_{\mathbb{R}}\eta^{+\ast
}H^{\prime2}\right)  H^{\prime}+O\left(  f^{\prime\prime}h^{\prime\prime
}\right)  .
\end{align*}
Note that
\[
\left(  \int_{\mathbb{R}}\left(  \eta^{+\ast}H^{\prime}\right)  ^{2}%
-\eta^{+\ast}\int_{\mathbb{R}}\eta^{+\ast}H^{\prime2}\right)  H^{\prime
}=O\left(  e^{-\sqrt{2}f}\right)  .
\]
Hence
\[
A^{-1}h^{\prime\prime}H^{\prime}-\left[  A^{-1}h^{\prime\prime}H^{\prime
}\right]  _{1}^{\Vert}=O\left(  h^{\prime\prime2}\right)  +O\left(
e^{-\sqrt{2}D}\right)  +O\left(  f^{\prime\prime}h^{\prime\prime}\right)  .
\]

Next, we have calculated that
\begin{align*}
\frac{\partial_{z_{1}}A}{2A}+r^{-1}\partial_{r}z_{1}  &  =-\frac
{f^{\prime\prime}}{\left(  1+f^{\prime2}\right)  ^{\frac{3}{2}}}-\frac
{z_{1}f^{\prime\prime2}}{\left(  1+f^{\prime2}\right)  ^{3}}\\
&  -\frac{f^{\prime}}{r_{1}\sqrt{1+f^{\prime2}}}-\frac{z_{1}f^{\prime2}}%
{r_{1}^{2}\left(  1+f^{\prime2}\right)  }\\
&  +O\left(  f^{\prime\prime3}\right)  +O\left(  \frac{f^{\prime3}}{r_{1}^{3}%
}\right)  .
\end{align*}
Using this expansion, we find
\begin{align*}
&  \left(  \frac{\partial_{z_{1}}A}{2A}+r^{-1}\partial_{r}z_{1}\right)
H^{\prime}-\left[  \left(  \frac{\partial_{z_{1}}A}{2A}+r^{-1}\partial
_{r}z_{1}\right)  H^{\prime}\right]  _{1}^{\Vert}\\
&  =-\left(  z_{1}-h\right)  \left(  \frac{f^{\prime\prime2}}{\left(
1+f^{\prime2}\right)  ^{3}}+\frac{f^{\prime2}}{r_{1}^{2}\left(  1+f^{\prime
2}\right)  }\right)  H^{\prime}\\
&  +O\left(  f^{\prime\prime3}\right)  +O\left(  \frac{f^{\prime3}}{r_{1}^{3}%
}\right)  +O\left(  e^{-\sqrt{2}D}\right)  .
\end{align*}
The rest of the terms always contains small order terms of $h^{\prime}$ or $h$
and the conclusion of the lemma readily follows.
\end{proof}

The next result gives us an estimate for the main interaction term between
$\mathcal{H}_{1}$ and $\mathcal{H}_{1}^{s}$ appeared in $E\left(  \bar
{u}\right)  .$ Let
\[
\mathbf{c}_{1}=3\sqrt{2}\int_{\mathbb{R}}H^{\prime2}\left(  s\right)
e^{\sqrt{2}s}ds.
\]

\begin{lemma}
\label{interaction}For $r_{1}>r_{0},$%
\[
\int_{\mathbb{R}}X^{\ast}\left[  3\eta^{+}H_{1}^{\prime}\left(  \mathcal{H}%
_{1}^{s}+1\right)  \left(  \mathcal{H}_{1}^{2}-1\right)  \right]
dz_{1}=-\mathbf{c}_{1}\left(  1+o\left(  1\right)  \right)  e^{-\sqrt
{2}D\left(  r_{1}\right)  }.
\]

\end{lemma}

\begin{proof}
Since $H^{\prime\prime}=H^{3}-H,$ $H_{1}^{2}-1=-\sqrt{2}H_{1}^{\prime},$ it
follows that
\begin{align*}
&  \int_{\mathbb{R}}X^{\ast}\left[  3\eta^{+}H_{1}^{\prime}\left(
\mathcal{H}_{1}^{s}+1\right)  \left(  \mathcal{H}_{1}^{2}-1\right)  \right]
dz_{1}\\
&  =-3\sqrt{2}\int_{\mathbb{R}}X^{\ast}\left[  \eta^{+}H_{1}^{\prime2}\left(
\mathcal{H}_{1}^{s}+1\right)  \right]  dz_{1}+o\left(  e^{-\sqrt{2}D}\right)
\\
&  =-3\sqrt{2}\int_{\mathbb{R}}X^{\ast}\left(  \eta^{+}H_{1}^{\prime2}\right)
e^{-\sqrt{2}\left\vert z_{2}\right\vert }dz_{1}\\
&  +O\left(  \int_{\mathbb{R}}X^{\ast}\left(  \eta^{+}H_{1}^{\prime2}\right)
e^{-2\sqrt{2}\left\vert z_{2}\right\vert }dz_{1}\right)  .
\end{align*}
Since by definition $\left\vert z_{2}\right\vert =\mathcal{D}\left(
r,z\right)  -\left\vert z_{1}\right\vert ,$ we find
\begin{align*}
&  \int_{\mathbb{R}}X^{\ast}\left[  3\eta^{+}H_{1}^{\prime}\left(
\mathcal{H}_{1}^{s}+1\right)  \left(  \mathcal{H}_{1}^{2}-1\right)  \right]
dz_{1}\\
&  =-3\sqrt{2}\int_{\mathbb{R}}X^{\ast}\left(  \eta^{+}H_{1}^{\prime2}\right)
e^{-\sqrt{2}\left(  X^{\ast}\mathcal{D-}\left\vert z_{1}\right\vert \right)
}dz_{1}\\
&  +O\left(  \int_{\mathbb{R}}X^{\ast}\left(  \eta^{+}H_{1}^{\prime2}\right)
e^{-2\sqrt{2}\left(  X^{\ast}\mathcal{D-}\left\vert z_{1}\right\vert \right)
}dz_{1}\right) \\
&  =-3\sqrt{2}\int_{\mathbb{R}}X^{\ast}\left(  \eta^{+}H_{1}^{\prime2}\right)
e^{-\sqrt{2}\left(  X^{\ast}\mathcal{D-}\left\vert z_{1}\right\vert \right)
}dz_{1}\\
&  +o\left(  e^{-\sqrt{2}D\left(  r_{1}\right)  }\right) \\
&  =-3\sqrt{2}e^{-\sqrt{2}D\left(  r_{1}\right)  }\int_{\mathbb{R}}X^{\ast
}\left(  H_{1}^{\prime2}\right)  e^{\sqrt{2}\left\vert z_{1}\right\vert
}dz_{1}\\
&  +o\left(  e^{-\sqrt{2}D\left(  r_{1}\right)  }\right)  .
\end{align*}
In the last equality, we have used the fact that in $\mathbb{E}^{+},$
\[
\left\vert X^{\ast}\mathcal{D}\left(  r_{1},z_{1}\right)  -D\left(
r_{1}\right)  \right\vert =o\left(  z_{1}\right)  .
\]

\end{proof}

\begin{remark}
Clearly the function $D\left(  r_{1}\right)  $ is strictly less than
$2f\left(  r_{1}\right)  .$ However, due to the fact that $f^{\prime}$ is
small,
\[
2f\left(  r_{1}\right)  -D\left(  r_{1}\right)  =o\left(  f\left(
r_{1}\right)  \right)  .
\]
It is exactly this fact that makes it possible for us to analyze the equation
satisfied by $f.$
\end{remark}

The previous result analyze directly the error $E\left(  \bar{u}\right)  $
from the definition of $\bar{u}.$ In the next result, we will express the
projection of $E\left(  \bar{u}\right)  $ onto $\eta^{+}H_{1}^{\prime}$ in
terms of the function $\phi.$ The main equation we will use is equation
$\left(  \ref{fi}\right)  $ satisfied by $\phi.$

\begin{lemma}
\label{P1}For each $r_{1}>r_{0},$
\begin{align*}
&  \left\vert \int_{\mathbb{R}}X^{\ast}\left(  \eta^{+}H_{1}^{\prime}E\left(
\bar{u}\right)  \right)  dz_{1}\right\vert \\
&  =O\left(  \left\Vert \phi^{\ast}\left(  r_{1},\cdot\right)  \right\Vert
_{\infty}^{2}\right)  +O\left(  \left\Vert \phi^{\ast}\left(  r_{1}%
,\cdot\right)  \right\Vert _{\infty}e^{-\sqrt{2}f\left(  r_{1}\right)
}\right)  +O\left(  \left\Vert \phi^{\ast}\left(  r_{1},\cdot\right)
\right\Vert _{\infty}\left\Vert \partial_{r_{1}}\phi^{\ast}\left(  r_{1}%
,\cdot\right)  \right\Vert _{\infty}\right) \\
+  &  O\left(  \left\Vert \partial_{r_{1}}\phi^{\ast}\left(  r_{1}%
,\cdot\right)  \right\Vert _{\infty}e^{-\sqrt{2}f\left(  r_{1}\right)
}\right)  +O\left(  \left\Vert \partial_{r_{1}}\phi^{\ast}\left(  r_{1}%
,\cdot\right)  \right\Vert _{\infty}^{2}\right)  +O\left(  \left\Vert
\phi^{\ast}\left(  r_{1},\cdot\right)  \right\Vert _{\infty}\left\Vert
\partial_{r_{1}}^{2}\phi^{\ast}\left(  r_{1},\cdot\right)  \right\Vert
_{\infty}\right) \\
+  &  O\left(  f^{\prime\prime}\left\Vert \partial_{z_{1}}\phi^{\ast}\left(
r_{1},\cdot\right)  \right\Vert _{\infty}\right)  +O\left(  f^{\prime\prime
}\left\Vert \partial_{r_{1}}\phi^{\ast}\left(  r_{1},\cdot\right)  \right\Vert
_{\infty}\right)  +O\left(  f^{\prime\prime\prime}\left\Vert \partial_{r_{1}%
}\phi^{\ast}\left(  r_{1},\cdot\right)  \right\Vert _{\infty}\right) \\
+  &  O\left(  r^{-1}\left\Vert \partial_{r_{1}}\phi^{\ast}\left(  r_{1}%
,\cdot\right)  \right\Vert _{\infty}\right)  +O\left(  r^{-1}f^{\prime
}\left\Vert \partial_{z_{1}}\phi^{\ast}\left(  r_{1},\cdot\right)  \right\Vert
_{\infty}\right)  .
\end{align*}

\end{lemma}

\begin{proof}
Multiplying both sides of $\left(  \ref{fi}\right)  $ with $\eta^{+}%
H_{1}^{\prime}\,\ $and integrating in $\mathbb{R}$, we get
\[
\int_{\mathbb{R}}X^{\ast}\left(  \eta^{+}H_{1}^{\prime}E\left(  \bar
{u}\right)  \right)  dz_{1}=\int_{\mathbb{R}}X^{\ast}\left(  \eta^{+}%
H_{1}^{\prime}L\phi\right)  dz_{1}-\int_{\mathbb{R}}X^{\ast}\left(  \eta
^{+}H_{1}^{\prime}P\left(  \phi\right)  \right)  dz_{1}.
\]

Let us calculate the term $\int_{\mathbb{R}}X^{\ast}\left(  \eta^{+}%
H_{1}^{\prime}L\phi\right)  dz_{1}.$ By formula $\left(  \ref{Laplacian}%
\right)  $ of Laplacian in the Fermi coordinate,
\begin{align*}
&  \int_{\mathbb{R}}X^{\ast}\left(  \eta^{+}H_{1}^{\prime}L\phi\right)
dz_{1}\\
&  =-\int_{\mathbb{R}}X^{\ast}\left(  \eta^{+}H_{1}^{\prime}\right)  \left[
A^{-1}\partial_{r_{1}}^{2}+\frac{1}{2}\frac{\partial_{z_{1}}A}{A}%
\partial_{z_{1}}-\frac{1}{2}\frac{\partial_{r_{1}}A}{A^{2}}\partial_{r_{1}%
}\right]  \phi^{\ast}dz_{1}\\
&  -\int_{\mathbb{R}}X^{\ast}\left(  \eta^{+}H_{1}^{\prime}\right)
r^{-1}\left[  \partial_{r_{1}}\phi^{\ast}\partial_{r}r_{1}+\partial_{z_{1}%
}\phi^{\ast}\partial_{r}z_{1}\right]  dz_{1}\\
&  +\int_{\mathbb{R}}\left[  -X^{\ast}\left(  \eta^{+}H_{1}^{\prime}\right)
\partial_{z_{1}}^{2}\phi^{\ast}+X^{\ast}\left(  \eta^{+}H_{1}^{\prime}\left(
3\bar{u}^{2}-1\right)  \phi\right)  \right]  dz_{1}.
\end{align*}
We first estimate $\int_{\mathbb{R}}X^{\ast}\left(  \eta^{+}H_{1}^{\prime
}\right)  A^{-1}\partial_{r_{1}}^{2}\phi^{\ast}dz_{1}.$ Differentiating the
identity
\begin{equation}
\int_{\mathbb{R}}X^{\ast}\left(  \eta^{+}H_{1}^{\prime}\phi\right)  dz_{1}=0
\label{fiorth}%
\end{equation}
with respect to $r_{1},$ we obtain%
\begin{align*}
\int_{\mathbb{R}}X^{\ast}\left(  \eta^{+}H_{1}^{\prime}\right)  \partial
_{r_{1}}\phi^{\ast}dz_{1}  &  =-\int_{\mathbb{R}}\partial_{r_{1}}\left[
X^{\ast}\left(  \eta^{+}H_{1}^{\prime}\right)  \right]  \phi^{\ast}dz_{1}\\
&  =-\int_{\mathbb{R}}\partial_{r_{1}}\left(  \eta^{+\ast}\right)  H^{\prime
}\phi^{\ast}dz_{1}+\int_{\mathbb{R}}\eta^{+\ast}H^{\prime\prime}h^{\prime}%
\phi^{\ast}dz_{1}.
\end{align*}
Here $H^{\prime}$ and $H^{\prime\prime}$ are evaluated at $z_{1}-h\left(
r_{1}\right)  .$ Therefore we could apply Lemma \ref{h} and obtain
\begin{align*}
&  \int_{\mathbb{R}}X^{\ast}\left(  \eta^{+}H_{1}^{\prime}\right)
\partial_{r_{1}}\phi^{\ast}\\
&  =O\left(  \left\Vert \phi^{\ast}\left(  r_{1},\cdot\right)  \right\Vert
_{\infty}e^{-\sqrt{2}f\left(  r_{1}\right)  }\right)  +O\left(  \left\Vert
\phi^{\ast}\left(  r_{1},\cdot\right)  \right\Vert _{\infty}\left\Vert
\partial_{r_{1}}\phi^{\ast}\left(  r_{1},\cdot\right)  \right\Vert _{\infty
}\right)  .
\end{align*}
Similarly, differentiating $\left(  \ref{fiorth}\right)  $ with respect to
$r_{1}$ twice, and using the fact that $A^{-1}-\frac{1}{1+f^{\prime2}%
}=O\left(  f^{\prime\prime}\right)  ,$ we get%
\begin{align*}
&  \int_{\mathbb{R}}X^{\ast}\left(  \eta^{+}H_{1}^{\prime}\right)
A^{-1}\partial_{r_{1}}^{2}\phi^{\ast}dz_{1}\\
&  =\frac{1}{1+f^{\prime2}}\int_{\mathbb{R}}X^{\ast}\left(  \eta^{+}%
H_{1}^{\prime}\right)  \partial_{r_{1}}^{2}\phi^{\ast}dz_{1}+O\left(
\left\Vert \partial_{r_{1}}^{2}\phi^{\ast}\left(  r_{1},\cdot\right)
\right\Vert _{\infty}\left\vert f^{\prime\prime}\left(  r_{1}\right)
\right\vert \right) \\
&  =-\frac{2}{1+f^{\prime2}}\int_{\mathbb{R}}\partial_{r_{1}}\left[  X^{\ast
}\left(  \eta^{+}H_{1}^{\prime}\right)  \right]  \partial_{r_{1}}\phi^{\ast
}-\frac{1}{1+f^{\prime2}}\int_{\mathbb{R}}\partial_{r_{1}}^{2}\left[  X^{\ast
}\left(  \eta^{+}H_{1}^{\prime}\right)  \right]  \phi^{\ast}\\
&  +O\left(  \left\Vert \partial_{r_{1}}^{2}\phi^{\ast}\left(  r_{1}%
,\cdot\right)  \right\Vert _{\infty}\left\vert f^{\prime\prime}\left(
r_{1}\right)  \right\vert \right) \\
&  =O\left(  \left\Vert \partial_{r_{1}}\phi^{\ast}\left(  r_{1},\cdot\right)
\right\Vert _{\infty}e^{-\sqrt{2}f\left(  r_{1}\right)  }\right)  +O\left(
\left\Vert \partial_{r_{1}}\phi^{\ast}\left(  r_{1},\cdot\right)  \right\Vert
_{\infty}^{2}\right) \\
&  +O\left(  \left\Vert \phi^{\ast}\left(  r_{1},\cdot\right)  \right\Vert
_{\infty}e^{-\sqrt{2}f\left(  r_{1}\right)  }\right)  +O\left(  \left\Vert
\phi^{\ast}\left(  r_{1},\cdot\right)  \right\Vert _{\infty}\left\Vert
\partial_{r_{1}}\phi^{\ast}\left(  r_{1},\cdot\right)  \right\Vert _{\infty
}\right) \\
&  +O\left(  \left\Vert \phi^{\ast}\left(  r_{1},\cdot\right)  \right\Vert
_{\infty}\left\Vert \partial_{r_{1}}^{2}\phi^{\ast}\left(  r_{1},\cdot\right)
\right\Vert _{\infty}\right)  .
\end{align*}

Next, using $\partial_{z_{1}}A=O\left(  f^{\prime\prime}\right)  ,$ we infer
\[
\int_{\mathbb{R}}\frac{\partial_{z_{1}}A}{A}\partial_{z_{1}}\phi^{\ast}%
X^{\ast}\left(  \eta^{+}H_{1}^{\prime}\right)  dz_{1}=O\left(  f^{\prime
\prime}\left\Vert \partial_{z_{1}}\phi^{\ast}\left(  r_{1},\cdot\right)
\right\Vert _{\infty}\right)  .
\]
Observe that $\partial_{r_{1}}A=O\left(  \left\vert f^{\prime\prime
}\right\vert +\left\vert f^{\prime\prime\prime}\right\vert \right)  ,$ thus
\[
\int_{\mathbb{R}}\frac{\partial_{r_{1}}A}{A^{2}}\partial_{r_{1}}\phi^{\ast
}X^{\ast}\left(  \eta^{+}H_{1}^{\prime}\right)  dz_{1}=O\left(  \left\vert
f^{\prime\prime}\right\vert \left\Vert \partial_{r_{1}}\phi^{\ast}\left(
r_{1},\cdot\right)  \right\Vert _{\infty}\right)  +O\left(  \left\vert
f^{\prime\prime\prime}\right\vert \left\Vert \partial_{r_{1}}\phi^{\ast
}\left(  r_{1},\cdot\right)  \right\Vert _{\infty}\right)  .
\]
We also have
\begin{align*}
&  \int_{\mathbb{R}}r^{-1}X^{\ast}\left(  \eta^{+}H_{1}^{\prime}\right)
\left[  \partial_{r_{1}}\phi^{\ast}\partial_{r}r_{1}+\partial_{z_{1}}%
\phi^{\ast}\partial_{r}z_{1}\right] \\
&  =O\left(  r^{-1}\left\Vert \partial_{r_{1}}\phi^{\ast}\left(  r_{1}%
,\cdot\right)  \right\Vert _{\infty}\right)  +O\left(  r^{-1}\left\vert
f^{\prime}\right\vert \left\Vert \partial_{z_{1}}\phi^{\ast}\left(
r_{1},\cdot\right)  \right\Vert _{\infty}\right)  .
\end{align*}
Another term we need to estimate is
\begin{align*}
&  \int_{\mathbb{R}}X^{\ast}\left(  \eta^{+}H_{1}^{\prime}\right)
\partial_{z_{1}}^{2}\phi^{\ast}+\int_{\mathbb{R}}X^{\ast}\left(  \eta^{+}%
H_{1}^{\prime}\left(  3\bar{u}^{2}-1\right)  \phi\right) \\
&  =\int_{\mathbb{R}}X^{\ast}\left(  \eta^{+}H_{1}^{\prime}\right)
\partial_{z_{1}}^{2}\phi^{\ast}+\int_{\mathbb{R}}X^{\ast}\left(  \eta^{+}%
H_{1}^{\prime}\left(  3H_{1}^{2}-1\right)  \phi\right) \\
&  +3\int_{\mathbb{R}}X^{\ast}\left(  \eta^{+}H_{1}^{\prime}\left(  \bar
{u}^{2}-H_{1}^{2}\right)  \phi\right)  .
\end{align*}
To handle them, we integrate by parts for the first term, and for the last
term we use the fact
\[
\eta^{+}\left(  \bar{u}^{2}-H_{1}^{2}\right)  H_{1}^{\prime}=O\left(  \eta
^{+}\left(  \mathcal{H}_{1}^{s}+1\right)  H_{1}^{\prime}\right)  =O\left(
e^{-D}\right)  .
\]
We conclude that
\begin{align*}
&  \int_{\mathbb{R}}X^{\ast}\left(  \eta^{+}H_{1}^{\prime}\right)
\partial_{z_{1}}^{2}\phi^{\ast}+\int_{\mathbb{R}}X^{\ast}\left(  \eta^{+}%
H_{1}^{\prime}\left(  3\bar{u}^{2}-1\right)  \phi\right) \\
&  =O\left(  \left\Vert \phi^{\ast}\left(  r_{1},\cdot\right)  \right\Vert
_{\infty}e^{-\sqrt{2}f}\right)  +O\left(  \left\Vert \phi^{\ast}\left(
r_{1},\cdot\right)  \right\Vert _{\infty}\left\Vert \partial_{r_{1}}\phi
^{\ast}\left(  r_{1},\cdot\right)  \right\Vert _{\infty}\right)  .
\end{align*}

Finally, since $P\left(  \phi\right)  $ is higher order term of $\phi,$ the
estimate of the term $\int_{\mathbb{R}}X^{\ast}\left(  \eta^{+}H_{1}^{\prime
}P\left(  \phi\right)  \right)  dz_{1}$ is trivial. Combining all these
estimates, we get the desired result.
\end{proof}

With the error term $E\left(  \bar{u}\right)  $ being understood, we would
like to recall some properties of the operator $L.$ The next result are
essentially proved in \cite{MR2557944}, although here we need to do slight
modifications due to the presence of boundary terms. For each $r_{1}\geq
r_{0},$ let
\[
K_{r_{1}}:=\left\{  \left(  r,z\right)  \in E,-f\left(  r_{1}\right)
+\frac{1}{f^{\prime}\left(  r_{1}\right)  }\left(  r-r_{1}\right)  <z<f\left(
r_{1}\right)  -\frac{1}{f^{\prime}\left(  r_{1}\right)  }\left(
r-r_{1}\right)  \right\}  ,
\]
and $J_{r_{1}}:=E\backslash K_{r_{1}}$. Note that $J_{r_{0}}$ is simply the
set $\mathbb{\hat{E}}$ introduced before.

\begin{lemma}
\label{L}Suppose $\varphi$ is a solution of the equation
\[
L\varphi=\Theta\text{ in }J_{r_{1}}.
\]
and $\varphi\left(  r,z\right)  =\varphi\left(  r,-z\right)  ,$ $\varphi
_{1}^{\Vert}=0.$ Then
\[
\left\vert \varphi^{\ast}\left(  s,z\right)  \right\vert \leq C\left\Vert
\varphi\right\Vert _{L^{\infty}\left(  \partial J_{r_{1}}\right)  }e^{r_{1}%
-s}+\left\Vert \Theta\right\Vert _{L^{\infty}\left(  J_{r_{1}}\right)
},\text{ for each }s>r_{1}.
\]

\end{lemma}

\begin{proof}
Let $\theta$ be a cutoff function such that
\[
\theta\left(  Z\right)  =\left\{
\begin{array}
[c]{l}%
1,\text{ if }Z\in J_{r_{1}}\text{ and }dist\left(  Z,\partial J_{r_{1}%
}\right)  >1,\\
0,\,Z\notin J_{r_{1}}\text{ }.
\end{array}
\right.
\]
Consider the function $\bar{\varphi}:=\theta\varphi.$ Then $\bar{\varphi}$
satisfies
\begin{equation}
L\bar{\varphi}=\theta\Theta+\Delta\theta\varphi+2\nabla\theta\nabla\varphi.
\label{fibar}%
\end{equation}
Modify the function $\bar{u}$ in $E\backslash J_{r_{1}}$ if necessary, we
could get a function, still denoted by $\bar{u},$ whose nodal curve are almost
horizontal and around its nodal curve, it looks like the heteroclinic function
$H.$ With slight abuse of notation, the corresponding linearized operator will
be still write as $L.$ We could also assume that $\bar{\varphi}$ satisfies the
orthogonality condition.

Denote $\Theta_{1}=\theta\Theta,$ $\Theta_{2}=\Delta\theta\varphi
+2\nabla\theta\nabla\varphi.$ For $i=1,2,$ consider the equation
\begin{equation}
L\bar{\varphi}_{i}=\Theta_{i}+k_{i}\left(  r_{1}\right)  \eta^{+}H_{1}%
^{\prime}+\left(  k_{i}\eta^{+}H_{1}^{\prime}\right)  ^{s} \label{fii}%
\end{equation}
where $\bar{\varphi}_{i},\bar{k}_{i}$ are unknown functions and we require
$\int_{\mathbb{R}}X^{\ast}\left(  \eta^{+}H_{i}^{\prime}\bar{\varphi}%
_{i}\right)  dz_{1}=0,i=1,2.$ By the results of \cite{MR2557944}, one could
find solution $\left(  \bar{\varphi}_{i},\bar{k}_{i}\right)  $ to this
equation. This pair of solution is indeed unique and hence we have the
following estimate for $\bar{\varphi}_{1}$:%
\[
\left\Vert \bar{\varphi}_{1}\right\Vert _{L^{\infty}}\leq C\left\Vert
\Theta_{1}\right\Vert _{L^{\infty}}.
\]
For the function $\bar{\varphi}_{2},$ since $\Theta_{2}$ has better decay
property than $\Theta_{1},$ we could estimate
\[
\left\vert \bar{\varphi}_{2}\left(  r,z\right)  \right\vert \leq C\left\Vert
\Theta_{2}\right\Vert _{L^{\infty}}e^{r_{1}-r}.
\]
Notice that
\[
L\left(  \bar{\varphi}_{1}+\bar{\varphi}_{2}\right)  =\Theta_{1}+\Theta
_{2}+\left(  k_{1}+k_{2}\right)  \eta^{+}H_{1}^{\prime}+\left[  \left(
k_{1}+k_{2}\right)  \eta^{+}H_{1}^{\prime}\right]  ^{s}.
\]
But the equation
\[
L\bar{\varphi}_{3}=\Theta_{1}+\Theta_{2}+k_{3}\eta^{+}H_{1}^{\prime}+\left[
k_{3}\eta^{+}H_{1}^{\prime}\right]  ^{s}%
\]
also has a unique pair of solution $\left(  \bar{\varphi}_{3},k_{3}\right)  $
if we require $\int_{\mathbb{R}}X^{\ast}\left(  \eta^{+}H_{1}^{\prime}%
\bar{\varphi}_{3}\right)  =0.$ This implies that $\bar{\varphi}=\bar{\varphi
}_{1}+\bar{\varphi}_{2},$ the conclusion of this proposition readily follows.
\end{proof}

Once we have an estimate in $L^{\infty}$ norm, we could get a corresponding
Holder norm estimate by Schauder's elliptic estimate. Now we proceed to the
proof of Proposition \ref{P3}.

\begin{proof}
[Proof of Proposition \ref{P3}]Recall that
\begin{align*}
&  \int_{\mathbb{R}}X^{\ast}\left(  \eta^{+}H_{1}^{\prime}E\left(  \bar
{u}\right)  \right) \\
&  =\int_{\mathbb{R}}X^{\ast}\left(  \eta^{+}H_{1}^{\prime}E\left(
\mathcal{H}_{1}\right)  \right)  +\int_{\mathbb{R}}X^{\ast}\left(  \eta
^{+}H_{1}^{\prime}3\left(  \mathcal{H}_{1}^{s}+1\right)  \left(
\mathcal{H}_{1}^{2}-1\right)  \right) \\
&  +\int_{\mathbb{R}}X^{\ast}\left(  \eta^{+}H_{1}^{\prime}E\left(
\mathcal{H}_{1}^{s}\right)  \right)  +\int_{\mathbb{R}}X^{\ast}\left(
\eta^{+}H_{1}^{\prime}3\left(  \mathcal{H}_{1}+1\right)  \left(
\mathcal{H}_{1}^{s}+1\right)  ^{2}\right)  .
\end{align*}
We have already analyzed the first two terms. For the third term, slightly
abuse the notation, we have
\begin{align*}
\int_{\mathbb{R}}X^{\ast}\left(  \eta^{+}H_{1}^{\prime}E\left(  \mathcal{H}%
_{1}^{s}\right)  \right)  dz_{1}  &  =\int_{\mathbb{R}}X^{\ast}\left(
\eta^{+}H_{1}^{\prime}\right)  A^{-1}\left(  -h^{\prime\prime}H^{\prime
}+h^{\prime2}H^{\prime\prime}\right)  dz_{1}\\
&  +\int_{\mathbb{R}}X^{\ast}\left(  \eta^{+}H_{1}^{\prime}\right)  \left(
\frac{\partial_{z_{2}}A}{2A}+r^{-1}\partial_{r}z_{2}\right)  H^{\prime}%
dz_{1}\\
&  +\int_{\mathbb{R}}X^{\ast}\left(  \eta^{+}H_{1}^{\prime}\right)  \left(
\frac{\partial_{r_{2}}A}{2A^{2}}-r^{-1}\partial_{r}r_{2}\right)  \left(
h^{\prime}H^{\prime}\right)  dz_{1}.
\end{align*}
It should be pointed out that here $A,h$ are evaluated at $r_{2}$ and
$H^{\prime}$ is evaluated at $z_{2}-h\left(  r_{2}\right)  ,$ and $h^{\prime}$
means the derivative with respect to $r_{2}$($\left(  r_{2},z_{2}\right)  $
being the Fermi coordinate with respect to $z=-f\left(  r\right)  $)$,$ rather
than $r_{1}.$ This indeed makes the analysis a little more complicated. We
have
\begin{align*}
&  \int_{\mathbb{R}}X^{\ast}\left(  \eta^{+}H_{1}^{\prime}\right)
A^{-1}\left(  -h^{\prime\prime}H^{\prime}+h^{\prime2}H^{\prime\prime}\right)
dz_{1}\\
&  +\int_{\mathbb{R}}X^{\ast}\left(  \eta^{+}H_{1}^{\prime}\right)  \left(
\frac{\partial_{r_{2}}A}{2A^{2}}-r^{-1}\partial_{r}r_{2}\right)  \left(
h^{\prime}H^{\prime}\right)  dz_{1}\\
&  =o\left(  e^{-\sqrt{2}D}\right)  .
\end{align*}
For the same reason,
\[
\int_{\mathbb{R}}X^{\ast}\left(  \eta^{+}H_{1}^{\prime}\right)  \left(
\frac{\partial_{z_{2}}A}{2A}+r^{-1}\partial_{r}z_{2}\right)  H^{\prime}%
dz_{1}=o\left(  e^{-D}\right)  .
\]
It follows that
\[
\int_{\mathbb{R}}X^{\ast}\left(  \eta^{+}H_{1}^{\prime}E\left(  \mathcal{H}%
_{1}^{s}\right)  \right)  dz_{1}=o\left(  e^{-\sqrt{2}D}\right)  .
\]
Similarly,
\[
\int_{\mathbb{R}}X^{\ast}\left(  \eta^{+}H_{1}^{\prime}3\left(  \mathcal{H}%
_{1}+1\right)  \left(  \mathcal{H}_{1}^{s}+1\right)  ^{2}\right)  =o\left(
e^{-\sqrt{2}D}\right)  .
\]
It then follows from Lemma $\ref{P2}$ and Lemma $\ref{interaction}$ that
\begin{align*}
&  \int_{\mathbb{R}}X^{\ast}\left(  \eta^{+}H_{1}^{\prime}E\left(  \bar
{u}\right)  \right) \\
&  =-\left(  1+o\left(  1\right)  \right)  \left[  \frac{\mathbf{c}%
_{0}h^{\prime\prime}}{1+f^{\prime2}}+\frac{\mathbf{c}_{0}f^{\prime\prime}%
}{\left(  1+f^{\prime2}\right)  ^{\frac{3}{2}}}+\frac{\mathbf{c}_{0}f^{\prime
}}{r_{1}\sqrt{1+f^{\prime2}}}\right] \\
&  -\mathbf{c}_{1}\left(  1+o\left(  1\right)  \right)  e^{-\sqrt{2}D\left(
r_{1}\right)  }+O\left(  h^{\prime2}\right)  +O\left(  h^{\prime}%
f^{\prime\prime}\right) \\
&  +O\left(  h^{\prime}f^{\prime\prime\prime}\right)  +O\left(  h^{\prime
}r^{-1}\right)  +O\left(  h^{\prime\prime}f^{\prime\prime}\right)  .
\end{align*}
This together with Lemma \ref{P1} leads to the following equation:
\begin{align}
&  \frac{\mathbf{c}_{0}f^{\prime\prime}}{\left(  1+f^{\prime2}\right)
^{\frac{3}{2}}}+\frac{\mathbf{c}_{0}h^{\prime\prime}}{1+f^{\prime2}}%
+\frac{\mathbf{c}_{0}f^{\prime}}{r_{1}\sqrt{1+f^{\prime2}}}-\left(  1+o\left(
1\right)  \right)  \mathbf{c}_{1}e^{-\sqrt{2}D}\label{f2}\\
&  =O\left(  h^{\prime\prime}f^{\prime\prime}\right)  +O\left(  h^{\prime
2}\right)  +O\left(  h^{\prime}f^{\prime\prime}\right)  +O\left(  h^{\prime
}r^{-1}\right)  +O\left(  h^{\prime}f^{\prime\prime\prime}\right) \nonumber\\
&  +O\left(  \left\Vert \phi^{\ast}\left(  r_{1},\cdot\right)  \right\Vert
_{\infty}^{2}\right)  +O\left(  \left\Vert \phi^{\ast}\left(  r_{1}%
,\cdot\right)  \right\Vert _{\infty}e^{-\sqrt{2}f\left(  r_{1}\right)
}\right)  +O\left(  \left\Vert \phi^{\ast}\left(  r_{1},\cdot\right)
\right\Vert _{\infty}\left\Vert \partial_{r_{1}}\phi^{\ast}\left(  r_{1}%
,\cdot\right)  \right\Vert _{\infty}\right) \nonumber\\
&  +O\left(  \left\Vert \partial_{r_{1}}\phi^{\ast}\left(  r_{1},\cdot\right)
\right\Vert _{\infty}e^{-\sqrt{2}f\left(  r_{1}\right)  }\right)  +O\left(
\left\Vert \partial_{r_{1}}\phi^{\ast}\left(  r_{1},\cdot\right)  \right\Vert
_{\infty}^{2}\right)  +O\left(  \left\Vert \phi^{\ast}\left(  r_{1}%
,\cdot\right)  \right\Vert _{\infty}\left\Vert \partial_{r_{1}}^{2}\phi^{\ast
}\left(  r_{1},\cdot\right)  \right\Vert _{\infty}\right) \nonumber\\
&  +O\left(  f^{\prime\prime}\left\Vert \partial_{z_{1}}\phi^{\ast}\left(
r_{1},\cdot\right)  \right\Vert _{\infty}\right)  +O\left(  f^{\prime\prime
}\left\Vert \partial_{r_{1}}\phi^{\ast}\left(  r_{1},\cdot\right)  \right\Vert
_{\infty}\right)  +O\left(  f^{\prime\prime\prime}\left\Vert \partial_{r_{1}%
}\phi^{\ast}\left(  r_{1},\cdot\right)  \right\Vert _{\infty}\right)
\nonumber\\
&  +O\left(  r^{-1}\left\Vert \partial_{r_{1}}\phi^{\ast}\left(  r_{1}%
,\cdot\right)  \right\Vert _{\infty}\right)  +O\left(  r^{-1}f^{\prime
}\left\Vert \partial_{z_{1}}\phi^{\ast}\left(  r_{1},\cdot\right)  \right\Vert
_{\infty}\right)  .\nonumber
\end{align}

On the other hand,
\[
L\phi=\left[  E\left(  \bar{u}\right)  \right]  ^{\bot}+\left[  E\left(
\bar{u}\right)  \right]  ^{\Vert}+P\left(  \phi\right)  .
\]
By Lemma \ref{P1}, $\left[  E\left(  \bar{u}\right)  \right]  ^{\Vert
}+P\left(  \phi\right)  $ is small compared to $\phi$. In view of Lemma
\ref{L}, we need to estimate $\left[  E\left(  \bar{u}\right)  \right]
^{\bot}.$

By $\left(  \ref{Eu}\right)  ,$
\begin{align*}
\left[  E\left(  \bar{u}\right)  \right]  ^{\bot}  &  =\left[  E\left(
\mathcal{H}_{1}\right)  \right]  ^{\bot}+\left[  3\left(  \mathcal{H}_{1}%
^{s}+1\right)  \left(  \mathcal{H}_{1}^{2}-1\right)  \right]  ^{\bot}\\
&  +\left[  E\left(  \mathcal{H}_{1}^{s}\right)  \right]  ^{\bot}+\left[
3\left(  \mathcal{H}_{1}+1\right)  \left(  \mathcal{H}_{1}^{s}+1\right)
^{2}\right]  ^{\bot}.
\end{align*}
Let us analyze the term $\left[  E\left(  \mathcal{H}_{1}^{s}\right)  \right]
^{\bot}.$ We know that
\begin{align*}
E\left(  \mathcal{H}_{1}^{s}\right)   &  =A^{-1}\left(  -h^{\prime\prime
}H^{\prime}+h^{\prime}H^{\prime\prime}\right)  +\left(  \frac{\partial_{z_{2}%
}A}{2A}+r^{-1}\partial_{r}z_{2}\right)  H^{\prime}\\
&  +\left(  \frac{\partial_{r_{2}}A}{2A^{2}}-r^{-1}\partial_{r}r_{2}\right)
h^{\prime}H^{\prime}.
\end{align*}
Hence in the upper plane, due to the presence of $H^{\prime}$ in each term, we
find that directly that
\begin{equation}
\left[  E\left(  \mathcal{H}_{1}^{s}\right)  \right]  ^{\bot}=O\left(
e^{-\sqrt{2}f}\right)  . \label{Es}%
\end{equation}
Note that this is not an optimal estimate and each term in $E\left(
\mathcal{H}_{1}^{s}\right)  $ is mulitplied by $h^{\prime},h^{\prime\prime
},f^{\prime\prime}$ or $r^{-1}f^{\prime}.$ However, these terms are in general
not evaluated at $r_{1}.$ One also has%
\[
\left[  3\left(  \mathcal{H}_{1}+1\right)  \left(  \mathcal{H}_{1}%
^{s}+1\right)  ^{2}\right]  ^{\bot}=o\left(  e^{-\sqrt{2}D}\right)  .
\]

The above estimates combined with Lemma \ref{P} and \ref{interaction} leads
to
\begin{align}
\left[  E\left(  \bar{u}\right)  \right]  ^{\bot}  &  =\left[  E\left(
\mathcal{H}_{1}^{s}\right)  \right]  ^{\bot}+O\left(  f^{\prime\prime
2}\right)  +O\left(  r^{-2}f^{\prime2}\right)  +O\left(  h^{\prime\prime
}f^{\prime\prime}\right)  +O\left(  h^{\prime}r^{-1}\right) \label{Euo}\\
&  +O\left(  h^{\prime\prime2}\right)  +O\left(  h^{\prime2}\right)  +O\left(
h^{\prime}f^{\prime\prime}\right)  +O\left(  h^{\prime}f^{\prime\prime\prime
}\right)  +O\left(  e^{-\sqrt{2}D}\right)  .\nonumber
\end{align}
Notice that there is a term $f^{\prime\prime2}$ appeared in the right hand
side. \ Insert $\left(  \ref{f2}\right)  $ into $\left(  \ref{Euo}\right)  $
we get
\begin{align}
\left\vert \left(  E\left(  \bar{u}\right)  \right)  ^{\bot}\right\vert  &
=\left[  E\left(  \mathcal{H}_{1}^{s}\right)  \right]  ^{\bot}+O\left(
r^{-2}f^{\prime2}\right)  +O\left(  h^{\prime}r^{-1}\right) \label{Eorth}\\
&  +O\left(  h^{\prime2}\right)  +O\left(  h^{\prime}f^{\prime\prime\prime
}\right)  +O\left(  e^{-\sqrt{2}D}\right) \nonumber\\
&  +O\left(  h^{\prime}h^{\prime\prime}\right)  +O\left(  hf^{\prime
\prime\prime}\right)  +O\left(  h^{\prime\prime2}\right) \nonumber\\
&  +O\left(  \left\Vert \phi^{\ast}\left(  r_{1},\cdot\right)  \right\Vert
_{\infty}^{2}\right)  +O\left(  \left\Vert \partial_{r_{1}}\phi^{\ast}\left(
r_{1},\cdot\right)  \right\Vert _{\infty}^{2}\right) \nonumber\\
&  +O\left(  \left\Vert \partial_{z_{1}}\phi^{\ast}\left(  r_{1},\cdot\right)
\right\Vert _{\infty}^{2}\right)  .\nonumber
\end{align}
Note that the term involving $h$ and its derivatives are essentially
controlled by $\phi.$ Using similar arguments as that of Lemma \ref{L}, we
find that for $s>r_{1},$
\begin{equation}
\left\vert \phi^{\ast}\left(  s,z_{1}\right)  \right\vert \leq C\left\Vert
e^{-D}\right\Vert _{L^{\infty}\left(  \left(  r_{1},+\infty\right)  \right)
}+C\left\Vert r^{-2}f^{\prime2}\right\Vert _{L^{\infty}\left(  \left(
r_{1}+\infty\right)  \right)  }+\left\Vert \phi\right\Vert _{L^{\infty}\left(
\partial J_{r_{1}}\right)  }e^{r_{1}-s}. \label{Esfi}%
\end{equation}
We remark that there is the term $\left[  E\left(  \mathcal{H}_{1}^{s}\right)
\right]  ^{\bot}$ appearing in $\left(  \ref{Eorth}\right)  .$ Initially, by
$\left(  \ref{Es}\right)  $, it is only of the order $O\left(  e^{-f}\right)
.$ But we would using a bootstrap argument to show that this term is indeed of
the order $O\left(  e^{-D}\right)  .$

The estimate $\left(  \ref{Esfi}\right)  $ could be applied repeatedly. This
yields that for some large but fixed constant $l$,
\begin{align*}
\phi\left(  r,z\right)   &  \leq Ce^{-D\left(  r-l\right)  }+\frac{C}{\left(
r-l\right)  ^{2}}+C\left\Vert \phi\right\Vert _{L^{\infty}\left(  \partial
J_{r-l}\right)  }e^{-l}\\
&  \leq Ce^{-D\left(  r-l\right)  }+\frac{C}{\left(  r-l\right)  ^{2}}\\
&  +C^{2}\left[  e^{-D\left(  r-2l\right)  }+\frac{1}{\left(  r-2l\right)
^{2}}+\left\Vert \phi\right\Vert _{L^{\infty}\left(  \mathbb{\partial}%
J_{r-2l}\right)  }e^{-l}\right]  e^{-l}\\
&  \leq Ce^{-D\left(  r\right)  }+\frac{C}{r^{2}}+C\left\Vert \phi\right\Vert
_{L^{\infty}\left(  \mathbb{\partial}J_{r_{0}}\right)  }e^{-r}.
\end{align*}
The last inequality follows from
\[
\sum_{j=1}^{k}\frac{1}{\left(  r-jl\right)  ^{2}}e^{-jl}\leq\frac{C}{r^{2}},
\]
and by Lemma \ref{slope}, $f^{\prime}$ is small, which together with
$D=\left(  2+o\left(  1\right)  \right)  f$ implies that
\[
\sum_{j=1}^{k}e^{-D\left(  r-jl\right)  }e^{-jl}\leq Ce^{-D\left(  r\right)
}.
\]
This concludes the proof.
\end{proof}

Up to now, we have assumed $\left(  \ref{Assum}\right)  .$ But a priori, we
don't know whether this is fulfilled or not. It the sequel, we sketch the
proof without this assumption. The basic principle is, using  equation
$\left(  \ref{f2}\right)  $ and similar arguments above, one could prove that
assumption $\left(  \ref{Assum}\right)  $ indeed holds.

\begin{proof}
[Proof of Proposition \ref{P3} without Assumption (A)]The main reason that we
introduce the Assumption $\left(  \ref{Assum}\right)  $ is that in the the
error $E\left(  \bar{u}\right)  $ is related to $e^{-2\sqrt{2}\bar{d}}$ and we
could control it by $e^{-D}.$ Now without Assumption $\left(  \ref{Assum}%
\right)  $, to handle the term $O\left(  e^{-2\sqrt{2}\bar{d}}\right)  ,$ we
would like to use the following basic geometrical fact concerning the Fermi
coordinate (see \cite{MR3148064} for a proof): Using the fact that $f^{\prime
}\left(  r_{1}\right)  $ is small, there exists a point $\bar{r}$($\bar{r}$
may be equal to $r_{1}$) with
\[
\left\vert \bar{r}-r_{1}\right\vert =o\left(  \min\left\{  d_{1}\left(
r_{1}\right)  ,d_{2}\left(  r_{1}\right)  \right\}  \right)
\]
such that
\begin{equation}
\min\left\{  d_{1}\left(  r_{1}\right)  ,d_{2}\left(  r_{1}\right)  \right\}
\geq\frac{1}{\left\vert f^{\prime\prime}\left(  \bar{r}\right)  \right\vert }.
\label{radius}%
\end{equation}

Let us still assume $\left\vert \bar{d}^{\prime}\right\vert \leq\frac{1}{2}$
and $\left\vert \bar{d}^{\prime\prime}\right\vert \leq C$ for the moment. Then
repeating the previous arguments, one could show that for $r_{1}>r_{0},$
\begin{equation}
\left\vert f^{\prime\prime}\left(  r_{1}\right)  \right\vert \leq Ce^{-r_{1}%
}+\frac{C}{r_{1}}+Ce^{-\sqrt{2}D\left(  r_{1}\right)  }+C\left\Vert
e^{-2\sqrt{2}\bar{d}}\right\Vert _{L^{\infty}\left(  r_{1},+\infty\right)
}.\label{f''}%
\end{equation}
We claim there exists a universal constant $C_{0}$ such that
\begin{equation}
\bar{d}\left(  r_{1}\right)  =3f\left(  r_{1}\right)  \text{ for }r_{1}%
>C_{0}.\label{size}%
\end{equation}

Suppose to the contrary that $\left(  \ref{size}\right)  $ is not satisfied at
some large point $s.$ Without loss of generality we assume $\left\Vert
e^{-2\sqrt{2}\bar{d}}\right\Vert _{L^{\infty}\left(  \left(  s,+\infty\right)
\right)  }=e^{-2\sqrt{2}\bar{d}\left(  s\right)  }.$ By $\left(
\ref{radius}\right)  ,$ we could find $s_{1}$ with
\[
\left\vert s_{1}-s\right\vert =o\left(  \min\left\{  d_{1}\left(  s\right)
,d_{2}\left(  s\right)  \right\}  \right)  =o\left(  f\left(  s\right)
\right)
\]
such that
\begin{equation}
\min\left\{  d_{1}\left(  s\right)  ,d_{2}\left(  s\right)  \right\}
>\frac{1}{\left\vert f^{\prime\prime}\left(  s_{1}\right)  \right\vert
}.\label{di}%
\end{equation}
We would like to show that for this point $s_{1},$ necessarily
\begin{equation}
\frac{1}{\left\vert f^{\prime\prime}\left(  s_{1}\right)  \right\vert }%
\geq3f\left(  s\right)  .\label{aim}%
\end{equation}
Once this is proved, we then get a contradiction and hence the proof will be complished.

To prove $\left(  \ref{aim}\right)  ,$ we could assume $s_{1}\neq s,$
otherwise using estimate $\left(  \ref{f''}\right)  $, we obtain $\left(
\ref{aim}\right)  $. For the point $s_{1},$ we have
\begin{align*}
\left\vert f^{\prime\prime}\left(  s_{1}\right)  \right\vert  &  \leq
Ce^{-s_{1}}+\frac{C}{s_{1}}+Ce^{-\sqrt{2}D\left(  s_{1}\right)  }+C\left\Vert
e^{-2\sqrt{2}\bar{d}}\right\Vert _{L^{\infty}\left(  s_{1},+\infty\right)  }\\
&  =Ce^{-s_{1}}+\frac{C}{s_{1}}+Ce^{-\sqrt{2}D\left(  s_{1}\right)
}+Ce^{-2\sqrt{2}\bar{d}\left(  \bar{s}_{1}\right)  }%
\end{align*}
for some point $\bar{s}_{1}\geq s_{1}.$ If $Ce^{-2\sqrt{2}\bar{d}\left(
\bar{s}_{1}\right)  }<\frac{1}{2}\left\vert f^{\prime\prime}\left(
s_{1}\right)  \right\vert ,$ then
\[
\left\vert f^{\prime\prime}\left(  s_{1}\right)  \right\vert \leq Ce^{-s_{1}%
}+\frac{C}{s_{1}}+Ce^{-\sqrt{2}D\left(  s_{1}\right)  },
\]
which implies $\left(  \ref{aim}\right)  .$ Hence we could assume
\begin{equation}
Ce^{-2\sqrt{2}\bar{d}\left(  \bar{s}_{1}\right)  }\geq\frac{1}{2}\left\vert
f^{\prime\prime}\left(  s_{1}\right)  \right\vert . \label{d2}%
\end{equation}
This in particular implies that
\[
\bar{d}\left(  \bar{s}_{1}\right)  \leq\frac{1}{2\sqrt{2}}\left(  \ln\frac
{2C}{f^{\prime\prime}\left(  s_{1}\right)  }\right)  \leq\ln f\left(
s\right)  .
\]
For the point $\bar{s}_{1},$ by $\left(  \ref{radius}\right)  ,$ we could find
a point $s_{2}$ with
\[
s_{2}\geq\bar{s}_{1}+o\left(  \min\left\{  d_{1}\left(  \bar{s}_{1}\right)
,d_{2}\left(  \bar{s}_{1}\right)  \right\}  \right)  \geq s_{1}-\ln f\left(
s\right)  ,
\]
such that
\[
\min\left\{  d_{1}\left(  \bar{s}_{1}\right)  ,d_{2}\left(  \bar{s}%
_{1}\right)  \right\}  \geq\frac{1}{f^{\prime\prime}\left(  s_{2}\right)  }.
\]
Note that by $\left(  \ref{d2}\right)  ,$%
\[
Ce^{-\frac{2\sqrt{2}}{f^{\prime\prime}\left(  s_{2}\right)  }}\geq\frac{1}%
{2}\left\vert f^{\prime\prime}\left(  s_{1}\right)  \right\vert .
\]
Now repeat this argument and we get a sequence $\left\{  s_{n}\right\}  $
with
\begin{equation}
s_{n+1}\geq s_{n}-\ln\frac{1}{\left\vert f_{n}^{\prime\prime}\left(  s\right)
\right\vert }, \label{sn}%
\end{equation}
such that
\[
Ce^{-\frac{2\sqrt{2}}{f^{\prime\prime}\left(  s_{n+1}\right)  }}\geq\frac
{1}{2}\left\vert f^{\prime\prime}\left(  s_{n}\right)  \right\vert .
\]
If the process stops at the step $n_{0}<f\left(  s\right)  .$ Then
\[
f^{\prime\prime}\left(  s_{n_{0}}\right)  \leq Ce^{-s_{n_{0}}}+\frac
{1}{s_{n_{0}}}+e^{-D\left(  s_{n_{0}}\right)  }.
\]
By $\left(  \ref{sn}\right)  ,$ $s_{n_{0}}>s-f\left(  s\right)  $ and
$D\left(  s_{n_{0}}\right)  >f\left(  s\right)  .$ This then implies $\left(
\ref{aim}\right)  .$ On the other hand, if the process doesn't stop at the
step $n_{0}>f\left(  s\right)  ,$ then the sequence $\rho_{n}:=\frac{2\sqrt
{2}}{\left\vert f^{\prime\prime}\left(  s_{n}\right)  \right\vert }$
satisfies
\begin{equation}
\rho_{n}\geq Ce^{\rho_{n+1}}. \label{r}%
\end{equation}
$\left(  \ref{r}\right)  $ leads to the inequality
\[
\rho_{1}\geq Cn_{0}^{2}\geq Cf^{2}\left(  s\right)  ,
\]
This also implies $\left(  \ref{aim}\right)  .$

Let us we return back to the assumption that $\left\vert \bar{d}^{\prime
}\right\vert \leq\frac{1}{2}$ and $\left\vert \bar{d}^{\prime\prime
}\right\vert \leq C.$ Essentially, this assumption is to make sure that the
cutoff function $\eta$ has bounded first and second derivatives and thus the
error of approximate solution could be uniformly controlled. However, with the
original definition of $\bar{d}$ this assumption may not be true. To overcome
this difficulty, we should modify the function $\bar{d}$(that is, one needs to
modify the domain $\mathcal{B}_{u}$). Let us be more precise. Fix a small
positive constant $\delta.$ We define new function $\hat{d}$ by
\[
\hat{d}\left(  r_{1}\right)  :=\inf\left\{  \bar{d}\left(  s\right)
+\delta\left(  r_{1}-s\right)  :r_{0}<s<r_{1};\bar{d}\left(  s\right)
-\delta\left(  r_{1}-s\right)  :s>r_{1}\right\}  .
\]
Modify $\hat{d}$ such that it becomes $C^{2}.$ We then define the domain
$\mathcal{B}_{u}$ using this new function $\hat{d}.$ Using similar arguments
as before, we could show $\left(  \ref{size}\right)  .$

Now $\left(  \ref{size}\right)  $ implies that the size of the Fermi
coordinate is actually large enough for our purpose, that is, Assumption
$\left(  \ref{Assum}\right)  $ is indeed satisfied.
\end{proof}

With all these preparation, we proceed to prove Proposition \ref{compact}, the
main result of this section.

\begin{proof}
[Proof of Proposition \ref{compact} and Theorem \ref{main2}]We split the proof
into several steps.

\textit{Step 1.} We first show that there exists a universal constant
$C,\delta>0$ such that
\[
p\left(  r\right)  \geq C+\left(  \frac{\sqrt{2}}{2}+\delta\right)  \ln r.
\]
We would like to use equation $\left(  \ref{f2}\right)  $. Note that there is
a term involving the third derivative of $f$ in this equation, although one
expects that $f^{\prime\prime\prime}$ decays like $r^{-3},$ a priori we don't
have any decay information for it. However, we at least know that
$f^{\prime\prime\prime}\left(  r_{1}\right)  $ tends to zero as $r_{1}$
tending to infinity. Our aim is to show the following estimate(not optimal):%
\begin{equation}
\left\vert f^{\prime\prime\prime}\left(  r_{1}\right)  \right\vert \leq
Cr_{1}^{-1}+Ce^{-D\left(  r_{1}\right)  }. \label{f'''}%
\end{equation}
To obtain this estimate, we would like to differentiate  equation $\left(
\ref{f2}\right)  $ with respect to $r_{1}.$ Using the fact that $f^{\left(
3\right)  }$ and $f^{\left(  4\right)  }$ are small for $r_{1}$ large(at this
stage, we don't know the decay rate for these two terms, but they are
multiplied by terms related to $\phi$), we obtain the following estimate
(which is not optimal but enough for our purpose)
\begin{align*}
&  f^{\prime\prime\prime}+\sqrt{1+f^{\prime2}}h^{\prime\prime\prime}\\
&  =O\left(  h^{\prime\prime}\right)  +O\left(  r_{1}^{-1}\right)  +O\left(
e^{-\sqrt{2}D}\right) \\
&  +O\left(  h^{\prime}\right)  +O\left(  \left\Vert \phi^{\ast}\left(
r_{1},\cdot\right)  \right\Vert _{\infty}\right)  +O\left(  \left\Vert
\partial_{r_{1}}\phi^{\ast}\left(  r_{1},\cdot\right)  \right\Vert _{\infty
}e^{-\sqrt{2}f\left(  r_{1}\right)  }\right) \\
&  +O\left(  \left\Vert \partial_{r_{1}}^{2}\phi^{\ast}\left(  r_{1}%
,\cdot\right)  \right\Vert _{\infty}\right)  +O\left(  \left\Vert
\partial_{r_{1}}\partial_{z_{1}}\phi^{\ast}\left(  r_{1},\cdot\right)
\right\Vert _{\infty}\right)  .
\end{align*}
Hence by the $C^{2,\alpha}$ estimate of $\phi$, to prove $\left(
\ref{f'''}\right)  ,$ it will be suffice to obtain an estimate for the
function $h^{\prime\prime\prime},$ which is essentially controlled by
$\partial_{r_{1}}^{3}\phi^{\ast}.$ To achieve this, we consider the equation
\begin{equation}
L\phi=\left[  E\left(  \bar{u}\right)  \right]  ^{\bot}+\left[  E\left(
\bar{u}\right)  \right]  ^{\Vert}+P\left(  \phi\right)  . \label{L2}%
\end{equation}
Here we have in mind that by Lemma \ref{P1}, $\left[  E\left(  \bar{u}\right)
\right]  ^{\Vert}$ is expressed as a small order term of $\phi.$ We also have
the estimate $\left(  \ref{Eorth}\right)  $ for the source term $\left[
E\left(  \bar{u}\right)  \right]  ^{\bot}.$ Now differentiate equation
$\left(  \ref{L2}\right)  $ with respect to $r_{1},$ use the $L^{\infty}$ norm
estimate of $\phi$ and the fact that
\[
\partial_{r_{1}}\left[  E\left(  \bar{u}\right)  \right]  ^{\bot}=O\left(
r_{1}^{-1}\right)  +O\left(  e^{-D}\right)  +o\left(  h^{\prime\prime\prime
}\right)  ,
\]
we find that for $r_{1}$ large,
\[
\left\vert \partial_{r_{1}}^{3}\phi^{\ast}\right\vert \leq Cr_{1}^{-1}%
+Ce^{-D}.
\]
This then implies $\left(  \ref{f'''}\right)  .$

With the decay estimate of $f^{\prime\prime\prime}$ available, equation
$\left(  \ref{f2}\right)  $ could be refined to
\begin{equation}
\left(  \frac{r_{1}p^{\prime}}{\sqrt{1+p^{\prime2}}}\right)  ^{\prime}%
=\frac{\mathbf{c}_{1}}{\mathbf{c}_{0}}r_{1}\left(  1+o\left(  1\right)
\right)  e^{-\sqrt{2}D}+O\left(  r_{1}^{-2}\right)  :=I_{1}+I_{2}.
\label{p-equation}%
\end{equation}
Here $p=f+\sqrt{1+f^{\prime2}}h$ is the function introduced in Lemma \ref{P2}.
At this stage, one of the technical difficulty is that this equation involves
the function $D.$ As we mentioned before, we expect that $D$ is very close to
$2p.$ But without a priori estimate for $f^{\prime}$ and $f,$ $\left\vert
D-2p\right\vert $ in principle could be large. In any case, we at least know
that $D<2p.$

To proceed, we shall fix a small constant $\bar{\delta}>0$ which will be
determined later. For any interval $\left(  t_{1},t_{2}\right)  \subset\left(
r_{0},+\infty\right)  ,$ integrating equation $\left(  \ref{p-equation}%
\right)  $ in $r_{1}$ yields \
\begin{equation}
\frac{t_{2}p^{\prime}\left(  t_{2}\right)  }{\sqrt{1+p^{\prime2}\left(
t_{2}\right)  }}-\frac{t_{1}p^{\prime}\left(  t_{1}\right)  }{\sqrt
{1+p^{\prime2}\left(  t_{1}\right)  }}=\int_{t_{1}}^{t_{2}}I_{1}+\int_{t_{1}%
}^{t_{2}}I_{2}. \label{inte}%
\end{equation}
Keep in mind that $\int_{t_{1}}^{t_{2}}I_{1}$ is always positive, and by
$\left\vert I_{2}\left(  r\right)  \right\vert \leq C_{0}r^{-2}$,
\begin{equation}
\left\vert \int_{t_{1}}^{t_{2}}I_{2}\right\vert \leq\frac{C_{0}}{t_{1}}.
\label{I2}%
\end{equation}
Set $t^{\ast}:=\max\left\{  \frac{2C_{0}}{\bar{\delta}},r_{0}\right\}  .$ We
consider two cases. \ Case 1:
\[
\frac{t^{\ast}p^{\prime}\left(  t^{\ast}\right)  }{\sqrt{1+p^{\prime2}\left(
t^{\ast}\right)  }}\geq\frac{\sqrt{2}}{2}+\bar{\delta}.
\]
Then by $\left(  \ref{inte}\right)  $ and $\left(  \ref{I2}\right)  $, for all
$t>t^{\ast},$
\[
\frac{tp^{\prime}\left(  t\right)  }{\sqrt{1+p^{\prime2}\left(  t\right)  }%
}\geq\frac{\sqrt{2}}{2}+\frac{\bar{\delta}}{2}.
\]
Case 2:
\[
\frac{t^{\ast}p^{\prime}\left(  t^{\ast}\right)  }{\sqrt{1+p^{\prime2}\left(
t^{\ast}\right)  }}<\frac{\sqrt{2}}{2}+\bar{\delta}.
\]
In this case, let $\left(  t^{\ast},\hat{t}\right)  $ be an interval such
that
\[
\frac{tp^{\prime}\left(  t\right)  }{\sqrt{1+p^{\prime2}\left(  t\right)  }%
}\leq\frac{\sqrt{2}}{2}+2\bar{\delta},\forall t\in\left(  t^{\ast},\hat
{t}\right)  .
\]
Then the fact that $\left\vert p^{\prime}\right\vert =o\left(  1\right)  $ and
a simple integration yield
\[
p\left(  t\right)  \leq\left(  \frac{\sqrt{2}}{2}+3\bar{\delta}\right)  \ln
t+C,\forall t\in\left(  t^{\ast},\hat{t}\right)  .
\]
This in particular implies that $D\left(  r\right)  -2f\left(  r\right)
=o\left(  1\right)  $ and hence for any $t_{1},t_{2}\in\left(  t^{\ast}%
,\hat{t}\right)  ,$
\[
\int_{t_{1}}^{t_{2}}se^{-\sqrt{2}D\left(  s\right)  }ds\geq\frac{C}%
{t_{1}^{6\sqrt{2}\bar{\delta}}}-\frac{C}{t_{2}^{6\sqrt{2}\bar{\delta}}}.
\]
We choose $\bar{\delta}$ such that $6\sqrt{2}\bar{\delta}<1.$ Then by
inequality $\left(  \ref{I2}\right)  ,$ if $t_{2}$ is large,
\[
\int_{t_{1}}^{t_{2}}I_{1}+\int_{t_{1}}^{t_{2}}I_{2}>0.
\]
This implies
\[
\frac{t_{2}p^{\prime}\left(  t_{2}\right)  }{\sqrt{1+p^{\prime2}\left(
t_{2}\right)  }}>\frac{t_{1}p^{\prime}\left(  t_{1}\right)  }{\sqrt
{1+p^{\prime2}\left(  t_{1}\right)  }}.
\]

Combing the analysis for Case 1 and Case 2, we find that there exist
universal constants $\delta$ and $C_{1}$ such that
\begin{equation}
rp^{\prime}\left(  r\right)  >\frac{\sqrt{2}}{2}+\delta,\text{ }r>C_{1}.
\label{lower}%
\end{equation}
This proves Theorem \ref{main2}.

With the lower bound $\left(  \ref{lower}\right)  $ available, one could show
that
\[
e^{-\sqrt{2}D}=e^{-2\sqrt{2}p}+O\left(  r^{-(2+\alpha)}\right)  .
\]
It also follows that
\[
p^{\prime}\left(  r\right)  =\frac{k}{r}+O\left(  r^{-1-\alpha}\right)  ,
\]
for some $\alpha>0.$ Here $k$ is the growth rate of $u.$ Therefore,
\[
p\left(  r\right)  =k\ln r+O\left(  r^{-\alpha}\right)  +C.
\]
With this estimate at hand, we find that $u_{n}$ converges to a two-end
solution $u_{0}$ strongly.
\end{proof}

\section{\label{Moduli}Moduli space theory of two-end solutions}

\subsection{Preliminary results}

Generally speaking, the structure of the set of bounded entire solutions to
the Allen-Cahn equation could potentially be very complicated. However, if one
impose certain conditions at infinity for the solution, then it could be
simpler. The moduli space theory for multiple-end solutions of the Allen-Cahn
equation in $\mathbb{R}^{2}$ was developed in \cite{dkp-2009}. This theory
tells us that if $u$ is a $2k$-end solutions in $\mathbb{R}^{2}$ and
nondegenerate, then around $u$ (in suitable sense)$,$ the set of $2k$-end
solutions is actually a $2k$-dimensional manifold. This fact has been used in
an essential way in the classification of four-end solutions of Allen-Cahn
equation in $\mathbb{R}^{2}$(\cite{MR3148064}). In this section, we would like
to develop the corresponding moduli space theory for two-end solutions in
$\mathbb{R}^{3}$. Our main result states that the moduli space of two-end
solutions has a structure of real analytic variety of formal dimension 1.

To begin with, let us recall some preliminary results about real analytic
operators. Compared to $C^{\infty}$ operators, real analytic operators has
better structures. We refer to \cite{MR1956130}, \cite{MR0375019},
\cite{MR1962054} for more details.

Let $X$ and $Y$ be Banach spaces and $U$ an open subset of $X.$ We first
recall the notion of real analytic operator.

\begin{definition}
A map $F:U\rightarrow Y$ is real analytic at $x_{0}\in U$ if there exists a
$\delta>0$ such that
\[
F\left(  x\right)  -F\left(  x_{0}\right)  =%
{\displaystyle\sum\limits_{k=1}^{+\infty}}
m_{k}\left(  \left(  x-x_{0}\right)  ,...,\left(  x-x_{0}\right)  \right)
,\text{ for }\left\vert x-x_{0}\right\vert <\delta,
\]
where $m_{k}$ is a symmetric $k-$linear operator, and there exists $r>0$ such
that
\[
\sup_{k>0}r^{k}\left\Vert m_{k}\right\Vert <+\infty.
\]
The function is said to be real analytic on $U$ if it is real analytic at
every point of $U.$
\end{definition}

Let $F:U\rightarrow Y$ be a real analytic functional. Suppose that $dF\left(
x\right)  $ is a Fredholm operator of index $1.$ Assume there exists a map
$\Lambda:\left(  0,\varepsilon\right)  \rightarrow Y$ such that $F\left(
\Lambda\left(  s\right)  \right)  =0,$ and $dF\left(  \Lambda\left(  s\right)
\right)  :X\rightarrow Y,$ is surjective for all $s\in\left(  0,\varepsilon
\right)  .$ Let
\[
S=\left\{  x\in U:F\left(  x\right)  =0\right\}  .
\]
The following theorem has been proved in the book of Buffoni and
Toland\cite{MR1956130}.

\begin{theorem}
\label{structure}Suppose all bounded closed subsets of $S$ are compact in $X.$
Then there exists an extension of $\Lambda,$ denoted by $\bar{\Lambda}:$%
\[
\left(  0,+\infty\right)  \rightarrow X,
\]
satisfying: (1) $\bar{\Lambda}$ is continuous. (2) The set of points where
$dF$ is not surjective has no accumulation points. (3) One of the following
happens: (i) $\left\Vert \bar{\Lambda}\left(  s\right)  \right\Vert
\rightarrow+\infty,$ as $s\rightarrow+\infty;$ (ii) $\bar{\Lambda}\left(
s\right)  $ approaches the boundary of $U$ as $s$ tends to infinity; (iii)
$\bar{\Lambda}\left(  \left(  0,+\infty\right)  \right)  $ is a closed loop.
\end{theorem}

Basically, this theorem tells us that if one has a real analytic variety which
comes from the zero set of a real analytic operator and its formal dimension
is one, and assume further that on the variety there are some points where the
operator is surjective, then under certain compactness assumption, one could
find a continuous path of solutions where at most finitely many solutions are
degenerate (the linearized operator is not surjective).

\subsection{\label{real analytic}The real analytic structure of the moduli
space}

Let $u$ be a two-end solution. Then $u$ satisfies $\left(  \ref{axial}\right)
$ and there are constants $k$ and $c_{k}\in\mathbb{R}$ such that
\begin{equation}
\left\Vert u\left(  r,\cdot\right)  -H\left(  \cdot-k\ln r-c_{k}\right)
\right\Vert _{L^{\infty}\left(  [0,+\infty\right)  )}\rightarrow0,\text{ as
}r\rightarrow+\infty. \label{asym}%
\end{equation}
We also assume $k>\sqrt{2}.$ The moduli space theory of noncompact geometric
objects with controlled geometry at infinity(minimal surfaces with finite
total curvature, singular Yamabe metrics, constant mean curvature surfaces
with Delaunay ends) has been developed in \cite{MR1371233}, \cite{MR1356375},
\cite{MR1487630}. In this paper, we will not investigate the general moduli
space theory for the Allen-Cahn equation in dimension three. Instead, we shall
only consider those solutions satisfying $\left(  \ref{axial}\right)  .$ Our
aim is to show that the set of solutions of $\left(  \ref{axial}\right)  $
satisfying $\left(  \ref{asym}\right)  $ has the structure of a real analytic
variety with formal dimension $1.$ Additionally, if a solution $u$ is
nondegenerate, then around $u,$ this real analytic variety is actually a one
dimensional real analytic manifold.

Consider the function
\[
f_{k,b}\left(  r\right)  =k\cosh^{-1}\left(  k^{-1}r\right)  +b.
\]
where $k$ and $b$ are real parameters, $k>\sqrt{2}.$ Obviously, $f$ represents
a vertically translated catenoidal end and we have in mind that $f$ is the
asymptotic curve of the nodal line of a solution $u.$ The moduli space theory
for two-end solutions will,  roughly speaking,  state that around a given solution
$u$ there is a one dimensional family of solutions, with $k$ or $b$ being
possible candidates for the local parameters. To make this statement more
precise, we need to analyze the mapping property of the linearized Allen-Cahn
operator around $u.$

Let us introduce some notations. As in Section \ref{compactness}, we shall
also carry out the analysis through the Fermi coordinate. Let $\mathcal{B}%
_{u}$ be the domain where the Fermi coordinate of $z=f\left(  r\right)  $ is
well defined(suitably modified if necessary). We still use $\eta$ to denote a
smooth cutoff function supported in $\mathcal{B}_{u}$. Similarly, we have the
cutoff function $\eta^{+}$ supported in $\mathcal{B}_{u}\cap\mathbb{E}^{+}$.
Similarly, we have the map $X:\left(  r_{1},z_{1}\right)  \rightarrow\left(
r,z\right)  .$

We need to work in suitable functional spaces $S_{1},S_{2}$ which will be
described now. Let $C_{1}>0$ be a fixed constant and $\delta>0$ small. A
$C^{2,\alpha}$ function $\Xi\in$ $S_{1}$, if and only if $\Xi\left(
r,z\right)  =\Xi\left(  r,-z\right)  $ and it satisfies the following
condition: In $\mathbb{E}^{+}$, when $r_{1}>C_{1},$%

\[
X^{\ast}\Xi\left(  r_{1},z_{1}\right)  =p\left(  r_{1}\right)  X^{\ast}%
\eta^{+}\left(  r_{1},z_{1}\right)  H^{\prime}\left(  z_{1}\right)  +X^{\ast
}\psi\left(  r_{1},z_{1}\right)  ,
\]
where $\int_{\mathbb{R}}X^{\ast}\left(  \eta^{+}\psi\right)  H^{\prime}%
dz_{1}=0,$
\[
\left\Vert \left(  1+r_{1}\right)  ^{2}p\right\Vert _{L^{\infty}}+\left\Vert
\left(  1+r_{1}\right)  ^{3}p^{\prime}\right\Vert _{L^{\infty}}+\left\Vert
\left(  1+r_{1}\right)  ^{4}p^{\prime\prime}\right\Vert _{C^{0,\alpha}%
}<+\infty,
\]
and
\[
\left\Vert \left(  1+r_{1}\right)  ^{4}X^{\ast}\psi\right\Vert _{C^{2,\alpha}%
}<+\infty.
\]

Similarly, a $C^{0,\alpha}$ function $\Xi\in$ $S_{2}$ if and only if
$\Xi\left(  r,z\right)  =\Xi\left(  r,-z\right)  $ and it satisfies the
following condition: In $\mathbb{E}^{+},$ when $r_{1}>C_{1},$%

\[
X^{\ast}\Xi\left(  r_{1},z_{1}\right)  =p\left(  r_{1}\right)  X^{\ast}%
\eta^{+}\left(  r_{1},z_{1}\right)  H^{\prime}\left(  z_{1}\right)  +X^{\ast
}\psi\left(  r_{1},z_{1}\right)  ,
\]
where $\int_{\mathbb{R}}X^{\ast}\left(  \eta^{+}\psi\right)  H^{\prime}%
dz_{1}=0,$
\[
\left\Vert \left(  1+r_{1}\right)  ^{4}p\right\Vert _{C^{0,\alpha}}<+\infty.
\]
and%
\[
\left\Vert \left(  1+r_{1}\right)  ^{4}X^{\ast}\psi\right\Vert _{C^{0,\alpha}%
}<+\infty.
\]

Let $\mathcal{L}$ be the linearized operator of the Allen-Cahn equation around
a two-end solution $u,$ that is,
\[
\mathcal{L}=\mathcal{L}_{u}:=\Delta_{\left(  r,z\right)  }+r^{-1}\partial
_{r}+1-3u^{2}.
\]
For each point $p,$ let $\left(  r_{1}\left(  p\right)  ,z_{1}\left(
p\right)  \right)  $ be its Fermi coordinate with respect to $f_{k,b}.$ Let
$p_{s,t}^{\prime}$ be the point whose Fermi coordinate with respect to the
curve $z=f_{k+s,b+t}$ is still equal to $\left(  r_{1}\left(  p\right)
,z_{1}\left(  p\right)  \right)  .$ For $p$ in the upper half plane, now let
$\Phi_{s,t}$ be the map defined by
\[
\Phi_{s,t}\left(  p\right)  =\eta^{+}p_{s,t}^{\prime}+\left(  1-\eta
^{+}\right)  p,
\]
while for $p=\left(  r,z\right)  $ in the lower half plane, we let
\[
\Phi_{s,t}\left(  p\right)  =-\Phi_{s,t}\left(  r,-z\right)  .
\]
In this way, we have defined the family of maps $\Phi_{s,t}$ which is even
with respect to the $r$ axis. For $\left\vert s\right\vert ,\left\vert
t\right\vert $ small, $\Phi_{s,t}$ is a diffeomorphism. By definition
$\Phi_{0,0}\left(  p\right)  =p.$ Similarly, we introduce the family of
diffeomorphism%
\[
\Psi_{s,t}\left(  p\right)  :=\eta p_{s,t}^{\prime}+\left(  1-\eta\right)  p.
\]

Similar arguments as that of Section \ref{compact} shows that at far away an
axially symmetric two-end solution $u$ could be written as:%
\[
u=\mathcal{H}_{1}+\mathcal{H}_{1}^{s}+1+\phi.
\]
Here%
\[
\mathcal{H}_{1}=\eta H_{1}+\left(  1-\eta\right)  \frac{H_{1}}{\left\vert
H_{1}\right\vert },
\]
where
\[
X^{\ast}H_{1}\left(  r_{1},z_{1}\right)  \mathcal{=}H\left(  z_{1}-h\left(
r_{1}\right)  \right)
\]
and
\[
\int_{\mathbb{R}}X^{\ast}\left(  \eta^{+}\phi\right)  H^{\prime}dz_{1}=0.
\]
Notice that $\phi$ decays like $r^{-4}$ at infinity.

For $\left(  s,t,\psi\right)  \in\mathbb{R}^{2}\oplus S_{1},$ we then define a
family of function $u_{s,t,\psi}$ such that at far away
\[
u_{s,t,\psi}=\mathcal{H}_{1}\mathcal{\circ}\Psi_{s,t}^{-1}+\left[
\mathcal{H}_{1}\mathcal{\circ}\Psi_{s,t}^{-1}\right]  ^{s}+\left(  \phi
+\psi\right)  \circ\Phi_{s,t}^{-1},
\]
and in a fixed large ball, $u_{s,t,\psi}=u.$ Clearly, for $u_{0,0,0}=u.$
Consider the nonlinear map $N:\mathbb{R}^{2}\oplus S_{1}\rightarrow S_{2}$
given by%
\[
\left(  s,t,\psi\right)  \rightarrow\left[  \Delta u_{s,t,\psi}+u_{s,t,\psi
}-u_{s,t,\psi}^{3}\right]  \circ\Phi_{s,t}.
\]
The reason that $N$ maps $\mathbb{R}^{2}\oplus S_{1}$ into $S_{2}$ lies in the
asymptotic behavior of $u$, see Section \ref{compactness}. We have%
\begin{align*}
\partial_{s}N\left(  s,t,\psi\right)   &  =\left[  \mathcal{L}_{u_{s,t,\psi}%
}\partial_{s}u_{s,t,\psi}\right]  \circ\Phi_{s,t}+\partial_{s}\Phi_{s,t}%
\cdot\nabla\left[  \Delta u_{s,t,\psi}+u_{s,t,\psi}-u_{s,t,\psi}^{3}\right]
\circ\Phi_{s,t},\\
\partial_{t}N\left(  s,t,\psi\right)   &  =\left[  \mathcal{L}_{u_{s,t,\psi}%
}\partial_{t}u_{s,t,\psi}\right]  \circ\Phi_{s,t}+\partial_{t}\Phi_{s,t}%
\cdot\nabla\left[  \Delta u_{s,t,\psi}+u_{s,t,\psi}-u_{s,t,\psi}^{3}\right]
\circ\Phi_{s,t},\\
\partial_{\psi}N\left(  s,t,\psi\right)  G  &  =\left[  \mathcal{L}%
_{u_{s,t,\psi}}G\right]  \circ\Phi_{s,t}.
\end{align*}
In particular, although $DN$ is not exactly equal to $\mathcal{L},$ it is a
small perturbation of it for $s,t$ small. Observe that when $\left(
s,t,\psi\right)  =\left(  0,0,0\right)  ,$ $DN$ actually is equal to
$\mathcal{L}.$ Let
\begin{align*}
\gamma_{1}  &  =\partial_{s}u_{s,t,\psi}|_{\left(  s,t,\psi\right)  =\left(
0,0,0\right)  },\\
\gamma_{2}  &  =\partial_{t}u_{s,t,\psi}|_{\left(  s,t,\psi\right)  =\left(
0,0,0\right)  }.
\end{align*}
Introduce the deficiency space
\[
\mathcal{D=}span\left\{  \gamma_{1},\gamma_{2}\right\}  .
\]
The next result concerns the mapping property of the linearized operator
$\mathcal{L}$ and is the main result of this section.

\begin{proposition}
The operator $\mathcal{L}:S_{1}\oplus\mathcal{D}\rightarrow S_{2}$ is a
Fredholm operator of index $1.$
\end{proposition}

\begin{proof}
Let us prove the result under the additional assumption that $f\left(
0\right)  $ is large and $\left\Vert f^{\prime}\right\Vert _{L^{\infty}%
},\left\Vert f^{\prime\prime}\right\Vert _{L^{\infty}}$ are very small. In the
general case, we could modify the function $u$ inside a compact set into a
function whose nodal lines are almost parallel, and use the fact that the
corresponding linearized operator is a compact perturbation(thus the Fredholm
index is preserved) of $\mathcal{L}.$ We shall split the proof into several steps.

\textbf{Step 1}. For each function $\Theta\in S_{2},$ we shall find a solution
$\Xi$ to the equation%
\begin{equation}
\mathcal{L}\Xi=\Theta. \label{e1}%
\end{equation}

By the definition of $S_{2},$ in $\mathbb{E}^{+},$ for $r_{1}$ large,
\[
X^{\ast}\Theta\left(  r_{1},z_{1}\right)  =\omega\left(  r_{1}\right)
X^{\ast}\eta^{+}H^{\prime}\left(  z_{1}\right)  +X^{\ast}\varphi\left(
r_{1},z_{1}\right)  ,
\]
for some function $\omega$ and $\varphi,$ where $\int_{\mathbb{R}}X^{\ast
}\left(  \eta^{+}\varphi\right)  H^{\prime}=0.$ To solve $\left(
\ref{e1}\right)  ,$ adopt the same notation as in Section \ref{compactness},
it will be suffice to solve the following system%
\begin{equation}
\left\{
\begin{array}
[c]{c}%
\left(  \mathcal{L}\Xi\right)  _{1}^{\Vert}=\Theta_{1}^{\Vert},\\
\left(  \mathcal{L}\Xi\right)  ^{\bot}=\Theta^{\bot}.
\end{array}
\right.  \label{sys1}%
\end{equation}
For convenience, we recall the expression of $\mathcal{L}$in the Fermi
coordinate.
\begin{align*}
\mathcal{L}=A^{-1}\partial_{r_{1}}^{2}+\partial_{z_{1}}^{2}  &  +\frac{1}%
{2}\frac{\partial_{z_{1}}A}{A}\partial_{z_{1}}-\frac{1}{2}\frac{\partial
_{r_{1}}A}{A^{2}}\partial_{r_{1}}\\
&  +r^{-1}\left(  \partial_{r}r_{1}\partial_{r_{1}}+\partial_{r}z_{1}%
\partial_{z_{1}}\right)  +1-3u^{2}.
\end{align*}
Let us first compute the action of $\mathcal{L}$ on functions of the form
$\xi\left(  r_{1}\right)  \eta^{+}H^{\prime}\left(  z_{1}\right)  .$ Using
$\left(  \ref{derivative}\right)  $ we get
\begin{align*}
\mathcal{L}\left(  \xi\eta^{+}H^{\prime}\right)   &  =A^{-1}\xi^{\prime\prime
}\eta^{+}H^{\prime}+\left[  -\frac{1}{2}\frac{\partial_{r_{1}}A}{A^{2}}%
+\frac{1}{r\left(  1+f^{\prime2}\right)  B}\right]  \xi^{\prime}\eta
^{+}H^{\prime}\\
&  +\left[  \frac{1}{2}\frac{\partial_{z_{1}}A}{A}-\frac{f^{\prime}}%
{r\sqrt{1+f^{\prime2}}}\right]  \xi\eta^{+}H^{\prime\prime}\\
&  +\left[  H^{\prime\prime\prime}+\left(  1-3u^{2}\right)  H^{\prime}\right]
\xi\eta^{+}\\
&  +2A^{-1}\xi^{\prime}\partial_{r_{1}}\eta^{+}H^{\prime}+A^{-1}\xi
\partial_{r_{1}}^{2}\eta^{+}H^{\prime}+\xi\partial_{z_{1}}\eta^{+}%
H^{\prime\prime}+\xi\partial_{z_{1}}^{2}\eta^{+}H^{\prime}\\
&  +\left[  -\frac{1}{2}\frac{\partial_{r_{1}}A}{A^{2}}+\frac{1}{r\left(
1+f^{\prime2}\right)  B}\right]  \xi\partial_{r_{1}}\eta^{+}H^{\prime}\\
&  +\left[  \frac{1}{2}\frac{\partial_{z_{1}}A}{A}-\frac{f^{\prime}}%
{r\sqrt{1+f^{\prime2}}}\right]  \xi\partial_{z_{1}}\eta^{+}H^{\prime}.
\end{align*}
The coefficient before $\xi^{\prime}$ has the estimate
\[
-\frac{1}{2}\frac{\partial_{r_{1}}A}{A^{2}}+\frac{1}{r\left(  1+f^{\prime
2}\right)  B}=\frac{1}{r}+O\left(  r^{-3}\right)  .
\]
Next we shall estimate the coefficient before $\xi.$ Recall that
\[
\frac{1}{2}\frac{\partial_{z_{1}}A}{A}-\frac{f^{\prime}}{r\sqrt{1+f^{\prime2}%
}}=O\left(  r^{-6}\right)  .
\]
We also have
\begin{align*}
H^{\prime\prime\prime}+\left(  1-3u^{2}\right)  H^{\prime}  &  =3\left(
H^{2}-u^{2}\right)  H^{\prime}\\
&  =3\left(  H+u\right)  \left(  H-u\right)  H^{\prime}.
\end{align*}
Since
\begin{align*}
u-H  &  =H\left(  z_{1}-h\left(  r_{1}\right)  \right)  -H\left(
z_{1}\right)  +O\left(  r^{-2-\alpha}\right) \\
&  =-h\left(  r_{1}\right)  H^{\prime}+O\left(  r^{-2-\alpha}\right)
=O\left(  r^{-2-\alpha}\right)  ,
\end{align*}
hence $H^{\prime\prime\prime}+\left(  1-3u^{2}\right)  H^{\prime}=O\left(
r^{-2-\alpha}\right)  .$ As for those terms involving derivatives of $\eta
^{+},$ they could be estimated by $O\left(  r^{-2-\alpha}\right)  .$ Combining
all the above estimate, we get
\[
\mathcal{L}\left(  \xi\eta^{+}H^{\prime}\right)  =A^{-1}\xi^{\prime\prime}%
\eta^{+}H^{\prime}+\left(  \frac{1}{r_{1}}+O\left(  r_{1}^{-3}\right)
\right)  \xi^{\prime}\eta^{+}H^{\prime}+O\left(  r^{-2-\alpha}\right)  \xi.
\]

Introduce the operator
\[
K_{1}:\xi\rightarrow\int_{\mathbb{R}}\mathcal{L}\left(  \xi\eta^{+}H^{\prime
}\right)  \eta^{+}H^{\prime}dz_{1}.
\]
It will be important to analyze the mapping property of $K_{1}.$ Let us
consider the equation
\begin{equation}
K_{1}\xi=\omega. \label{ODE}%
\end{equation}
This is a second order ODE. Suppose $\vartheta_{1},\vartheta_{2}$ are two
linearly independent solutions of the corresponding homogeneous equation:
$K_{1}\vartheta_{i}=0.$ We could assume $\vartheta_{1}$ is bounded near $0$.
Let $\chi$ be a cutoff function equals $0$ in $\left(  0,1\right)  $ and
equals to $1$ in $\left(  2,+\infty\right)  .$ Then using the variation of
parameter formula, one could show that $K_{1}$ is an isomorphism from the
space $\mathcal{S}_{1}\oplus span\left\{  \chi\vartheta_{2}\right\}  $ to
$\mathcal{S}_{2}.$ Here a function $p\in\mathcal{S}_{1}$ iff $p\in
C^{2,\alpha}\left(  \mathbb{R}\right)  $ and
\[
\left\Vert \left(  1+r_{1}\right)  ^{2}p\right\Vert _{L^{\infty}}+\left\Vert
\left(  1+r_{1}\right)  ^{3}p^{\prime}\right\Vert _{L^{\infty}}+\max
_{s}\left\Vert \left(  1+r_{1}\right)  ^{4}p^{\prime\prime}\right\Vert
_{C^{0,\alpha}\left(  \overline{B_{1}\left(  s\right)  }\right)  }<+\infty,
\]
and a function $p\in\mathcal{S}_{2}$ iff $p\in C^{0,\alpha}\left(
\mathbb{R}\right)  $ and
\[
\ \max_{s}\left\Vert \left(  1+r_{1}\right)  ^{4}p\right\Vert _{C^{0,\alpha
}\left(  \overline{B_{1}\left(  s\right)  }\right)  }<+\infty.
\]

Define the operator
\[
K_{2}:\xi\rightarrow\left[  \mathcal{L}\left(  \xi\eta^{+}H^{\prime}+\left[
\xi\eta^{+}H^{\prime}\right]  ^{s}\right)  \right]  ^{\bot}.
\]
To solve $\left(  \ref{sys1}\right)  ,$ it suffices to find solution $\left(
\xi,\psi\right)  $ to the system:
\begin{equation}
\left\{
\begin{array}
[c]{l}%
K_{1}\xi=\omega-\int_{\mathbb{R}}\eta^{+}H^{\prime}\mathcal{L}\psi dz_{1},\\
\left(  \mathcal{L}\psi\right)  ^{\bot}=\varphi-K_{2}\xi.
\end{array}
\right.  \label{sys2}%
\end{equation}

Let $\psi_{0}$ be the solution of
\[
\left(  \mathcal{L}\psi_{0}\right)  ^{\perp}=-K_{2}\left(  \chi\vartheta
_{2}\right)  ,
\]
satisfying $\int_{\mathbb{R}}X^{\ast}\left(  \psi_{0}\eta^{+}\right)
H^{\prime}dz_{1}=0.$ Note that at infinity, $\vartheta_{1}$ behaves like
$O\left(  \ln r\right)  $ or $O\left(  1\right)  .$ Hence $K_{2}\vartheta_{1}$
behaves like $O\left(  r^{-2-\alpha}\right)  .$ This implies that $\psi
_{0}=O\left(  r^{-2-\alpha}\right)  .$ Thus $\int_{\mathbb{R}}X^{\ast}\left(
\mathcal{L}\psi_{0}\eta^{+}\right)  H^{\prime}dz_{1}=O\left(  r^{-4}\right)
.$ Thanks to these estimates, one then could use a fixed point argument to get
a pair of solution $\left(  \xi,\psi\right)  $ for $\left(  \ref{sys2}\right)
$ with the form%
\[
\left\{
\begin{array}
[c]{c}%
\xi=c_{1}\chi\vartheta_{2}+\xi_{0},\\
\psi=c_{1}\psi_{0}+\bar{\psi},
\end{array}
\right.
\]
for some constant $c_{1},$ where $\xi_{0}\in$ $\mathcal{S}_{1}$ and $\bar
{\psi}\in S_{1}.$(Note that $\psi_{0}$ does not belong to $S_{1}$). Then we
get a corresponding solution $\Xi$ to $\left(  \ref{sys1}\right)  .$ We
emphasize that at this stage, we still don't know whether the solution $\Xi$
belongs to $S_{1},$ due to the non-decaying part
\begin{equation}
\hat{\Xi}:=\chi\left(  r_{1}\right)  \vartheta_{2}\left(  r_{1}\right)
\eta^{+}H^{\prime}\left(  z_{1}\right)  +\left[  \chi\left(  r_{1}\right)
\vartheta_{2}\left(  r_{1}\right)  \eta^{+}H^{\prime}\left(  z_{1}\right)
\right]  ^{s}+\psi_{0}. \label{nonde}%
\end{equation}

\textbf{Step 2. }Investigate the homogeneous equation
\begin{equation}
\mathcal{L\zeta}=0. \label{homoge}%
\end{equation}

Firstly, we wish to find a nontrivial solution to $\left(  \ref{homoge}%
\right)  .$ To achieve this, we use the fact that $\vartheta_{1}$ is in the
kernel of $K_{1}.$ A perturbation argument could be applied similarly as in
step 1 to get a function $\Xi_{0}$ solving $\mathcal{L}\Xi_{0}=0,$where
$\Xi_{0}$ is around $\eta^{+}\vartheta_{1}\left(  r_{1}\right)  H^{\prime
}\left(  z_{1}\right)  +\left[  \eta^{+}\vartheta_{1}\left(  r_{1}\right)
H^{\prime}\left(  z_{1}\right)  \right]  ^{s}.$ We could also assume $\Xi_{0}$
satisfying
\[
\int_{\mathbb{R}}X^{\ast}\left(  \Xi_{0}\eta^{+}\right)  H^{\prime}\left(
z_{1}\right)  dz_{1}=\vartheta_{1}\left(  r_{1}\right)  +\delta\vartheta
_{2}\left(  r_{1}\right)  +O\left(  r_{1}^{-\alpha}\right)  ,
\]
for certain $\delta\in\mathbb{R}$ and $\alpha>0.$

Secondly we show the solution $\Xi_{0}$ is in some sense unique. For this
purpose, let us assume $\Xi_{0}^{\prime}$ is another function solves
$\mathcal{L}\Xi_{0}^{\prime}=0$ and
\[
\int_{\mathbb{R}}X^{\ast}\left(  \Xi_{0}\eta^{+}\right)  H^{\prime}\left(
z_{1}\right)  dz_{1}=\vartheta_{1}\left(  r_{1}\right)  +\delta^{\prime
}\vartheta_{2}\left(  r_{1}\right)  +O\left(  r_{1}^{-\alpha}\right)  ,
\]
for certain $\delta^{\prime}\in\mathbb{R}$ and $\alpha>0.$ Then the function
$g:=\Xi_{0}^{\prime}-\Xi_{0}$ solves $\mathcal{L}g=0$ and
\begin{equation}
\int_{\mathbb{R}}X^{\ast}\left(  g\eta^{+}\right)  H^{\prime}\left(
z_{1}\right)  dz_{1}=\left(  \delta^{\prime}-\delta\right)  \vartheta
_{2}\left(  r_{1}\right)  +O\left(  r_{1}^{-\alpha}\right)  .
\label{gprojection}%
\end{equation}

We claim that $g=0.$ Indeed, writing $g$ as $\eta^{+}\xi\left(  r_{1}\right)
H^{\prime}\left(  z_{1}\right)  +\left[  \eta^{+}\xi\left(  r_{1}\right)
H^{\prime}\left(  z_{1}\right)  \right]  ^{s}+\varphi,$ where
\[
\int_{\mathbb{R}}X^{\ast}\left(  \eta^{+}\varphi\right)  H^{\prime}\left(
z_{1}\right)  dz_{1}=0
\]
and $\varphi=O\left(  r^{-\alpha}\right)  $ for some $\alpha>0.$ We have
\[
\left\{
\begin{array}
[c]{l}%
K_{1}\xi=\int_{\mathbb{R}}\eta^{+}H^{\prime}\mathcal{L}\varphi dz_{1},\\
\left(  \mathcal{L}\varphi\right)  ^{\bot}=-K_{2}\xi.
\end{array}
\right.
\]
Recall that in the Fermi coordinate with respect to $f_{k,b},$
\begin{align*}
&  \int_{\mathbb{R}}X^{\ast}\left(  \eta^{+}H_{1}^{\prime}\mathcal{L}%
\varphi\right)  dz_{1}\\
&  =-\int_{\mathbb{R}}X^{\ast}\left(  \eta^{+}H_{1}^{\prime}\right)  \left[
A^{-1}\partial_{r_{1}}^{2}+\frac{1}{2}\frac{\partial_{z_{1}}A}{A}%
\partial_{z_{1}}-\frac{1}{2}\frac{\partial_{r_{1}}A}{A^{2}}\partial_{r_{1}%
}\right]  \varphi^{\ast}dz_{1}\\
&  -\int_{\mathbb{R}}X^{\ast}\left(  \eta^{+}H_{1}^{\prime}\right)
r^{-1}\left[  \partial_{r_{1}}\varphi^{\ast}\partial_{r}r_{1}+\partial_{z_{1}%
}\varphi^{\ast}\partial_{r}z_{1}\right]  dz_{1}\\
&  +\int_{\mathbb{R}}\left[  -X^{\ast}\left(  \eta^{+}H_{1}^{\prime}\right)
\partial_{z_{1}}^{2}\varphi^{\ast}+X^{\ast}\left(  \eta^{+}H_{1}^{\prime
}\left(  3\bar{u}^{2}-1\right)  \varphi\right)  \right]  dz_{1}.
\end{align*}
Hence $\int_{\mathbb{R}}X^{\ast}\left(  \eta^{+}H_{1}^{\prime}\mathcal{L}%
\varphi\right)  dz_{1}=O\left(  r^{-2-\alpha}\right)  .$ By $\left(
\ref{gprojection}\right)  ,$ we could write $\xi=\beta\vartheta_{2}\left(
r_{1}\right)  +\rho,$ where $\beta\in\mathbb{R}$ and $\rho\left(
\cdot\right)  =O\left(  r_{1}^{-\alpha}\right)  .$ By the mapping property of
the operator $K_{1}$(note that $\vartheta_{1}$ is in the kernel of $K_{1}$),
for some $\sigma\in\left(  0,1\right)  ,$
\begin{equation}
\left\vert \beta\right\vert +\left\Vert \left(  1+r_{1}^{\alpha}\right)
\rho\right\Vert _{C^{2,\sigma}}\leq C\left\Vert \left(  1+r_{1}^{2+\alpha
}\right)  \int_{\mathbb{R}}\eta^{+}H^{\prime}\mathcal{L}\varphi dz_{1}%
\right\Vert _{C^{0,\sigma}}. \label{beta}%
\end{equation}
On the other hand, by the equation $\left(  \mathcal{L}\varphi\right)  ^{\bot
}=-K_{2}\xi,~$%
\[
\mathcal{L}\varphi=-K_{2}\xi+\left(  \mathcal{L}\varphi\right)  ^{\Vert}.
\]
Note that due to the orthogonality condition, $\left(  \mathcal{L}%
\varphi\right)  ^{\Vert}$ is small compared to $\varphi.$ Hence by the a
priori estimate of the operator $\mathcal{L}$, suitable weighted norm of
$\varphi$ could be controlled by
\[
o\left(  \left\vert \beta\right\vert \right)  +o\left\Vert \left(
1+r_{1}^{\alpha}\right)  \rho\right\Vert _{C^{2,\sigma}}.
\]
This together with $\left(  \ref{beta}\right)  $ yields that $g=0,$ which
implies $\Xi_{0}^{\prime}=\Xi_{0}.$ Hence the operator $\mathcal{L}%
:S_{1}\oplus\mathcal{D}\rightarrow S_{2}$ has at most one dimensional kernel.

\textbf{Step 3. }To finish the proof, it remains to show that the solutions
$\Xi$ in step 1 and $\Xi_{0}$ in step 2 indeed belong to the space
$S_{1}\oplus\mathcal{D}.$

Recall that the deficiency space $\mathcal{D}$ is spanned by $\gamma_{1}$ and
$\gamma_{2}.$ Consider the function $\mathcal{L}\gamma_{1},i=1,2.$ We know
that $\mathcal{L}\gamma_{2}\in S_{2}.$ Hence by Step 1, one could find a
solutions $g_{i}$ satisfy%
\[
\mathcal{L}g_{i}=\mathcal{L}\gamma_{i}.
\]
Note that
\begin{equation}
g_{i}-c_{i}\hat{\Xi}\in S_{1} \label{gi}%
\end{equation}
for some constants $c_{i},$ where the function $\hat{\Xi}$ is a non-decaying
term introduced in $\left(  \ref{nonde}\right)  .$

Hence by Step 2 and the asymptotic behavior of $g_{i}$ and $\gamma_{i},$
$g_{i}-\gamma_{i}=d_{i}\Xi_{0},i=1,2,$ for some constants $d_{i}.$ Then by
$\left(  \ref{gi}\right)  ,$
\[
c_{i}\hat{\Xi}-\gamma_{i}=d_{i}\Xi_{0},i=1,2.
\]
It follows that $\hat{\Xi}-k_{1,1}\gamma_{1}-k_{1,2}\gamma_{2}\in S_{1}$ and
$\bar{\Xi}_{0}-k_{2,1}\gamma_{1}+k_{2,2}\gamma_{2}\in S_{1},$ for some
constants $k_{i,j},i,j=1,2.$ The proof is completed.
\end{proof}

Having proved the Fredholm property, we proceed to show that the operators
involved are real analytic.

\begin{lemma}
The map $N$ is real analytic.
\end{lemma}

\begin{proof}
This follows from the fact that $\Delta$ is a linear operator and the function
$u^{3}-u$ is a real analytic function. Note that the subtle point here is that
the definition of $N$ involves the diffeomorphisms $\Phi_{s,t},\Psi_{s,t}.$
These maps are certainly not real analytic with respect to the $r,z$
variables, since there is a cutoff function appeared in their definition.
However, these diffeomorphisms are indeed real analytic with respect to the
parameters $s$ and $t,$ which could be seem from the explicit expression
$\left(  \ref{Fermi2}\right)  $ of the Fermi coordinate(Notice that $f$
depends analytically on $k$ and $b$). Indeed, for a point $p=\left(
r,z\right)  $, by definition,
\[
\Psi_{s,t}\left(  p\right)  =\eta p_{s,t}+\left(  1-\eta\right)  p.
\]
Recall that the Fermi coordinate of $p_{s,t}$ with respect to $f_{k+s,b+t}$ is
equal to $\left(  r_{1},z_{1}\right)  $, which is the Fermi coordinate of $p$
with respect to the curve $f_{k+s,b+t}.$ Hence we have the relations
\[
\left\{
\begin{array}
[c]{c}%
r=r_{1}-\frac{z_{1}f_{k,b}^{\prime}}{\sqrt{1+\left(  f_{k,b}^{\prime}\right)
^{2}}},\\
z=f_{k,b}\left(  r_{1}\right)  +\frac{z_{1}}{\sqrt{1+\left(  f_{k,b}^{\prime
}\right)  ^{2}}},
\end{array}
\right.
\]
and%
\[
\left\{
\begin{array}
[c]{c}%
\bar{r}=r_{1}-\frac{z_{1}f_{k+s,b+t}^{\prime}}{\sqrt{1+\left(  f_{k+s,b+t}%
^{\prime}\right)  ^{2}}},\\
\bar{z}=f_{k+s,b+t}\left(  r_{1}\right)  +\frac{z_{1}}{\sqrt{1+\left(
f_{k+s,b+t}^{\prime}\right)  ^{2}}}.
\end{array}
\right.
\]
The real analyticity follows from these relations.
\end{proof}

\begin{definition}
A solution $u$ is nondegenerate, if and only if the linearized operator
$\mathcal{L}:S_{1}\oplus\mathcal{D}\rightarrow S_{2}$ is surjective.
\end{definition}

By the results of \cite{MR3281950} and \cite{MR3019512}, nondegenerate two-end
solutions do exist.

\begin{proposition}
\label{th}The set of solutions to $\left(  \ref{axial}\right)  $ satisfying
$\left(  \ref{asym}\right)  $ has a structure of real analytic variety of
formal dimension 1. Furthermore, if a solution $u$ is nondegenerate, then
locally around $u,$ the solution set is a one dimensional real analytic manifold.
\end{proposition}

\begin{proof}
The function $u_{s,t,\psi}$ is a solution of the Allen-Cahn equation, if and
only if
\begin{equation}
N\left(  s,t,\psi\right)  =0. \label{non}%
\end{equation}
Since $N\left(  0,0,0\right)  =0,$ we write equation $\left(  \ref{non}%
\right)  $ in the form
\[
DN\left(  0,0,0\right)  \left(  s,t,\psi\right)  +\int_{0}^{1}\left[
DN\left(  ls,lt,l\psi\right)  -DN\left(  0,0,0\right)  \right]  \left(
s,t,\psi\right)  dl=0.
\]
Then the result follows from the fact that $N$ is real analytic and of
Fredholm index $1$ and the real analytic implicit function theorem(for
example, see \cite{MR1956130}).
\end{proof}

\section{\label{Uniqueness}Analysis of solutions on the boundary of the moduli
space}

Two different types of two-end solutions to  equation $\left(
\ref{axial}\right)  $ have been constructed using Lyapunov-Schmidt reduction
method in \cite{MR3281950} and \cite{MR3019512}. Let us briefly describe these
solutions. The first type of solutions is constructed in \cite{MR3281950} and
has the property that their nodal curves are close to suitable scaling of a
solution to the Toda system. We call them Toda type solutions. The growth rate
of these solutions is close to $\sqrt{2}$ (but greater than $\sqrt{2}$). The
second class of solutions are those constructed in \cite{MR3019512}. Their
nodal sets are close to the catenoids, which we know are described by the
function $\varepsilon r=\cosh\left(  \varepsilon z\right)  ,$ where
$\varepsilon$ is a small parameter. We call them catenoid type solution. The
growth rate of these solutions are of the order $\varepsilon^{-1},$ hence
tends to infinity as $\varepsilon\rightarrow0.$

As we discussed in Section 1, we expect that the moduli space of two-end
solutions is diffeomorphic to $\mathbb{R}.$ In particular, we expect that
there exists a one-parameter family of solutions, at one end of this
family(the \textquotedblleft boundary\textquotedblright\ of the moduli space),
the solutions should be the Toda type solutions, while on the other end of the
moduli space, the solutions should be the catenoid type solutions.

In this section, we would like to analyze the solutions near the boundary\ of
the moduli space. Our purpose is to prove that if $\mathcal{P}_{u}$(recall
that $\mathcal{P}_{u}$ is the intersection point of the nodal set of $u$ with
the coordinate axes) is on the $z$ axis and $\left\vert \mathcal{P}%
_{u}\right\vert $ is large, then the growth rate of $u$ is close to $\sqrt{2}%
$. (with additional efforts, one could also show that $u$ is actually a Toda
type solution, but the proof of Theorem \ref{main} don't need this fact.) We
shall also show that if $\mathcal{P}_{u}$ is on the $r$ axis and $\left\vert
\mathcal{P}_{u}\right\vert $ is large then $u$ is a catenoid type solution.

\subsection{Analysis of Toda type solutions}

We shall first analyze the solutions whose nodal set has two components which
are very far away from each other. We expect that these solutions are Toda
type. Our main result here is

\begin{proposition}
\label{To}Let $u$ be two-end solution. Suppose $\mathcal{P}_{u}$ is on the $z$
axis and $\left\vert \mathcal{P}_{u}\right\vert $ is large. Then the growth
rate of $u$ is close to $\sqrt{2}.$
\end{proposition}

We use $q\left(  \cdot\right)  =q_{\varepsilon}\left(  \cdot\right)  $ to
denote the solution of the following Toda equation:%
\begin{equation}
\mathbf{c}_{0}q^{\prime\prime}+\frac{\mathbf{c}_{0}}{r}q^{\prime}%
-\mathbf{c}_{1}e^{-2\sqrt{2}q}=0,\text{ }q^{\prime}\left(  0\right)  =0.
\label{Toda}%
\end{equation}
Observe that explicitly,
\[
q_{\varepsilon}\left(  r\right)  =\mathtt{q}\left(  \varepsilon r\right)
-\frac{\sqrt{2}}{2}\ln\varepsilon
\]
for some $\varepsilon>0,$ where $\mathtt{q}\left(  r\right)  =\frac{1}%
{2\sqrt{2}}\ln\frac{\left(  1+ar^{2}\right)  ^{2}}{8}$ with $a=\frac{2\sqrt
{2}\mathbf{c}_{1}}{\mathbf{c}_{0}}.$ In particular, $q_{\varepsilon}\left(
r\right)  -\sqrt{2}\ln r-C_{\varepsilon}$ tends to $0$ as $r$ tends to
infinity with $C_{\varepsilon}$ a constant depending on $\varepsilon.$ In the
sequel, we choose $\varepsilon$ such that $q\left(  0\right)  =\left\vert
\mathcal{P}_{u}\right\vert .$

We shall follow similar notations as that of Section \ref{compactness}. For
example, the nodal curve of $u$ in the upper half plane will be the graph of
function $f.$ The Fermi coordinate with respect to the graph of $f$ will be
denoted by $\left(  r_{1},z_{1}\right)  .$ We also have the cutoff functions
$\eta,\eta^{+},$ and the solution $u$ will be written the form $u=\bar{u}%
+\phi,$ where $\bar{u}$ is an approximate solution:
\[
\bar{u}=\mathcal{H}_{1}+\mathcal{H}_{1}^{s}+1,
\]
with the function $\mathcal{H}_{1}$ defined similarly as that of Section
\ref{compactness} using the cutoff function $\eta$ and the heteroclinic
solution $H.$

The main idea of the proof is to compare $f$ with the solution $q_{\varepsilon
}$ of $\left(  \ref{Toda}\right)  ,$ by analyzing the equation satisfied by
$f.$ The main step will be establishing suitable decay estimate for the
function $\phi,$ as we have already done in Section \ref{compactness}. Our
starting point is the fact that if $\mathcal{P}_{u}$ is on the $z$ axis and
$\left\vert \mathcal{P}_{u}\right\vert $ is large$,$ then $\left\Vert
f^{\prime}\right\Vert _{L^{\infty}\left(  \left(  0,+\infty\right)  \right)
}$ will be small$.$ This follows from an application of the balancing formula.

We shall get a $L^{\infty}$ estimate for $\phi.$ In the following, $\left\Vert
\cdot\right\Vert _{\infty}$ stands for the $L^{\infty}\left(  \left(
0+\infty\right)  \right)  $ norm. To simplify the notations, we only consider
the case that the radius of the Fermi coordinate is large enough.

\begin{lemma}
\label{fi1}Suppose $u$ satisfies the assumption of Proposition \ref{To}. Then%
\[
\left\Vert \phi\right\Vert _{\infty}\leq C\left\Vert \frac{f^{\prime2}}{r^{2}%
}\right\Vert _{\infty}+Ce^{-2\sqrt{2}f\left(  0\right)  }.
\]

\end{lemma}

\begin{proof}
We only sketch the proof, since many computations will be similar to that of
Section \ref{compactness}.

Recall that $\phi$ satisfies the equation
\begin{equation}
L\phi=\left[  E\left(  \bar{u}\right)  \right]  ^{\parallel}+\left[  E\left(
\bar{u}\right)  \right]  ^{\perp}+P\left(  \phi\right)  . \label{L3}%
\end{equation}
$P\left(  \phi\right)  $ is a higher order term of $\phi,$ and $\left[
E\left(  \bar{u}\right)  \right]  ^{\parallel}=E\left(  \bar{u}\right)
_{1}^{\parallel}+E\left(  \bar{u}\right)  _{2}^{\parallel}$ is also small
compared to $\phi.$ As in Section \ref{compactness},
\begin{align}
\left[  E\left(  \bar{u}\right)  \right]  ^{\bot}  &  =\left[  E\left(
\mathcal{H}_{1}^{s}\right)  \right]  ^{\bot}+O\left(  f^{\prime\prime
2}\right)  +O\left(  r^{-2}f^{\prime2}\right)  +O\left(  h^{\prime\prime
}f^{\prime\prime}\right)  +O\left(  h^{\prime}r^{-1}\right) \label{ver}\\
&  +O\left(  h^{\prime\prime2}\right)  +O\left(  h^{\prime2}\right)  +O\left(
h^{\prime}f^{\prime\prime}\right)  +O\left(  h^{\prime}f^{\prime\prime\prime
}\right)  +O\left(  e^{-\sqrt{2}D}\right)  .\nonumber
\end{align}
Project  equation $\left(  \ref{L3}\right)  $ onto $\eta^{+}H_{1}^{\prime
},$ we could show
\begin{align}
&  \frac{\mathbf{c}_{0}f^{\prime\prime}}{\left(  1+f^{\prime2}\right)
^{\frac{3}{2}}}+\frac{\mathbf{c}_{0}h^{\prime\prime}}{1+f^{\prime2}}%
+\frac{\mathbf{c}_{0}f^{\prime}}{r_{1}\sqrt{1+f^{\prime2}}}-\left(  1+o\left(
1\right)  \right)  \mathbf{c}_{1}e^{-\sqrt{2}D}\label{fderivatives}\\
&  =O\left(  h^{\prime\prime}f^{\prime\prime}\right)  +O\left(  h^{\prime
2}\right)  +O\left(  h^{\prime}f^{\prime\prime}\right)  +O\left(  h^{\prime
}r^{-1}\right)  +O\left(  h^{\prime}f^{\prime\prime\prime}\right) \nonumber\\
&  +O\left(  \left\Vert \phi^{\ast}\left(  r_{1},\cdot\right)  \right\Vert
_{\infty}^{2}\right)  +O\left(  \left\Vert \phi^{\ast}\left(  r_{1}%
,\cdot\right)  \right\Vert _{\infty}e^{-\sqrt{2}f\left(  r_{1}\right)
}\right)  +O\left(  \left\Vert \phi^{\ast}\left(  r_{1},\cdot\right)
\right\Vert _{\infty}\left\Vert \partial_{r_{1}}\phi^{\ast}\left(  r_{1}%
,\cdot\right)  \right\Vert _{\infty}\right) \nonumber\\
&  +O\left(  \left\Vert \partial_{r_{1}}\phi^{\ast}\left(  r_{1},\cdot\right)
\right\Vert _{\infty}e^{-\sqrt{2}f\left(  r_{1}\right)  }\right)  +O\left(
\left\Vert \partial_{r_{1}}\phi^{\ast}\left(  r_{1},\cdot\right)  \right\Vert
_{\infty}^{2}\right)  +O\left(  \left\Vert \phi^{\ast}\left(  r_{1}%
,\cdot\right)  \right\Vert _{\infty}\left\Vert \partial_{r_{1}}^{2}\phi^{\ast
}\left(  r_{1},\cdot\right)  \right\Vert _{\infty}\right) \nonumber\\
&  +O\left(  f^{\prime\prime}\left\Vert \partial_{z_{1}}\phi^{\ast}\left(
r_{1},\cdot\right)  \right\Vert _{\infty}\right)  +O\left(  f^{\prime\prime
}\left\Vert \partial_{r_{1}}\phi^{\ast}\left(  r_{1},\cdot\right)  \right\Vert
_{\infty}\right)  +O\left(  f^{\prime\prime\prime}\left\Vert \partial_{r_{1}%
}\phi^{\ast}\left(  r_{1},\cdot\right)  \right\Vert _{\infty}\right)
\nonumber\\
&  +O\left(  r^{-1}\left\Vert \partial_{r_{1}}\phi^{\ast}\left(  r_{1}%
,\cdot\right)  \right\Vert _{\infty}\right)  +O\left(  r^{-1}f^{\prime
}\left\Vert \partial_{z_{1}}\phi^{\ast}\left(  r_{1},\cdot\right)  \right\Vert
_{\infty}\right)  .\nonumber
\end{align}
$\ $Insert this into $\left(  \ref{ver}\right)  ,$ we obtain
\begin{align*}
\left\vert \left(  E\left(  \bar{u}\right)  \right)  ^{\bot}\right\vert  &
=\left[  E\left(  \mathcal{H}_{1}^{s}\right)  \right]  ^{\bot}+O\left(
r^{-2}f^{\prime2}\right)  +O\left(  h^{\prime}r^{-1}\right) \\
&  +O\left(  h^{\prime2}\right)  +O\left(  h^{\prime}f^{\prime\prime\prime
}\right)  +O\left(  e^{-\sqrt{2}D}\right) \\
&  +O\left(  h^{\prime}h^{\prime\prime}\right)  +O\left(  hf^{\prime
\prime\prime}\right)  +O\left(  h^{\prime\prime2}\right) \\
&  +O\left(  \left\Vert \phi^{\ast}\left(  r_{1},\cdot\right)  \right\Vert
_{\infty}^{2}\right)  +O\left(  \left\Vert \partial_{r_{1}}\phi^{\ast}\left(
r_{1},\cdot\right)  \right\Vert _{\infty}^{2}\right) \\
&  +O\left(  \left\Vert \partial_{z_{1}}\phi^{\ast}\left(  r_{1},\cdot\right)
\right\Vert _{\infty}^{2}\right)  .
\end{align*}
Then the a priori estimate of the operator $L$ tells us that
\[
\left\Vert \phi\right\Vert _{\infty}\leq\ C\left\Vert \frac{f^{\prime2}}%
{r^{2}}\right\Vert _{\infty}+Ce^{-2\sqrt{2}f\left(  0\right)  }.
\]

\end{proof}

Our next aim is to estimate the $L^{\infty}$ norm of $\frac{f^{\prime}\left(
r\right)  }{r}.$

\begin{lemma}
Let $\varepsilon$ be introduced above, then
\[
\left\Vert \frac{f^{\prime}}{r}\right\Vert _{\infty}+\left\Vert f^{\prime
\prime}\right\Vert _{\infty}\leq C\varepsilon^{2}.
\]

\end{lemma}

\begin{proof}
The starting point of the proof is still  equation $\left(
\ref{fderivatives}\right)  .$ Applying Lemma \ref{fi1}, we infer from $\left(
\ref{fderivatives}\right)  $ that
\begin{align}
&  \frac{\mathbf{c}_{0}f^{\prime\prime}}{\left(  1+f^{\prime2}\right)
^{\frac{3}{2}}}+\frac{\mathbf{c}_{0}f^{\prime}}{r_{1}\left(  1+f^{\prime
2}\right)  ^{\frac{1}{2}}}-\mathbf{c}_{1}e^{-\sqrt{2}D}\label{T1}\\
&  =O\left(  \varepsilon^{2}\right)  .\nonumber
\end{align}
Let $t_{0}$ be a point where
\[
\frac{f^{\prime}\left(  t_{0}\right)  }{t_{0}}=\left\Vert \frac{f^{\prime}}%
{r}\right\Vert _{\infty}.
\]
This point exists because $\frac{f^{\prime}}{r}\rightarrow0,$ as
$r\rightarrow+\infty.$

At the point $t_{0},$ by $\left(  \ref{T1}\right)  ,$ $f$ will satisfy
\begin{align}
&  \mathbf{c}_{0}f^{\prime\prime}\left(  t_{0}\right)  +\mathbf{c}_{0}%
\frac{f^{\prime}\left(  t_{0}\right)  }{t_{0}}-\mathbf{c}_{1}e^{-\sqrt
{2}D\left(  t_{0}\right)  }\label{t0}\\
&  =O\left(  \varepsilon^{2}\right)  .\nonumber
\end{align}
We claim
\begin{equation}
f^{\prime\prime}\left(  t_{0}\right)  \geq0. \label{fsecond}%
\end{equation}
Indeed, if this is not true, then due to the fact that $f^{\prime\prime
}\left(  0\right)  \geq0,$ there will be another point $t_{1}<t_{0}$ such that
$f^{\prime\prime}\left(  t_{1}\right)  =0,$ and $f^{\prime\prime}\left(
t\right)  <0,$ $t\in\left(  t_{1},t_{0}\right)  .$ Since $f^{\prime}\geq0,$ we
find
\[
\left(  \frac{f^{\prime}\left(  r\right)  }{r}\right)  ^{\prime}%
=\frac{f^{\prime\prime}-r^{-1}f^{\prime}}{r}<0,\text{ }r\in\left(  t_{1}%
,t_{0}\right)  .
\]
This contradicts with the fact that $\frac{f^{\prime}\left(  t_{0}\right)
}{t_{0}}=\left\Vert \frac{f^{\prime}}{r}\right\Vert _{\infty}.$ From $\left(
\ref{t0}\right)  $ and $\left(  \ref{fsecond}\right)  $, it may be concluded
that
\begin{align}
\frac{f^{\prime}\left(  t_{0}\right)  }{t_{0}}  &  =\frac{\mathbf{c}_{1}%
}{\mathbf{c}_{0}}e^{-\sqrt{2}D\left(  t_{0}\right)  }-f^{\prime\prime}\left(
t_{0}\right)  +O\left(  \varepsilon^{2}\right) \nonumber\\
&  \leq Ce^{-2\sqrt{2}f\left(  0\right)  }\leq C\varepsilon^{2}. \label{f'}%
\end{align}
Here we have used the fact that $D\left(  t_{0}\right)  \geq2f\left(
0\right)  .$ This proves estimate for the $L^{\infty}$ norm of $\frac
{f^{\prime}}{r}.$ The estimate for $f^{\prime\prime}$ is a direct consequence
of $\left(  \ref{f'}\right)  $ and $\left(  \ref{t0}\right)  .$
\end{proof}

Let $p=f+\sqrt{1+f^{\prime2}}h.$ Next we show that $p$ is indeed close to the
solution $q$ of the Toda equation in a large interval. Set $b=b_{\varepsilon
}=\left\vert \frac{\ln\varepsilon}{\varepsilon}\right\vert .$

\begin{lemma}
$\label{toda}$There exists $\alpha>0,$ such that
\[
\left\vert p\left(  r\right)  -q\left(  r\right)  \right\vert \leq
C\varepsilon^{\alpha},\text{ }r\in\left(  0,b\right)  .
\]

\end{lemma}

\begin{proof}
Denote by $\mathbf{e}$ the function $p\left(  \varepsilon^{-1}r\right)
+\frac{\sqrt{2}}{2}\ln\varepsilon-\mathtt{q}\left(  r\right)  .$ Then
$\mathbf{e}\left(  0\right)  =O\left(  \varepsilon^{2}\right)  $ and
$\mathbf{e}^{\prime}\left(  0\right)  =0.$ By the previous lemmas, $\left\Vert
\phi\right\Vert _{\infty}\leq\varepsilon^{2}.$ Using this fact, we see that
the function $p=\sqrt{1+f^{\prime2}}h+f$ satisfies the equation
\[
\frac{\mathbf{c}_{0}p^{\prime\prime}}{\left(  1+p^{\prime2}\right)  ^{\frac
{3}{2}}}+\frac{\mathbf{c}_{0}p^{\prime}}{r_{1}\left(  1+p^{\prime2}\right)
^{\frac{1}{2}}}-\mathbf{c}_{1}e^{-\sqrt{2}D}=O\left(  \varepsilon^{2+\alpha
}\right)  .
\]
In particular, this combined with the fact that $D-2f=O\left(  \varepsilon
^{\alpha}\right)  $ yields
\[
p^{\prime\prime}+\frac{p^{\prime}}{r_{1}}-e^{-2\sqrt{2}p}=O\left(
\varepsilon^{2+\alpha}\right)  .
\]
From this equation, we infer that in the region where $\mathbf{e}$ is
$o\left(  1\right)  ,$ $\mathbf{e}$ satisfies
\[
\mathbf{e}^{\prime\prime}+\frac{\mathbf{e}^{\prime}}{r}-e^{-2\sqrt
{2}\mathtt{q}}\mathbf{e}=O\left(  \mathbf{e}^{2}\right)  +O\left(
\varepsilon^{\alpha}\right)  .
\]
The conclusion of the lemma then follows from the variation of parameters formula.
\end{proof}

By definition $q\left(  r\right)  =\mathtt{q}\left(  \varepsilon r\right)
-\frac{\sqrt{2}}{2}\ln\varepsilon,$ hence $rq^{\prime}\left(  r\right)
=\varepsilon r\mathtt{q}^{\prime}\left(  \varepsilon r\right)  ,$ which
implies that
\[
bq^{\prime}\left(  b\right)  =\sqrt{2}+o\left(  1\right)  .
\]
Notice that in Lemma \ref{toda}, we actually could also estimate $\left\vert
\mathbf{e}^{\prime}\right\vert \leq C\varepsilon^{\alpha}.$ From this, we
infer that
\begin{equation}
bf^{\prime}\left(  b\right)  =\sqrt{2}+o\left(  1\right)  . \label{fb}%
\end{equation}

Now we are in a position to prove the main result of this section.

\begin{proof}
[Proof of Proposition \ref{To}]Similar arguments as before yields
\[
\left\vert \phi\left(  r,z\right)  \right\vert \leq C\frac{1}{1+r^{2}%
}+Ce^{-D\left(  r\right)  }.
\]
Equation $\left(  \ref{fderivatives}\right)  $ then becomes
\[
\frac{p^{\prime\prime}}{\left(  1+p^{\prime2}\right)  ^{\frac{3}{2}}}%
+\frac{p^{\prime}}{r_{1}\left(  1+p^{\prime2}\right)  ^{\frac{1}{2}}}=\left(
1+o\left(  1\right)  \right)  e^{-\sqrt{2}D\left(  r_{1}\right)  }+O\left(
r_{1}^{-3}\right)  .
\]
Integrating from $t_{0}$ to $t_{1}$ leads to
\[
\frac{t_{1}p^{\prime}\left(  t_{1}\right)  }{\left(  1+p^{\prime}\left(
t_{1}\right)  ^{2}\right)  ^{\frac{1}{2}}}-\frac{t_{0}p^{\prime}\left(
t_{0}\right)  }{\left(  1+p^{\prime}\left(  t_{0}\right)  ^{2}\right)
^{\frac{1}{2}}}=\left(  1+o\left(  1\right)  \right)  \int_{t_{0}}^{t_{1}%
}e^{-\sqrt{2}D\left(  s\right)  }sds+O\left(  t_{0}^{-1}\right)  .
\]
This together with $\left(  \ref{fb}\right)  $ tells us that
\[
p^{\prime}\left(  r\right)  r\geq\sqrt{2}+o\left(  1\right)  \text{ for }r>b.
\]
Let $\delta>0$ be a fixed small constant. We claim that when $\varepsilon$ is
small,
\begin{equation}
p^{\prime}\left(  r\right)  r\leq\sqrt{2}+\delta\text{ for }r>b. \label{f3}%
\end{equation}
Indeed, suppose $\left(  b,b^{\ast}\right)  $ is the maximal interval where
$p^{\prime}\left(  r\right)  r\leq\sqrt{2}+\delta.$ Then in this interval,
elementary geometrical facts implies that
\begin{align*}
D\left(  r\right)   &  \geq\sqrt{2}\left\vert \ln\varepsilon\right\vert
+2\mathtt{q}\left(  b\right)  +2\left(  \sqrt{2}+o\left(  1\right)  \right)
\int_{b}^{r}\frac{ds}{s}-C\\
&  =\sqrt{2}\left\vert \ln\varepsilon\right\vert +2\mathtt{q}\left(  b\right)
+2\left(  \sqrt{2}+o\left(  1\right)  \right)  \left(  \ln r-\ln b\right)  -C.
\end{align*}
Therefore we have
\begin{align*}
\frac{b^{\ast}p^{\prime}\left(  b^{\ast}\right)  }{\left(  1+p^{\prime}\left(
b^{\ast}\right)  ^{2}\right)  ^{\frac{1}{2}}}-\frac{bp^{\prime}\left(
b\right)  }{\left(  1+p^{\prime}\left(  b\right)  ^{2}\right)  ^{\frac{1}{2}%
}}  &  =\left(  1+o\left(  1\right)  \right)  \int_{b}^{b^{\ast}}e^{-\sqrt
{2}D\left(  s\right)  }sds+O\left(  b^{-1}\right) \\
&  =O\left(  b^{-1}\right)  .
\end{align*}
This implies that when $\varepsilon$ is small, $b^{\ast}=+\infty.$ Applying
$\left(  \ref{f3}\right)  ,$ we finally get $\lim_{r\rightarrow+\infty
}p^{\prime}\left(  r\right)  r=\allowbreak\sqrt{2}+o\left(  1\right)  .$ The
proof is thus completed.
\end{proof}

\subsection{Uniqueness of catenoid type solutions}

Beside the planes, catenoid is the first example of embedded minimal surfaces
with finite total curvature; it is rotationally symmetric with respect to its
axis and the only minimal surface of revolution(up to a homothety). In the
$\left(  r,z\right)  $ coordinate, the one-parameter family of catenoids
$\mathcal{C}_{\varepsilon}$ could be represented by the function
\[
\varepsilon r=\cosh\left(  \varepsilon z\right)  ,
\]
with $\varepsilon>0$ being the parameter. As we mentioned before, in
\cite{MR3019512}, for each $\varepsilon$ sufficiently small, a solution
$u_{\varepsilon}$ of the Allen-Cahn equation is constructed. The nodal set of
$u_{\varepsilon}$ is close to the catenoid $\mathcal{C}_{\varepsilon}$. In
particular, $\mathcal{P}_{u_{\varepsilon}}$ is on the $r$ axis, $\left\vert
\mathcal{P}_{u_{\varepsilon}}\right\vert $ is of the order $O\left(
\varepsilon^{-1}\right)  ,$ and the growth rate of $u_{\varepsilon}$ is also
of the order $O\left(  \varepsilon^{-1}\right)  .$

In this section, we wish to prove that this one-parameter family of solutions
$u_{\varepsilon}$ is unique. This is the content of the following

\begin{proposition}
\label{uniq}Let $u$ be a two-end solution of the Allen-Cahn equation. Suppose
$\mathcal{P}_{u}$ is on the $r$ axis and $\left\vert \mathcal{P}%
_{u}\right\vert $ is large. Then there exists a small $\varepsilon^{\prime}>0$
such that $u=u_{\varepsilon^{\prime}}.$
\end{proposition}

The nodal curve $\mathcal{N}_{u}$ in $\mathbb{E}^{+}$ could be written as
\[
\left\{  \left(  r,z\right)  :z=f\left(  r\right)  \right\}  ,
\]
where $f$ is a function
\[
f:[\left\vert \mathcal{P}_{u}\right\vert ,+\infty)\rightarrow\mathbb{R}.
\]
Note that by the results in the previous section, $f$ is asymptotic to
$c_{1}\ln r+c_{2}$ as $r\rightarrow+\infty$ for some constants $c_{1}$ and
$c_{2}.$ One could also write $\mathcal{N}_{u}\cap\mathbb{E}^{+}=\left\{
\left(  r,z\right)  :r=g\left(  z\right)  \right\}  $ for some function
$g:[0,+\infty)\rightarrow\mathbb{R}.$

For convenience we introduce the parameter $\varepsilon=\left\vert
\mathcal{P}_{u}\right\vert ^{-1}.$ Observe that by the assumption of
Proposition \ref{uniq}, $\varepsilon$ is small, thus by the validity of De
Giorgi conjecture in $\mathbb{R}^{3},$ locally around the nodal curve, $u$
looks like the heteroclinic solution.

Let $l$ be a large but fixed constant. As a preliminary step, we would like to
get some rough information about the slope of the function $f.$

\begin{lemma}
\label{sl}We have
\[
f^{\prime}\left(  r\right)  >\frac{C}{l}+o\left(  1\right)  ,\text{ }%
r\in\left[  \varepsilon^{-1},l\varepsilon^{-1}\right]  .
\]

\end{lemma}

\begin{proof}
Let $X=\left(  0,0,1\right)  $ be a constant vector field$.$ We have
\[
\int_{0}^{l\varepsilon^{-1}}\left[  \frac{1}{2}u_{r}^{2}+F\left(  u\right)
\right]  rdr\geq Cl\varepsilon^{-1}.
\]
On the other hand,
\begin{align*}
&  \int_{\partial\Omega\cap\left\{  z>0\right\}  }\left\{  \left[  \frac{1}%
{2}u_{r}^{2}+F\left(  u\right)  \right]  X\cdot v-\left(  \nabla u\cdot
X\right)  \left(  \nabla u\cdot v\right)  \right\}  dS\\
&  \sim f^{\prime}\left(  l\varepsilon^{-1}\right)  l\varepsilon^{-1}+o\left(
1\right)  l\varepsilon^{-1}.
\end{align*}
Combine these two estimates and use the balancing formula, we get the desired result.
\end{proof}

To get more precise information, we again need to work in the Fermi coordinate
$\left(  s,t\right)  $ around the nodal curve $\mathcal{N}_{u},$ where $s$ is
the signed distance to $\mathcal{N}_{u}$ and $t$ is a parametrization of
$\mathcal{N}_{u}$.

We slightly abuse the notation and still use $\mathcal{B}_{u}$ to denote the
maximal domain where the Fermi coordinate is well-defined. Similarly as
before, let $\eta$ be a cutoff function supported in $\mathcal{B}_{u},$ with
$\nabla\eta$ supported near the boundary of $\mathcal{B}_{u}.$ Introduce the
approximate solution
\[
\bar{u}\left(  r,z\right)  =\eta\mathcal{H}+\left(  1-\eta\right)
\frac{\mathcal{H}}{\left\vert \mathcal{H}\right\vert },
\]
where $\mathcal{H}$ is defined through $X^{\ast}\mathcal{H}\left(  s,t\right)
=H\left(  s-h\left(  t\right)  \right)  .$ Here $X$ is the map $\left(
s,t\right)  \rightarrow\left(  r,z\right)  .$ Write $u=\bar{u}\mathcal{+}\phi
$. The small function $h$ is chosen such that $\phi$ satisfies the orthogonal
condition%
\[
\int_{\mathbb{R}}X^{\ast}\left(  \eta\phi\mathcal{H}^{\prime}\right)  ds=0,
\]
where $X^{\ast}\mathcal{H}^{\prime}\left(  s,t\right)  =H^{\prime}\left(
s-h\left(  t\right)  \right)  .$

To analyze the solution $u,$ we firstly study $\mathcal{N}_{u}$ in the region
where $r\in\left(  \varepsilon^{-1},l\varepsilon^{-1}\right)  .$ In this
region, many calculations will be explicit to in the Fermi coordinate $\left(
x,y\right)  $ with respect to the graph of the function $g,$ which is
\[
\left\{
\begin{array}
[c]{l}%
r=g\left(  y\right)  +\frac{x}{\sqrt{1+g^{\prime2}}},\\
z=y-\frac{xg^{\prime}}{\sqrt{1+g^{\prime2}}}.
\end{array}
\right.
\]
Hence $x$ is the signed distance and $y$ is a parametrization of the curve.

We shall estimate the $L^{\infty}$ norm of the perturbation term $\phi.$

\begin{lemma}
\label{fi2}Suppose $u$ satisfies the assumption of Proposition \ref{uniq} and
$\phi$ is defined above. Then
\[
\left\Vert \phi\right\Vert _{\infty}\leq C\varepsilon^{2}.
\]

\end{lemma}

\begin{proof}
We only consider the case that the size of the Fermi coordinate is large enough.  The general case could be handled using the arguments of Section
\ref{compactness}. Again many computations here are similar as before.

We need to analyze $E\left(  \bar{u}\right)  .$ First of all, consider the
case $r\in\left(  \varepsilon^{-1},l\varepsilon^{-1}\right)  .$ In the Fermi
coordinate, the error of the approximate solution has the form%
\begin{align*}
E\left(  \bar{u}\right)   &  =\frac{H^{\prime\prime}h^{\prime2}-H^{\prime
}h^{\prime\prime}}{A}+\left(  \frac{\partial_{x}A}{2A}+\frac{\partial_{r}x}%
{r}\right)  H^{\prime}\\
&  +\left(  \frac{\partial_{y}A}{2A^{2}}-\frac{\partial_{r}y}{r}\right)
H^{\prime}h^{\prime}.
\end{align*}
Keep in mind that $H^{\prime}$ is evaluated at $x-h\left(  y\right)  ,$ not
$x.$ By Lemma \ref{sl}, $\left\vert g^{\prime}\right\vert \leq C,$ hence
\[
\frac{\partial_{r}x}{r}=\frac{\frac{1}{\sqrt{1+g^{\prime2}}}}{g+\frac{x}%
{\sqrt{1+g^{\prime2}}}}=\frac{1}{g\sqrt{1+g^{\prime2}}}-\frac{x}{g^{2}\left(
1+g^{\prime2}\right)  }+O\left(  g^{-3}\right)  .
\]
From $\left\vert g^{\prime\prime}\right\vert =o\left(  1\right)  $, we infer
that
\[
\frac{\partial_{x}A}{2A}=-\frac{g^{\prime\prime}}{\left(  1+g^{\prime
2}\right)  ^{\frac{3}{2}}}+\frac{\left(  g^{\prime\prime}\right)  ^{2}%
x}{\left(  1+g^{\prime2}\right)  ^{3}}+O\left(  g^{\prime\prime3}\right)  .
\]
It follows from these expansions that the projection of $E\left(  \bar
{u}\right)  $ onto $\mathcal{H}^{\prime}$ has the form
\begin{align*}
\int_{\mathbb{R}}X^{\ast}\left(  \eta\mathcal{H}^{\prime}E\left(  \bar
{u}\right)  \right)  dx  &  =-\frac{\mathbf{c}_{0}g^{\prime\prime}}{\left(
1+g^{\prime2}\right)  ^{\frac{3}{2}}}+\frac{\mathbf{c}_{0}}{g\sqrt
{1+g^{\prime2}}}\\
&  -h^{\prime\prime}\int_{\mathbb{R}}\frac{H^{\prime2}}{A}+O\left(
h^{\prime2}\right)  +O\left(  h^{\prime}g^{\prime\prime}\right) \\
&  +O\left(  h^{\prime}g^{\prime\prime\prime}\right)  +O\left(  g^{\prime
\prime3}\right)  +O\left(  g^{-3}\right) \\
&  +O\left(  hg^{\prime\prime2}\right)  +O\left(  hg^{-2}\right)  .
\end{align*}
Set $\tilde{h}\left(  y\right)  =\sqrt{1+g^{\prime2}}h\left(  y\right)  $ and
$p_{1}\left(  y\right)  =g\left(  y\right)  +\tilde{h}\left(  y\right)  .$
Then
\begin{align*}
&  -\frac{g^{\prime\prime}}{1+g^{\prime2}}+\frac{1}{g}-\frac{h^{\prime\prime}%
}{\sqrt{1+g^{\prime2}}}\\
&  =-\frac{g^{\prime\prime}}{1+g^{\prime2}}+\frac{1}{g}-\frac{\tilde
{h}^{\prime\prime}}{1+g^{\prime2}}-\frac{2\tilde{h}^{\prime}}{\sqrt
{1+g^{\prime2}}}\left(  \frac{1}{\sqrt{1+g^{\prime2}}}\right)  ^{\prime}\\
&  -\frac{\tilde{h}}{\sqrt{1+g^{\prime2}}}\left(  \frac{1}{\sqrt{1+g^{\prime
2}}}\right)  ^{\prime\prime}\\
&  =-\frac{p_{1}^{\prime\prime}}{1+p_{1}^{\prime2}}+\frac{1}{p_{1}}+O\left(
g^{\prime\prime}h\right)  +O\left(  g^{\prime\prime}h^{\prime}\right)
+O\left(  g^{\prime\prime\prime}h\right)  +O\left(  g^{\prime\prime\prime
}h^{\prime}\right) \\
&  +O\left(  g^{\prime\prime}h\right)  +O\left(  g^{\prime\prime}h^{\prime
}\right)  +O\left(  h^{2}\right)  +O\left(  hh^{\prime}\right)  +O\left(
g^{-1}h\right)  .
\end{align*}
It follows that
\begin{align*}
\sqrt{1+g^{\prime2}}\int_{\mathbb{R}}X^{\ast}\left(  \eta\mathcal{H}^{\prime
}E\left(  \bar{u}\right)  \right)  dx  &  =-\frac{\mathbf{c}_{0}p_{1}%
^{\prime\prime}}{1+p_{1}^{\prime2}}+\frac{\mathbf{c}_{0}}{p_{1}}\\
&  +O\left(  g^{\prime\prime}h\right)  +O\left(  g^{\prime\prime}h^{\prime
}\right)  +O\left(  g^{\prime\prime\prime}h\right)  +O\left(  g^{\prime
\prime\prime}h^{\prime}\right) \\
&  +O\left(  g^{\prime\prime}h\right)  +O\left(  g^{\prime\prime}h^{\prime
}\right)  +O\left(  h^{2}\right)  +O\left(  hh^{\prime}\right)  +O\left(
g^{-1}h\right) \\
&  +O\left(  h^{\prime2}\right)  +O\left(  g^{\prime\prime3}\right)  +O\left(
g^{-3}\right)  +O\left(  h^{\prime}h^{\prime\prime}\right)  .
\end{align*}
On the other hand, we could also show that
\begin{equation}
\int_{\mathbb{R}}X^{\ast}\left(  \eta\mathcal{H}^{\prime}E\left(  \bar
{u}\right)  \right)  dx=o\left(  \left\Vert \phi^{\ast}\left(  y,\cdot\right)
\right\Vert _{C^{2}}\right)  +o\left(  \left\Vert \partial_{y}\phi^{\ast
}\left(  y,\cdot\right)  \right\Vert _{C^{0}}\right)  +o\left(  \left\Vert
\partial_{y}^{2}\phi^{\ast}\left(  y,\cdot\right)  \right\Vert _{C^{0}%
}\right)  . \label{equ1}%
\end{equation}
Here the norm is taken as a function of $x$ variable. As a consequence of the
above two equations, we have
\begin{align}
-\frac{p_{1}^{\prime\prime}}{1+p_{1}^{\prime2}}+\frac{1}{p_{1}}  &  =O\left(
h^{\prime2}\right)  +O\left(  h^{\prime}g^{\prime\prime}\right)  +O\left(
h^{\prime}g^{\prime\prime\prime}\right)  +O\left(  g^{\prime\prime3}\right)
\nonumber\\
&  +O\left(  g^{-3}\right)  +O\left(  hg^{\prime\prime2}\right)  +O\left(
hg^{-2}\right)  +O\left(  h^{\prime}h^{\prime\prime}\right) \nonumber\\
&  +o\left(  \left\Vert \phi^{\ast}\left(  y,\cdot\right)  \right\Vert
_{C^{2}}\right)  +o\left(  \left\Vert \partial_{y}\phi^{\ast}\left(
y,\cdot\right)  \right\Vert _{C^{0}}\right)  +o\left(  \left\Vert \partial
_{y}^{2}\phi^{\ast}\left(  y,\cdot\right)  \right\Vert _{C^{0}}\right)  .
\label{g}%
\end{align}

Next we consider the case $r>\varepsilon^{-1}l.$ In this case, one could use
the Fermi coordinate $\left(  r_{1},z_{1}\right)  $ with respect the curve
$z=f\left(  r\right)  .$ The analysis of $E\left(  \bar{u}\right)  $ in this
case is almost same as that of the previous sections and we omit the details.

Now we write the equation satisfied by $\phi$ into the form%
\[
L\phi=\left[  E\left(  \bar{u}\right)  \right]  ^{\bot}+\left[  E\left(
\bar{u}\right)  \right]  ^{\Vert}+P\left(  \phi\right)  ,
\]
where
\[
\left[  E\left(  \bar{u}\right)  \right]  ^{\Vert}=\frac{\int_{\mathbb{R}%
}X^{\ast}\left(  \eta\mathcal{H}^{\prime}E\left(  \bar{u}\right)  \right)
dx}{\int_{\mathbb{R}}\eta^{2}\mathcal{H}^{\prime2}dx}\eta\mathcal{H}^{\prime
}\text{ and }\left[  E\left(  \bar{u}\right)  \right]  ^{\bot}=E\left(
\bar{u}\right)  -\left[  E\left(  \bar{u}\right)  \right]  ^{\Vert}.
\]
In terms of Fermi coordinate, in the region where $\left\vert g^{\prime
}\right\vert \leq C,$
\[
\left[  E\left(  \bar{u}\right)  \right]  ^{\bot}=O\left(  \left\vert
p_{1}^{\prime\prime}\right\vert ^{2}\right)  +O\left(  p_{1}^{-2}\right)  .
\]
In the region where $\left\vert f^{\prime}\right\vert \leq C,$ $\left[
E\left(  \bar{u}\right)  \right]  ^{\bot}=O\left(  r_{1}^{-2}\right)  .$ By
the a priori estimate of $L,$ we could obtain
\[
\left\Vert \phi\right\Vert _{\infty}\leq\left\Vert \left[  E\left(  \bar
{u}\right)  \right]  ^{\bot}\right\Vert _{\infty}\leq C\varepsilon^{2}.
\]
We emphasize here that $\left\Vert \left[  E\left(  \bar{u}\right)  \right]
^{\bot}\right\Vert _{\infty}$ should be estimated in the whole plane. It is
also worth mentioning  that the term $O\left(  \left\vert p_{1}^{\prime\prime
}\right\vert ^{2}\right)  $ in $\left[  E\left(  \bar{u}\right)  \right]
^{\bot}$ should be handled using equation $\left(  \ref{g}\right)  .$
\end{proof}

Our next aim is to show that in the interval $\left(  \varepsilon
^{-1},l\varepsilon^{-1}\right)  ,$ the function $r=p_{1}\left(  z\right)  $ is
close to the function $r=\varepsilon^{-1}\cosh\left(  \varepsilon z\right)  .$
This is the content of the following

\begin{lemma}
\label{p1}For $z\in\left(  0,f\left(  l\right)  \varepsilon^{-1}\right)  ,$
\[
\left\vert p_{1}\left(  z\right)  -\varepsilon^{-1}\cosh\left(  \varepsilon
z\right)  \right\vert \leq C\varepsilon.
\]

\end{lemma}

\begin{proof}
From equation $\left(  \ref{g}\right)  $ and the estimate of $\phi,$ we deduce
that the function $p_{1}$ satisfies
\[
\frac{p_{1}^{\prime\prime}}{1+p_{1}^{\prime2}}-\frac{1}{p_{1}}=O\left(
\varepsilon^{3}\right)  .
\]
At this stage, we introduce the scaled function $\bar{p}_{1}\left(  z\right)
=\varepsilon p_{1}\left(  \varepsilon^{-1}z\right)  .$ Then
\begin{align*}
\bar{p}_{1}^{\prime}\left(  z\right)   &  =p_{1}^{\prime}\left(
\varepsilon^{-1}z\right)  ,\\
\bar{p}_{1}^{\prime\prime}\left(  z\right)   &  =\varepsilon^{-1}p_{1}%
^{\prime\prime}\left(  \varepsilon^{-1}z\right)  .
\end{align*}
It follows that
\begin{equation}
\bar{p}_{1}^{\prime\prime}-\frac{1+\bar{p}_{1}^{\prime2}}{\bar{p}_{1}%
}=O\left(  \varepsilon^{2}\right)  ,\text{ }z\in\left(  0,l\right)  .
\label{gb}%
\end{equation}
Observe that the function $\cosh z$ satisfies the equation
\begin{equation}
\left(  \cosh z\right)  ^{\prime\prime}-\frac{1+\left(  \cosh z\right)
^{\prime2}}{\cosh z}=0. \label{cosh}%
\end{equation}
Let $\omega\left(  z\right)  =\bar{p}_{1}\left(  z\right)  -\cosh z.$ Then
\begin{align*}
\omega\left(  0\right)   &  =\bar{p}_{1}\left(  0\right)  -1=\varepsilon
p_{1}\left(  0\right)  -1\\
&  =\varepsilon\left(  \varepsilon^{-1}+h\left(  0\right)  \right)
-1=O\left(  \varepsilon^{2}\right)  .
\end{align*}
We claim that $\left\vert \omega\left(  z\right)  \right\vert \leq
C\varepsilon^{2}$ for $z\in\left(  0,l\right)  .$ Indeed, subtracting equation
$\left(  \ref{gb}\right)  $ with $\left(  \ref{cosh}\right)  ,$ we get
\[
\omega^{\prime\prime}-2\tanh z\omega^{\prime}+\omega=O\left(  \varepsilon
^{2}\right)  +O\left(  \omega^{2}\right)  +O\left(  \omega^{\prime2}\right)
:=\psi\left(  z\right)  .
\]
Let $\xi_{1}$ and $\xi_{2}$ be two linearly independent solutions of the
homogeneous equation:
\[
\xi_{i}^{\prime\prime}-2\tanh z\xi_{i}^{\prime}+\xi_{i}=0,i=1,2.
\]
Explicitly, we can choose
\begin{align*}
\xi_{1}\left(  z\right)   &  =\cosh^{\prime}z=\sinh z,\\
\xi_{2}\left(  z\right)   &  =\partial_{\varepsilon}\left(  \varepsilon
^{-1}\cosh\varepsilon z\right)  |_{\varepsilon=1}=-\cosh z+z\sinh z.
\end{align*}
The Wronsky of these two solutions are
\[
W\left(  z\right)  :=\left\vert
\begin{array}
[c]{cc}%
\sinh z & \cosh z\\
-\cosh z+z\sinh z & z\cosh z
\end{array}
\right\vert =\cosh^{2}z.
\]
By the variation of parameters formula,
\[
\omega\left(  z\right)  =\xi_{2}\left(  z\right)  \int_{0}^{z}\frac{\xi
_{1}\left(  s\right)  \psi\left(  s\right)  }{W\left(  s\right)  }ds-\xi
_{1}\left(  z\right)  \int_{0}^{z}\frac{\xi_{2}\left(  s\right)  \psi\left(
s\right)  }{W\left(  s\right)  }ds+O\left(  \varepsilon^{2}\right)  .
\]
The desired estimate comes from this formula.
\end{proof}

\begin{proof}
[Proof of Proposition \ref{uniq}]Let $p_{2}=f+\sqrt{1+f^{\prime2}}h.$ For
$r>l\varepsilon^{-1},$ projecting $E\left(  \bar{u}\right)  $ on
$\eta\mathcal{H}^{\prime}$ and perform similar calculation as in Section
\ref{compactness}, we could estimate the perturbation $\phi$ in algebraically
weighted norm(Remember that $h$ could be controlled by $\phi$). This leads to
the equation:
\[
\left(  \frac{r_{1}p_{2}^{\prime}}{\sqrt{1+p_{2}^{\prime2}}}\right)  ^{\prime
}=O\left(  \frac{\varepsilon^{2}}{1+\varepsilon^{2}r^{2}}\right)  .
\]
Introduce the scaling of $p_{2}$:%
\[
\bar{p}_{2}\left(  r_{1}\right)  =\varepsilon p_{2}\left(  \varepsilon
^{-1}r_{1}\right)  .
\]
We find that in $\left(  l,+\infty\right)  ,$ $\bar{p}_{2}$ satisfies has the
equation%
\begin{equation}
\left(  \frac{r_{1}\bar{p}_{2}^{\prime}}{\sqrt{1+\bar{p}_{2}^{\prime2}}%
}\right)  ^{\prime}=O\left(  \frac{\varepsilon^{2}}{1+r^{2}}\right)  .
\label{p2}%
\end{equation}
From this equation, we deduce
\[
\lim_{r_{1}\rightarrow+\infty}r_{1}\bar{p}_{2}^{\prime}\left(  r_{1}\right)
-l\bar{p}_{2}^{\prime}\left(  l\right)  =O\left(  \varepsilon^{2}\right)  .
\]
On the other hand, by Lemma \ref{p1},
\[
l\bar{p}_{2}^{\prime}\left(  l\right)  -1=O\left(  \varepsilon^{2}\right)  .
\]
Hence
\[
\lim_{r_{1}\rightarrow+\infty}r_{1}\bar{p}_{2}^{\prime}\left(  r_{1}\right)
=1+O\left(  \varepsilon^{2}\right)  .
\]
Consequently, the growth rate of $u$ is equal to
\[
\lim r_{1}p_{2}^{\prime}\left(  r_{1}\right)  =\frac{1}{\varepsilon}+O\left(
\varepsilon\right)  :=\frac{1}{\varepsilon^{\prime}}.
\]
Obviously, $\varepsilon^{\prime}=\varepsilon+O\left(  \varepsilon^{3}\right)
.$

Now let $r=\varepsilon^{\prime-1}\cosh\left(  \varepsilon^{\prime}z\right)  $
be the catenoid with the same growth rate as $u.$ Then there is a solution
$u_{\varepsilon^{\prime}}$ with the slope $\varepsilon^{\prime-1}$ whose nodal
line is close to this catenoid(the nodal line in $\mathbb{E}^{+}$ decaying
algebraically to a vertical small translation of the catenoid. Actually, the
error could be estimated by some positive power of $\varepsilon.$ Our aim is
to show that $u=u_{\varepsilon^{\prime}}.$ To see this, we need to estimate
the function $p_{1}\left(  z\right)  -\varepsilon^{\prime-1}\cosh\left(
\varepsilon^{\prime}z\right)  $ and $p_{2}\left(  r_{1}\right)  -\varepsilon
^{\prime-1}\cosh^{-1}\left(  \varepsilon^{\prime}r_1\right)  .$

Note that
\begin{align*}
\varepsilon\varepsilon^{\prime-1}\cosh\left(  \varepsilon^{\prime}%
\varepsilon^{-1}z\right)  -\cosh z  &  =\left(  1+O\left(  \varepsilon
^{2}\right)  \right)  \cosh\left(  1+O\left(  \varepsilon^{2}\right)
z\right)  -\cosh z\\
&  =O\left(  \varepsilon^{2}\right)  .
\end{align*}
Hence using Lemma \ref{p1}, we obtain $\bar{p}_{1}\left(  z\right)
-\varepsilon\varepsilon^{\prime-1}\cosh\left(  \varepsilon^{\prime}%
\varepsilon^{-1}z\right)  =O\left(  \varepsilon^{2}\right)  .$ This implies
that
\[
p_{1}\left(  z\right)  -\varepsilon^{\prime-1}\cosh\left(  \varepsilon
^{\prime}z\right)  =O\left(  \varepsilon\right)  .
\]
Using $\left(  \ref{p2}\right)  ,$ we then find that $p_{2}\left(
r_{1}\right)  -\varepsilon^{\prime-1}\cosh^{-1}\left(  \varepsilon^{\prime
}r_{1}\right)  =O\left(  \varepsilon\right)  .$

Therefore the nodal line of $u$ is close to the nodal line of $u_{\varepsilon
^{\prime}}$ at the order $O\left(  \varepsilon\right)  .$ Then,  applying the
mapping property of the Jacobi operator of the catenoid,   a contraction
mapping argument shows that $u=u_{\varepsilon^{\prime}}.$ We refer to
\cite{MR3148064} for similar arguments in the case of four-end solutions in
$\mathbb{R}^{2}.$
\end{proof}

\section{Concluding the Proof of Theorem \ref{main}}

In this section, combining the results of the previous sections, we would like
to finish the proof of Theorem \ref{main}.

Let $M$ be the set of all two-end solutions to the Allen-Cahn equation with
growth rate larger than $\sqrt{2}.$ Consider a catenoid type solution
$u_{\varepsilon}$ arising from a largely dilated catenoid. As we mentioned in
Section \ref{Moduli}, $u_{\varepsilon}$ is nondegenerate$.$By Proposition
\ref{th}, locally around $u_{\varepsilon}$, the set of two-end solutions is a
one dimensional real analytic manifold, which we denote as the image of a map
$\varrho$%
\[
\varrho:\left(  -\delta,\delta\right)  \rightarrow M,
\]
for some small $\delta>0.$ Using the compactness result in Section
\ref{compactness} and by the structure theorem $\left(  \ref{structure}%
\right)  $, $\varrho$ has a global continuation:%
\[
\varrho:\left(  -\delta,+\infty\right)  \rightarrow M.
\]
We claim that the growth rate of the solution $\varrho\left(  t\right)  $
tends to $\sqrt{2}$ as $t\rightarrow+\infty.$ Indeed, as $t\rightarrow
+\infty,$ $\mathcal{P}_{\varrho\left(  t\right)  }$ could not remain bounded.
(Recall that $\mathcal{P}_{\varrho\left(  t\right)  }$ is the intersection of
the nodal set of $\varrho\left(  t\right)  $ will the $r$ or $z$ axis.)
Otherwise by the compactness, the image of $\varrho$ will be a closed loop,
which could not be true since the family of catenoid type solutions
$u_{\varepsilon}$ are not compact. Also, as $t\rightarrow+\infty,$
$\mathcal{P}_{\varrho\left(  t\right)  }$ could not be on the $r$ axis, this
follows from the uniqueness of catenoid type solutions, Proposition
\ref{uniq}. Therefore, $\mathcal{P}_{\varrho\left(  t\right)  }$ will be on
the $z$ axis and $\left\vert \mathcal{P}_{\varrho\left(  t\right)
}\right\vert \rightarrow+\infty,$ as $t\rightarrow+\infty.$ By the analysis of
Toda type solutions, Proposition \ref{To}, the growth rate of the solution
$\varrho\left(  t\right)  $ will go to $\sqrt{2}$ as $t\rightarrow+\infty.$
This finishes the proof.

\section{Appendix}

\subsection{Monotonicity of two-end solutions}

In this appendix, we sketch the proof of monotonicity of two-end
solutions(Proposition \ref{monoto} in Section \ref{compactness}) using the
moving plane argument. Essentially, we follow the proof of monotonicity for
four-end solutions of 2D Allen-Cahn equation in \cite{MR2911416}. But their is
a slight difference here. Namely for the two-end solutions in dimension three,
to start the moving procedure at infinity, one need to have suitable control
of the asymptotic behavior of the solution(estimate $\left(  \ref{r2}\right)
$ below), while this is not needed in dimension two case. The asymptotic
expansion we need is provided by the results of Section \ref{compactness}.

\begin{proof}
[Proof of Proposition \ref{monoto}]Let $u$ be a two-end solution. We first
prove its monotonicity in the $r$ direction. To use the moving plane
machinery, we will work in the usual Euclidean coordinate $\left(
x,y,z\right)  .$ Set $U\left(  x,y,z\right)  :=u\left(  \sqrt{x^{2}+y^{2}%
},z\right)  =u\left(  r,z\right)  .$

Suppose the growth rate of $u$ is equal to $k>\sqrt{2}.$ Hence the asymptotic
curve of its nodal line in the first quadrant is
\[
z=k\ln r+b
\]
for some $b\in\mathbb{R}.$

Let $x_{0}\geq0$ be a parameter and define
\[
\bar{U}\left(  x,y,z\right)  =U\left(  x,y,z;x_{0}\right)  :=U\left(
2x_{0}-x,y,z\right)  .
\]
Certainly $\bar{U}\left(  x_{0},y,z\right)  =U\left(  x_{0},y,z\right)  .$
Note that $\bar{U}$ actually depends on the parameter $x_{0}$.

The first step is to show that the moving plane procedure could be started at
$+\infty.$ We claim that for $x_{0}$ large enough, say $x_{0}>a_{0},$
\[
\bar{U}\left(  x,y,z;x_{0}\right)  <U\left(  x,y,z\right)  ,\text{ for
}x<x_{0}.
\]
First of all, we consider the region where $|\bar{U}|$ is not close to $1.$
Recall that the nodal set of $\bar{U}$ is close to
\[
z=k\ln\sqrt{\left(  2x_{0}-x\right)  ^{2}+y^{2}}+b.
\]
We have
\begin{align}
&  \ln\sqrt{\left(  2x_{0}-x\right)  ^{2}+y^{2}}-\ln\sqrt{x^{2}+y^{2}%
}\nonumber\\
&  =\frac{1}{2}\ln\left(  1+\frac{4x_{0}^{2}-4x_{0}x}{r^{2}}\right)  .
\label{err}%
\end{align}
Note that for $x<x_{0},$ $\frac{4x_{0}^{2}-4x_{0}x}{r^{2}}>0.$ Fix a small
positive constant $\epsilon.$

If $\frac{4x_{0}^{2}-4x_{0}x}{r^{2}}>\epsilon,$ then by $\left(
\ref{err}\right)  ,$
\begin{equation}
\ln\sqrt{\left(  2x_{0}-x\right)  ^{2}+y^{2}}-\ln\sqrt{x^{2}+y^{2}}>\frac
{1}{2}\ln\left(  1+\epsilon\right)  . \label{dif}%
\end{equation}
By the results of Section \ref{compactness}, for $r$ large, we have
\begin{equation}
u\left(  r,z\right)  -H\left(  z_{1}\right)  =O\left(  r^{-2}\right)  .
\label{r2}%
\end{equation}
Here $z_{1}$ is the signed distance of $\left(  r,z\right)  $ to the nodal
line. As a consequence
\begin{align*}
&  \bar{U}\left(  x,y,z;x_{0}\right)  -U\left(  x,y,z\right) \\
&  =u\left(  \sqrt{\left(  2x_{0}-x\right)  ^{2}+y^{2}},z\right)  -u\left(
r,z\right) \\
&  =H\left(  \left[  z-k\ln\sqrt{\left(  2x_{0}-x\right)  ^{2}+y^{2}%
}-b\right]  \cos\theta_{1}\right) \\
&  -H\left(  \left[  z-k\ln\sqrt{x^{2}+y^{2}}-b\right]  \cos\theta_{2}\right)
+O\left(  r^{-2}\right)  ,
\end{align*}
where $\theta_{i}=O\left(  r^{-1}\right)  .$ It follows from $\left(
\ref{dif}\right)  $ that
\begin{align*}
\bar{U}\left(  x,y,z;x_{0}\right)  -U\left(  x,y,z\right)   &  \leq H\left(
z-k\ln\sqrt{\left(  2x_{0}-x\right)  ^{2}+y^{2}}-b\right) \\
&  -H\left(  z-k\ln\sqrt{x^{2}+y^{2}}-b\right)  +O\left(  r^{-2}\right) \\
&  <0,
\end{align*}
for $r$ large enough.

If $\frac{4x_{0}^{2}-4x_{0}x}{r^{2}}\in\left(  0,\varepsilon\right)  ,$ then
\begin{equation}
\ln\sqrt{\left(  2x_{0}-x\right)  ^{2}+y^{2}}-\ln\sqrt{x^{2}+y^{2}}%
=\frac{2x_{0}^{2}-2x_{0}x}{r^{2}}+O\left(  \frac{x_{0}^{2}\left(  x_{0}%
-x_{0}\right)  ^{2}}{r^{4}}\right)  . \label{low}%
\end{equation}
There are two possible cases. Case 1: $x\in\left(  -\infty,x_{0}-1\right)  .$
In this case, by $\left(  \ref{low}\right)  ,$
\[
\ln\sqrt{\left(  2x_{0}-x\right)  ^{2}+y^{2}}-\ln\sqrt{x^{2}+y^{2}}\geq
\frac{x_{0}}{r^{2}}.
\]
This together with $\left(  \ref{r2}\right)  $ implies that for $x_{0}$ large,
$\bar{U}<U.$ Case 2: $x\in\left(  x_{0}-1,x_{0}\right)  .$ In this case, by
the estimate
\[
\partial_{r}\left(  u\left(  r,z\right)  -H\left(  z_{1}\right)  \right)
=O\left(  r^{-2}\right)  ,
\]
we find that(here one should also use some estimates of the Fermi coordinate)
\[
\bar{U}-U\leq-C\frac{x_{0}^{2}-x_{0}x}{r^{2}}+O\left(  r^{-2}\left(
x_{0}-x\right)  \right)
\]
for certain positive constant $C.$ Hence if we choose $x_{0}$ large, $\bar
{U}<U.$ Therefore, to prove the claim, it remains to consider the region where
$\left\vert \bar{U}\right\vert \sim1.$ Observe that for $x_{0}$ large enough,
in the region where $\bar{U}\sim1,$ $U$ is also close to $1.$ Now let
$\varphi:=\bar{U}-U.$ Then in the region $x<x_{0},$ $\varphi$ satisfies
\[
-\Delta_{\left(  x,y,z\right)  }\varphi+\left(  \bar{U}^{2}+\bar{U}%
U+U^{2}-1\right)  \varphi=0.
\]
Since $\lim\sup_{\left\vert \left(  x,y,z\right)  \right\vert \rightarrow
+\infty}\varphi\leq0,$ by the maximum principle, $\varphi\left(  x,y,z\right)
<0,$ for $x<x_{0}.$ This proves the claim.

In the second step, we define
\[
x^{\ast}=\inf\left\{  \bar{x}:\bar{U}\left(  x,y,z;x_{0}\right)  <U\left(
x,y,z\right)  \text{ for }x<x_{0}\text{ and }x_{0}\in\left(  \bar{x}%
,a_{0}\right)  \right\}  .
\]
We show that $x^{\ast}=0.$ To see this, we first prove that $r^{\ast}<a_{0}.$
Indeed, by the first step,
\begin{align*}
\bar{U}\left(  x,y,z;a_{0}\right)   &  <U\left(  x,y,z\right)  \text{ for
}x<a_{0},\\
\bar{U}\left(  a_{0},y,z;a_{0}\right)   &  =U\left(  a_{0},y,z\right)  .
\end{align*}
Hence by the Hopf Lemma, $\partial_{x}\left(  \bar{U}\left(  \cdot
;a_{0}\right)  -U\right)  >0$ for $x=a_{0}.$ Then standard arguments together
with the asymptotic behavior of $U$ implies that $\bar{U}\left(
x,y,z;x_{0}\right)  <U,$ for $x<x_{0}$ and $x_{0}$ sufficiently close to
$a_{0}.$ Then one could use this type of arguments to show that $x^{\ast}%
=0$. (The plane could be moved to the left until the inequality $\bar{U}<U$ is
violated at infinity.)

The monotonicity in $z$ direction could be proved similarly using moving
plane. This completes the proof.
\end{proof}

$\bigskip$

$\bigskip$

$\mathbf{Acknowledgement}$ Y. Liu is partially supported by NSFC grant
11101141 and the Fundamental Research Funds for the Central Universities
13MS39. J. Wei is partially supported by NSERC of Canada.


\bigskip

\begin{thebibliography}{99}                                                                                               %


\bibitem {ACa}L. Ambrosio, X. Cabre. \textit{Entire solutions of semilinear
elliptic equations in }$\mathbb{R}^{3}$\textit{ and a conjecture of De
Giorgi,} J. Amer. Math. Soc. 13 (4)(2000) 725--739.

\bibitem {MR3281950}Oscar Agudelo, Manuel del Pino, and Juncheng Wei.
\textit{Solutions with multiple catenoidal ends to the Allen--Cahn equation in
}$\mathbb{R}^{3}$. J. Math. Pures Appl. (9), 103(1):142--218, 2015.

\bibitem {MR3057178}Francesca Alessio and Piero Montecchiari. \textit{Saddle
solutions for bistable symmetric semilinear elliptic equations}. NoDEA
Nonlinear Differential Equations Appl., 20(3):1317--1346, 2013.

\bibitem {MR1775735}Luigi Ambrosio and Xavier Cabre. \textit{Entire solutions
of semilinear elliptic equations in }$\mathbb{R}^{3}$ \textit{and a conjecture
of De Giorgi.} J. Amer. Math. Soc., 13(4):725--739 (electronic), 2000.

\bibitem {MR1956130}Boris Buffoni and John Toland. \textit{Analytic theory of
global bifurcation}. Princeton Series in Applied Mathematics. Princeton
University Press, Princeton, NJ, 2003.

\bibitem {MR0375019}E. N. Dancer. \textit{Global structure of the solutions of
non-linear real analytic eigenvalue problems}. Proc. London Math. Soc. (3),
27:747--765, 1973.

\bibitem {MR1962054}E. N. Dancer. \textit{Real analyticity and non-degeneracy}%
. Math. Ann., 325(2):369--392, 2003.

\bibitem {DFP}H. Dang, P. C. Fife and L. A. Peletier. \textit{Saddle solutions
of the bistable diffusion equation}. Z. Angew. Math. Phys. 43(6):984-998, 1993.

\bibitem {dkp-2009}Manuel del Pino, Michal Kowalczyk, and Frank Pacard.
\textit{Moduli space theory for the Allen-Cahn equation in the plane}. Trans.
Amer. Math. Soc., 365(2):721--766, 2013.

\bibitem {MR2557944}Manuel del Pino, Michal Kowalczyk, Frank Pacard, and
Juncheng Wei. \textit{Multiple-end solutions to the Allen-Cahn equation in
}$\mathbb{R}^{2}$. J. Funct. Anal., 258(2):458--503, 2010.

\bibitem {MM}M. del Pino, M. Kowalczyk, J. Wei. \textit{On De Giorgi's
conjecture in dimension }$N\geq9$\textit{,} Ann. of Math. (2) 174 (3) (2011) 1485--1569.

\bibitem {MR3019512}Manuel del Pino, Michal Kowalczyk, and Juncheng Wei.
\textit{Entire solutions of the Allen-Cahn equation and complete embedded
minimal surfaces of finite total curvature in }$\mathbb{R}^{3}$. J.
Differential Geom., 93(1):67--131, 2013.

\bibitem {GG}N. Ghoussoub, C. Gui, \textit{On a conjecture of De Giorgi and
some related problems,} Math. Ann. 311 (3) (1998) 481--491.

\bibitem {MR2911416}Changfeng Gui. \textit{Symmetry of some entire solutions
to the Allen-Cahn equation in two dimensions. }J. Differential Equations,
252(11):5853--5874, 2012.

\bibitem {GLW}Changfeng Gui, Yong Liu, and Juncheng Wei. \textit{Mountain pass
characterization of four-end solutions of Allen-Cahn in $\mathbb{R}^{2}$},
preprint 2015.

\bibitem {MR2906930}Nikolaos Kapouleas. \textit{Doubling and desingularization
constructions for minimal surfaces. }In Surveys in geometric analysis and
relativity, volume 20 of Adv. Lect. Math. (ALM), pages 281--325. Int. Press,
Somerville, MA, 2011.

\bibitem {ML}Michal Kowalczyk, Yong Liu. \textit{Nondegeneracy of the saddle
solution of the Allen--Cahn equation on the plane,} Proc. Amer. Math. Soc. 139
(12) (2011) 4319--4329.

\bibitem {MR2971030}Michal Kowalczyk, Yong Liu, Frank Pacard. \textit{The
space of 4-ended solutions to the Allen-Cahn equation in the plane}. Ann.
Inst. H. Poincare Anal. Non Lineaire 29 (2012), no. 5, 761--781.

\bibitem {MR3148064}Michal Kowalczyk, Yong Liu, and Frank Pacard. \textit{The
classification of four-end solutions to the Allen-Cahn equation on the plane}.
Anal. PDE, 6(7):1675--1718, 2013.

\bibitem {KLW}Michal Kowalczyk, Yong Liu, Frank Pacard, Juncheng Wei.
\textit{End-to-end construction for the Allen--Cahn equation in the plane}.
Calc. Var. Partial Differential Equations 52 (2015), no. 1-2, 281--302.

\bibitem {MR1371233}R. Kusner, R. Mazzeo, and D. Pollack. \textit{The moduli
space of complete embedded constant mean curvature surfaces}. Geom. Funct.
Anal., 6(1):120--137, 1996.

\bibitem {MR1837428}Rafe Mazzeo, Frank Pacard, and Daniel Pollack.
\textit{Connected sums of constant mean curvature surfaces in Euclidean
space}. J. Reine Angew. Math., 536:115--165, 2001.

\bibitem {MR1356375}Rafe Mazzeo, Daniel Pollack, and Karen Uhlenbeck.
\textit{Moduli spaces of singular Yamabe metrics}. J. Amer. Math. Soc.,
9(2):303--344, 1996.

\bibitem {PW}F. Pacard and J. Wei, \textit{Stable Solutions of Allen-Cahn
Equation over Simon's cone}, J. Funct. Anal.264(5): 1131-1167, 2013.

\bibitem {MR1487630}Joaquin Perez and Antonio Ros. \textit{The space of
complete minimal surfaces with finite total curvature as Lagrangian
submanifold}. Trans. Amer. Math. Soc., 351(10):3935--3952, 1999.

\bibitem {MR1901613}Joaquin Perez and Antonio Ros. \textit{Properly embedded
minimal surfaces with finite total curvature}. In The global theory of minimal
surfaces in flat spaces (Martina Franca, 1999), volume 1775 of Lecture Notes
in Math., pages 15--66. Springer, Berlin, 2002.

\bibitem {Savin}O. Savin. \textit{Regularity of flat level sets in phase
transitions}, Ann. of Math. (2) 169 (1) (2009) 41--78.

\bibitem {MR730928}Richard M. Schoen. \textit{Uniqueness, symmetry, and
embeddedness of minimal surfaces}. J. Differential Geom., 18(4):791--809
(1984), 1983.

\bibitem {MR1924593}Martin Traizet. \textit{An embedded minimal surface with
no symmetries}. J. Differential Geom., 60(1):103--153, 2002.
\end{thebibliography}
\end{document}